\documentclass[12pt,a4paper,reqno]{amsart}
\usepackage{amsmath}
\usepackage{amsthm}
\usepackage{amsfonts,amssymb,accents}
\usepackage{mathtools}
\usepackage{mathrsfs}
\usepackage{comment}

\allowdisplaybreaks
\usepackage{setspace}
\usepackage{amscd}
\usepackage[latin2]{inputenc}
\usepackage{enumerate}
\usepackage[mathscr]{eucal}
\usepackage{indentfirst}
\usepackage{enumitem}
\usepackage[margin=2.3cm]{geometry}

\usepackage[bookmarks=false]{hyperref}
\hypersetup{
	colorlinks=true,
	linkcolor=blue,
	filecolor=blue,
	citecolor = blue,      
	urlcolor=cyan,
}

\theoremstyle{plain}
\newtheorem{Th}{Theorem}[section]
\newtheorem{Lemma}[Th]{Lemma}
\newtheorem{Claim}[Th]{Claim}
\newtheorem{Cor}[Th]{Corollary}

\newtheorem{Theorem}[Th]{Theorem}

 \theoremstyle{definition}

\newtheorem{Rem}[Th]{Remark}
\newtheorem{?}[Th]{Problem}

\newtheorem{Result}[Th]{Result}

\theoremstyle{remark}

\newtheorem*{Notation}{Notation}
\newtheorem*{Notations}{Notations}
\newtheorem*{Assumptions}{Assumptions}
\newtheorem*{Consequences}{Consequences}

 \theoremstyle{plain}
\newtheorem{thm}{Theorem}[section]

\newtheoremstyle{TheoremNum}
{\topsep}{\topsep}              
{\itshape}                      
{}                              
{\bfseries}                     
{.}                             
{ }                             
{\thmname{#1}\thmnote{ \bfseries #3}}
\theoremstyle{TheoremNum}
\newtheorem{thmn}{Theorem}

\newcommand{\ra}{\rightarrow}

\newcommand{\p}{\mathbb{P}}
\newcommand{\e}{\mathbb{E}}

\newcommand{\h}{\mathcal{H}}

\newcommand{\re}{{\mathbb{R}}}

\newcommand{\ga}{\gamma}

\newcommand{\vol}{\text{Vol}}

\newcommand{\R}{R_n/n}
\newcommand{\yj}{y_j}

\newcommand{\nat}{\mathbb{N}}

\newcommand{\ffT}{B(T)}

\newcommand{\dis}{\mathbb{D}}

\newcommand{\uns}{unstable$^*$~}

\newcommand{\ee}{E^{*}}
\newcommand{\col}{\mathscr{C}^{*}}
\newcommand{\mi}{\mathcal{I}}
\newcommand{\sone}{\mathbb{S}^1}

\newcommand{\ep}{\epsilon}
\newcommand{\sph}{\mathbb{S}^2}

\newcommand{\fhat}{\hat{f}(x,y)}
\newcommand{\intc}{\int_{0}^{2\pi}}
\newcommand{\intr}{\int_{\mathbb{R}^2}}
\newcommand{\lamr}{\Lambda_R}
\newcommand{\tfm}{\tilde{f_{\mu}}}
\newcommand{\intl}{\int_{\Lambda_R}}

\newcommand{\nsq}{(-n,n)^2}
\newcommand{\psig}{\psi_{n}(R)}
\newcommand{\psigl}{\psi_{n}(L)}
\newcommand{\bgr}{B_{g_n}(R)}
\newcommand{\bgrl}{B_{g_n}(L)}
\newcommand{\pmn}{P_{n}^{m}}
\newcommand{\xb}{\overline{x}_1}
\newcommand{\gn}{g_n}
\newcommand{\eg}{g}
\newcommand{\mmin}{m_{-}}
\newcommand{\mpl}{m_{+}}
\newcommand{\deltas}{\delta_{*}}
\newcommand{\mc}{\mathcal{M}}
\newcommand{\E}{\mathscr{E}}

\newcommand{\U}{\mathcal{U}}

\newcommand{\tor}{\mathbb{T}^2}
\newcommand{\scf}{\mathscr{F}}
\newcommand{\ex}{\textup{exp}}
\newcommand{\bup}{\mathscr{B}_n}
\newcommand{\bupin}{\mathscr{B}_{n}^{-1}}
\newcommand{\natz}{\mathbb{N}_{0}}
\newcommand{\disc}{\mathcal{D}}
\newcommand{\lb}{\left(}
\newcommand{\rb}{\right)}
\newcommand{\xnot}{x_0}
\newcommand{\dnot}{\delta_0}
\newcommand{\nono}{\nonumber}
\newcommand{\zf}{\mathcal{Z}(f)}
\newcommand{\trees}{\mathcal{T}}
\newcommand{\tom}{\widetilde{\Omega}}
\newcommand{\torus}{\mathbb{T}^2}
\newcommand{\integers}{\mathbb{Z}}
\newcommand{\set}{\mathcal{S}}

\newcommand{\arnj}{\mathcal{F}_{n_j}}
\newcommand{\scrt}{\mathscr{T}}
\newcommand{\noin}{\noindent}
\newcommand{\rhom}{\nu}

\begin{document}
	\onehalfspacing
\title[Concentration for nodal component count]{Concentration for nodal component count of Gaussian Laplace eigenfunctions}

\author{Lakshmi Priya}
\address{Department of Mathematics, Indian Institute of Science, Bangalore 560012, India}
\email{lakshmip@iisc.ac.in}

\thanks{This work is supported  by CSIR-SPM fellowship (File No. SPM-07/079(0260)/2017-EMR-I)), CSIR, Government of India and by a UGC CAS-II grant (Grant No. F.510/25/CAS-II/2018(SAP-I))}


\begin{abstract}
	{We study nodal component count of the following Gaussian Laplace eigenfunctions: monochromatic random waves (MRW)  on $\mathbb{R}^2$, 
		arithmetic random waves (ARW) on $\mathbb{T}^2$ and random spherical harmonics (RSH) on $\mathbb{S}^2$.  Exponential concentration for nodal component count of RSH on $\mathbb{S}^2$ and ARW on $\mathbb{T}^2$ were established in \cite{NS} and \cite{Rozenshein} respectively. We prove exponential concentration for nodal component count in the following three cases: MRW on growing Euclidean balls in $\mathbb{R}^2$; RSH and ARW on geodesic balls, in $\mathbb{S}^2$ and $\mathbb{T}^2$ respectively, whose radius is slightly larger than the wavelength scale.}
	
\end{abstract}
\maketitle

\section{Introduction}
\subsection{Laplace eigenfunctions and nodal sets} For a smooth Riemannian  manifold $(M,g)$, let $\Delta_g$ denote the Laplace-Beltrami operator. For a smooth function $f: M \ra \re$, the \textit{nodal set} of $f$ is its zero set $ f^{-1}\{0\}$.  We call a connected component of $f^{-1}\{0\}$ and $M \setminus f^{-1}\{0\}$  a \textit{nodal component} and a \textit{nodal domain} of $f$ respectively.
A non-constant function $f$ is called an \textit{eigenfunction of the Laplacian} if there is a $\lambda \in \re$ such that  $\Delta_g f + \lambda f = 0$ holds on $M$. If $M$ is compact (and hence for $\sph$ and $\torus$), it is known that the eigenvalues of $-\Delta_g$  are all positive and can be enumerated as $\{\lambda_n (M)\}_{n \in \nat}$ with $\lambda_n(M) \nearrow \infty.$
\subsection{Motivation}

In this paper, we establish concentration results for the nodal component count of a huge class of Gaussian Laplace eigenfunctions on the plane  $\re^2$, the sphere $\sph$ and the torus $\torus$. 
Just like harmonic functions, eigenfunctions of the Laplacian  are also very regular and enjoy a host of nice properties (Section \ref{reglap}) which in turn also reflect on their nodal sets and nodal domains. There are several classical results which study the regularity of the nodal sets/domains of the eigenfunctions and among these are Courant's nodal domain theorem and Yau's conjecture which give bounds for the nodal domain count and the volume of the nodal set respectively;  it is interesting to study these quantities for random Laplace eigenfunctions and \textit{Gaussian Laplace eigenfunctions} (Sections \ref{RSH}, \ref{SGP} and \ref{secarw}) specifically. 
There are several other reasons for interest in this study. 
\begin{itemize}[ align=left,leftmargin=*,widest={9}]
	\item The \textit{random plane wave}, one of the many random functions we study on $\re^2$, is special.  Berry conjectured in \cite{Berry} (Berry's random wave conjecture) that the random plane wave is a universal object which models Laplace eigenfunctions  corresponding to high eigenvalues on any manifold whose geodesic flow is ergodic. 
	\item For Gaussian Laplace eigenfunctions corresponding to eigenvalue $\lambda_n$ on $\sph$ and $\torus$, we study their nodal component count on  geodesic balls whose radius $r_{n}$ is slightly larger than the Plank scale, that is $r_{n} = R_{n}/\sqrt{\lambda_n}$, where $\lim_{n \ra \infty} R_{n} = \infty$. The interest in studying Laplace eigenfunctions  on such geodesic balls comes from  the \textit{semiclassical eigenfunction hypothesis} of Berry \cite{Berry,Berry2}. 
Study of quantities including  the $L^2$ mass,  volume of the nodal set and the nodal component count on such geodesic balls for Gaussian Laplace eigenfunctions have  been carried out in  \cite{MDC, granwig, tacyhan, lesrud, sartori2019planckscale}. 
	\end{itemize}

\subsection{Prior work} 
The works of Nazarov--Sodin in \cite{NS2} and \cite{NS} were  major developments in the study of nodal component count of {random functions}.
Their method of proof and the techniques developed in their works have been extensively used by several authors to study interesting questions about  nodal sets of {random functions}. In this section, we introduce the Gaussian Laplace eigenfunctions considered in our work and  present relevant results known about their nodal component count upon which we build our results. 

\subsubsection{Random spherical harmonics} \label{RSH}

 On the two dimensional sphere $\mathbb{S}^2$, the eigenvalues of the Laplacian are  $n(n+1)$, $n\in \nat$. The eigenspace $\mathscr{V}_n$ corresponding to  eigenvalue $n(n+1)$ is the space of degree $n$ spherical harmonics and $\text{dim}(\mathscr{V}_n) = (2n+1)$. Let $\{f_m: -n \leq m \leq n\}$ be an orthonormal basis for  $\mathscr{V}_n$ with respect to the $L^2(\mathbb{S}^2)$ norm.  The degree $n$ random spherical harmonic  $\mathscr{F}_n $ is defined as follows
 \begin{align}\label{rsh}
 \mathscr{F}_n := \frac{1}{\sqrt{2n+1}} \sum_{m=-n}^{n}\xi_{m} f_{m} ,~\mbox{where $\xi_m \overset{\textup{i.i.d.}}{\sim} \mathcal{N}(0,1)$}.
 \end{align}

 The nodal component count for the ensemble of random spherical harmonics $\{\mathscr{F}_n\}_{n\in \nat}$ was studied in \cite{NS}. Let $N(\cdot)$ denote the nodal component count in $\sph$; convergence (almost surely and in $L^1$) of $N(\mathscr{F}_n)/4 \pi n^2$ to a non-zero constant $c_{NS}$ and exponential concentration around this constant was established in \cite{NS}.  
\begin{Result}\label{res1}(\cite{NS}, Theorem 1.1) There is a constant $c_{NS}>0$ such for every $ \ep > 0$, there are constants $c_{\ep}, C_{\ep}>0$ satisfying the following
	\begin{align*}
	\p\left(\left| \frac{N(\mathscr{F}_n)}{4 \pi n^2} - c_{NS} \right| > \ep \right) \leq C_{\ep}e^{-c_{\ep}n}. 
	\end{align*}
	
\end{Result}

\subsubsection{Stationary Gaussian processes}\label{SGP}
 Count of nodal components for stationary Gaussian processes on $\re^d$ ($d \geq 2$) was studied  in \cite{NS2}.  In later works \cite{BW, CanSar, SW}, the more specialized question of  counting nodal domains/components which satisfy certain properties specified in terms of their topological type, volume, boundary volume etc. was considered. The following is a restricted version  of the main theorem in \cite{NS2}. 

\begin{Result} \label{res2}(\cite{NS2}, Theorem 1) Let $F_{\nu}$ be a centered, stationary Gaussian process on $\re^2$ whose spectral measure is $\nu$. Assume further that $\nu$ has no atoms and is supported on the unit circle $\mathbb{S}^1$. Then there is a positive constant $c_{NS}(\nu)>0$ such that the following convergence happens almost surely and in $L^1$
	\begin{align}
	\frac{N_R(F_{\nu})}{\pi R^2} \ra c_{NS}(\nu)~\mbox{as $R \ra \infty$}, \label{lln}
	\end{align}
	where $N_R(\cdot)$ is the nodal component count in $B(R)$, the Euclidean ball of radius $R$ centered at the origin.
\end{Result}

In order to state one of the main results of \cite{SW}, we introduce the following notations. Let $\mathcal{T}$ be the collection of finite, rooted trees. Let $f: \re^2 \ra \re$ be a smooth function such that its zero set $\zf$ is a collection of disjoint smooth curves. Let $\gamma$ be a bounded component of $\zf$ (hence $\ga$ is necessarily a simple closed curve in $\re^2$).  To every such $\ga$, we assosciate an element  $e_{\ga} \in \mathcal{T}$ called its \textit{tree end} and this is done as follows. The vertices of $e_{\ga}$ are in one-one correpondence with the nodal domains of $f$ which lie in the interior of $\ga$. There is an edge between the vertices corresponding to two nodal domains if and only if they share a boundary. Note that $e_{\ga}$ captures the nesting configuration of the nodal components in the interior of $\ga$. 

For $R>0$ and $\scrt \in \trees$, let $N_{R}(f,\scrt)$ denote the count of nodal components of $f$  contained in $B(R)$ whose tree end is $\scrt$. Define a probability measure $\mu_{R}(f)$ on $\trees$ by
\begin{align*}
\mu_{R}(f) := \frac{1}{N_{R}(f)} \sum_{\scrt \in \trees}  N_{R}(f,\scrt) \cdot \delta_{\scrt}.
\end{align*}
If $F_{\nu}$ is as in Result \ref{res2}, then it follows by Bulinskaya's lemma (\cite{NS2}, Lemma 6) that almost surely $F_{\nu}$ does not have any singular zeros. Hence almost surely, the zero set of $F_{\nu}$ is a collection of disjoint smooth curves and so $N_R(F_{\nu},\scrt)$ makes sense.
In \cite{SW}, the authors study the asymptotic distribution of tree ends for stationary Gaussian processes on $\re^2$. The following is a restricted version of their results. 
\begin{Result}[\cite{SW}, Theorems 3.3, 4.2, 5.1, Proposition 5.3] \label{sw}With $F_{\nu}$ as in Result \ref{res2}, there exists a probability measure $\mu$ on $\trees$ (depending on $\nu$) whose support is the whole of $\trees$ and is such that the following holds for every $\scrt \in \trees$,
	\begin{align*}
	\e\left[\left|\frac{N_{R}(F_{\nu},\scrt)}{\pi R^2} - c_{NS}(\nu)\mu(\scrt)\right|\right] \ra 0~\text{as $R \ra \infty$.}	
	\end{align*}
\end{Result}
\begin{Rem}By a \textit{plane wave}, we mean a function $f:\re^2 \ra \re$ satisfying $\Delta f+ f =0$ on $\re^2$. In the setting of Results \ref{res2} and \ref{sw}, the assumption that $\nu$ is supported on $\mathbb{S}^1$ implies that almost surely $F_{\nu}$ is a plane wave.  \textit{Random plane wave} is the Gaussian process $F_{\nu_0}$, where $\nu_0$ is the uniform measure on $\mathbb{S}^1$. 
		\end{Rem}

\subsubsection{Arithmetic random waves}\label{secarw} Consider the standard 2-torus $\mathbb{T}^2 := \re^2/\mathbb{Z}^2$ equipped with the metric induced from $\re^2$. 
 Let $\set$ be the set of integers which can be expressed as a sum of two squares,
$\mathcal{S}:= \{a^2 + b^2: a,b \in \integers \}$. For $n \in \set$, define $E_n := 4\pi^2 n$.  The spectrum of the Laplacian on $\torus$ is $\{E_n: n \in \mathcal{S}\}$. For $n \in \set$, define
\begin{align*}
\Lambda_n &:= \{(a,b)\in \integers^2 : a^2 + b^2 = n\}.
\end{align*}
Note that if $\lambda \in \Lambda_n$, then $-\lambda \in \Lambda_n$. Let $\Lambda_{n}^{+} \subset \Lambda_n$ be such that for every $\lambda \in \Lambda_n$, exactly one of $\pm\lambda$ belongs to $\Lambda_{n}^{+}$. 
The eigenspace of $E_n$  is given by
\begin{align}\label{wn}
\mathcal{W}_n = \textup{Span}\{\cos(2\pi \lambda \cdot z), \sin(2\pi \lambda \cdot z): \lambda \in \Lambda_n^{+} \},
\end{align}
where $z \in \torus$ and $\lambda \cdot z$ is the standard inner product in $\re^2$. Hence $\text{dim}(\mathcal{W}_n) = |\Lambda_n|$. Note that the spanning set in \eqref{wn} is orthogonal in $L^2(\torus)$. For each $n \in \set$, define 

\begin{align}\label{arw}
\mathcal{F}_n (z) := \sqrt{\frac{2}{|\Lambda_n|}} \sum_{\lambda \in \Lambda_{n}^{+}} (\xi_{\lambda} \cos(2\pi \lambda \cdot z) + \eta_{\lambda} \sin(2\pi \lambda \cdot z)),
\end{align}
where $\{\xi_{\lambda}, \eta_{\lambda} : \lambda \in \Lambda_{n}^{+}\}$ are i.i.d. $\mathcal{N}(0,1)$ random variables. $\{\mathcal{F}_n\}_{n \in \set}$ is called the ensemble of arithmetic random waves. 
For $n \in \set$, $\nu_n$ defined below is a probability measure on $\mathbb{S}^1$,
\begin{align*}
\nu_n := \frac{1}{|\Lambda_n|} \sum_{\lambda \in \Lambda_{n}} \delta_{\lambda/\sqrt{n}}.
\end{align*}
As before, $N(\cdot)$ denotes the nodal component count in $\torus$. 
\begin{Result}[\cite{Rozenshein}, Theorem 1.3 and \cite{parwigman}, Theorem 1.5 (2)] \label{rozresult} There are constants $c,C >0$ such that for every $\ep >0$, the following holds. If $\{n_j:j \in \nat\} \subseteq \set$ is such that $\nu_{n_j} \Rightarrow \nu$, where $\nu$ is a probability measure on $\mathbb{S}^1$ with no atoms and $c_{NS}(\nu)$ is as in Result \ref{res2}, then  
	\begin{align*}
	\p\left( \left| \frac{N(\mathcal{F}_{n_j})}{n_j} - c_{NS}(\nu)\right| > \ep\right) \leq Ce^{-c\ep^{15}|\Lambda_{n_j}|}. 
	\end{align*}
\end{Result}
\subsection{Main results} 
 We prove concentration results for the count of nodal components in the setting of Result \ref{res2} and for random spherical harmonics and  arithmetic random waves on geodesic balls with radius slightly larger than the wavelength scale.
\subsubsection{Random monochromatic waves} 
In the setting of Results \ref{res2} and \ref{sw}, we prove concentration results for $N_R(F_{\nu})/\pi R^2$ and $N_R(F_{\nu},\mathscr{T})/\pi R^2$.

\begin{Notation}
	Let $\nu$ be a Borel probability measure on $\re^2$ which is supported on $\mathbb{S}^1$ and which has no atoms. Then $\nu$ can be thought of as a Borel probability measure on $\re$ which is supported on $[0,2\pi]$. Since $\nu$ has no atoms, the distribution function $\Phi_{\nu}(t):= \nu(-\infty,t]$ is continuous and  we let $\omega_{\nu}(\cdot)$ denote the modulus of continuity of $\Phi_{\nu}$.
	\end{Notation}

\begin{Theorem}\label{thmrpw}Let $F_{\nu}$ be a centered, stationary Gaussian process on $\re^2$ whose spectral measure  $\nu$ is a probability measure supported on the unit circle $\mathbb{S}^1$ and which has no atoms. Let {$\mathscr{T} \in \trees$}. Let $c_{NS}(\nu)$ and $\mu$ be as in Results \ref{res2} and \ref{sw} respectively. There are constants $a_{\nu},A >0$ such that for every $\ep>0$ and every $R>0$, we have
	\begin{equation} \label{concrpw}
	\begin{aligned}
	\p\left(\left| \frac{N_R (F_{\nu})}{\pi R^2} - c_{NS}(\nu) \right| > \ep \right) \leq A~e^{-a_{\nu}\ep^{16}/\omega_{\nu}(1/R) },\\
		\p\left(\left| \frac{N_R (F_{\nu},\mathscr{T})}{ \pi R^2} - c_{NS}(\nu) \mu(\mathscr{T}) \right| > \ep \right) \leq A~e^{-a_{\nu}\ep^{16}/\omega_{\nu}(1/R) }.
	\end{aligned}
	\end{equation}
\end{Theorem}

\begin{Rem}
	Suppose $\nu$ has a density $\phi_{\nu}$ w.r.t. the uniform measure $\nu_0$ on $\mathbb{S}^1$, then we have the following estimates for $\omega_{\nu}$. For $\phi_{\nu} \in L^{\infty}$, we have  $\omega_{\nu}(\ep) \lesssim \ep$ and if $\phi_{\nu} \in L^{p}$ for some $p >1$, then $\omega_{\nu}(\ep) \lesssim \ep^{(p-1)/p}$.
\end{Rem}
The following result establishes the sharpness of the estimate in Theorem \ref{thmrpw}  for a class of stationary Gaussian processes, including the random plane wave. 
\begin{Theorem} \label{thmlowerbound}
In the setting of Theorem \ref{thmrpw}, assume that $\nu$ is absolutely continuous w.r.t. the uniform measure on $\mathbb{S}^1$ and the corresponding density $\psi$ is such that $\psi >0$ on $\mathbb{S}^1$ and $\sqrt{\psi} \in C^5(\mathbb{S}^1)$. Then for every $\kappa >0$ small enough, there is a constant $C_{\kappa} >0$ such that for every $R >0$, we have 
\begin{align*}
\p \left( \frac{N_{R}(F_{\nu})}{\pi R^2} \leq \kappa \right) \geq e^{-C_{\kappa}R}. 
\end{align*} 
	\end{Theorem}

\subsubsection{Random spherical harmonics on geodesic balls}
Let $\mathscr{F}_n$ be the degree $n$ random spherical harmonics defined in \eqref{rsh} and let $p=(0,0,1) \in \sph$ be the north pole. For $0<R_n \leq \pi n$, let $N(\mathscr{F}_n, R_n)$ denote the nodal component count of $\mathscr{F}_n$ in $\mathcal{D}(p,R_n/\sqrt{n(n+1)})$, where $\mathcal{D}(p,r)$ is the geodesic ball in $\sph$ centered at $p$ and with radius $r$. We note that $\text{Vol}_{\sph}(\mathcal{D}(p,r)) = 4\pi \sin^2(r/2)$.
\begin{Theorem}\label{thmsph} There exists $c, C ,\delta>0$, such that for every $\epsilon >0$, the following holds for every $n \in \nat$ whenever $R_n$ is such that $R_n \leq \delta n$ and  $\displaystyle{\lim_{n \to \infty}} R_n = \infty$
	\begin{align*}
	\p\left( \left|\frac{N(\mathscr{F}_n,R_n)}{4 \pi n^2 \sin^2(R_{n}/2n)} - c_{NS} \right| \geq \ep \right) \leq Ce^{-c \ep^{16} R_n}.
	\end{align*}
	\end{Theorem}
\noin We note that by taking $R_n = \delta n$ in Theorem \ref{thmsph}, we get exponential concentration for the nodal component count of $\mathscr{F}_n$ in a fixed geodesic ball of radius approximately $\delta$. If we cover $\sph$ by a minimal collection of geodesic balls of radius $\delta$ and use the conclusion of Theorem \ref{thmsph} in these balls, a simple union bound recovers Result \ref{res1} (with $\ep^{16}$ in the exponent instead of $\ep^{15}$ as in Result \ref{res1}). 

\subsubsection{Arithmetic random waves on geodesic balls}  We use the notations introduced in Section \ref{secarw}. For $n \in \mathcal{S}$ and $0 <  R_n \leq \sqrt{n}/2$, let $N(\mathcal{F}_n, R_n)$ denote the nodal component count of $\mathcal{F}_n$ in $B(0,R_n/\sqrt{n}) \subset \torus$. 
	\begin{Theorem}\label{thmarw} There are constants $c,C>0$ such that the following holds. Fix $\kappa >0$ and  let $R_n$ be such that $2\pi (\log n)^{1+\frac{\log 2}{3}+ \kappa} \leq R_n \leq \sqrt{n}/2$. There exists a density one subset $\set' \subseteq \set$ such that whenever $\{n_j: j \in \nat\} \subset \set'$ satisfies $\nu_{n_j} \Rightarrow \nu$, where $\nu$ is a probability measure on $\mathbb{S}^1$ with no atoms,  then  for every $\ep >0$, we have
	\begin{align*}
	\p\left( \left| \frac{N(\mathcal{F}_{n_j},R_{n_{j}})}{ \pi R_{n_{j}}^{2}} - c_{NS}(\nu) \right| \geq \ep\right) \leq C e^{-c\ep^{16}|\Lambda_{n_j}|}. 
	\end{align*}
\end{Theorem}
\subsection{Common proof}
We essentially use the proof strategy of Result \ref{res1} to prove our results (Theorems \ref{thmrpw}, \ref{thmsph} and \ref{thmarw}). This  strategy  is quite general and relies mainly on the properties of Laplace eigenfunctions which are presented in Section \ref{reglap}. It is useful to note that these properties hold for all the manifolds we consider, namely  $\re^2$, $\sph$  and $\mathbb{T}^2$. In Theorem \ref{thmrpw}, we study the nodal component count in $B(R)$ of random functions which almost surely satisfy $\Delta F + F =0$. 
The random function $\mathscr{F}_n$ \eqref{rsh} and  $\mathcal{F}_n$ \eqref{arw} when expressed in local coordinates, followed by an appropriate scaling will correspond to random functions on $\nsq$ satisfying $\Delta_{g}f+f=0$, for some smooth metric $g$. With this transformation, nodal sets of $\mathscr{F}_n$ and $\mathcal{F}_n$ in geodesic balls (considered in Theorems \ref{thmsph} and \ref{thmarw} respectively) will correspond to nodal sets of random functions satisying $\Delta_{g}f+f=0$ in balls {$B_{g}(0,R_n)$}. The above discussion suggests that  we might be able to study the questions considered in Theorems \ref{thmrpw}, \ref{thmsph} and \ref{thmarw} in one-shot and this is the purpose of Theorem \ref{commonthm}, whose complete version and proof are presented in Section \ref{commonproof}.  We state this theorem below; we assume that  $g_n$ is a smooth metric on $(-n,n)^2$, $B_{g_n}(R)$ is the ball of radius $R$ centered at the origin, $F_n :(-n,n)^2 \ra \re$ is a centered Gaussian process satisfying $\Delta_{g_n}F_n + F_n =0$, $N_R$ and  $N_{R}(\cdot,\scrt)$ defined in \eqref{defcompcount} are nodal component counts in $\bgr$.
\vskip .2cm
\begin{thmn}[\ref{commonthm}] 
	Suppose that $F_n$, $g_n$ and $R=R_n$ satisfy the assumptions \ref{as1}--\ref{as6}. Then there are constants $c,C>0$ such that for every $\ep >0$, the following holds for every $n \in \nat$ and for $n_R(\cdot) = N_R(\cdot)$ and $N_{R}(\cdot,\mathscr{T})$ 
	\begin{align*}
	\p\left( \left|\frac{n_{R}(F_n)}{\textup{Vol}_{g_n} [B_{g_n}(R)]} - \emph{Median}\left(\frac{n_{R}(F_n)}{\textup{Vol}_{g_n} [B_{g_n}(R)]}\right)\right| \geq \ep \right) \leq Ce^{-c\ep^{16}/\psi_{n}(R)},
	\end{align*}
	where $\psi_n$ is as in \eqref{a5}.
\end{thmn}
\subsection{Plan of the paper}  
In Section \ref{secdis}, we present the key ideas in the proof of our results and discuss the main challenge in the proof. 
 In Section \ref{secprelim}, we present some known results about Laplace eigenfunctions and other probabilistic ingredients used in our proofs.

 Sections \ref{pf1}, \ref{pf2} and \ref{pfarw} are devoted to showing how Theorems \ref{thmrpw}, \ref{thmsph} and \ref{thmarw} follow from Theorem \ref{commonthm}. We prove Theorem \ref{thmlowerbound} in Section \ref{seclowerbound}. In Section \ref{seclemmaproof}, we prove Lemma \ref{lemmamain} which is the main technical ingredient used in proving Theorem \ref{thmsph}.


\section{Discussion and Idea of the proof} \label{secdis}
\subsection{Semi-locality of the nodal component count} \label{semiloc}
Unlike measuring nodal volume, counting nodal components is a non-local problem.   To know precisely the total nodal volume, one merely has to tile the space by \textit{small} sized sets  and then sum up the nodal volume in all of them. This prescription  does not, however, work for counting nodal components. One of the main reasons which enabled the study of nodal component count of random functions in \cite{NS2}, \cite{NS} and \cite{Rozenshein} is the fact that  the nodal component count for those random functions  turned out to be a \textit{semi-local} quantity, i.e., an overwhelming proportion of the nodal components have \textit{moderate} diameter. Hence tiling the space with moderate sized sets and summing the nodal component  count in all  these sets provides a very good estimate for the total nodal component count.  Nazarov--Sodin (Result \ref{res1}) and Rozenshein (Result \ref{rozresult}) establish semi-locality required in their respective study of nodal component count with the following arguments.
\begin{enumerate}[ align=left,leftmargin=*,widest={7}]
	\item 
	An upper bound on the total nodal length of spherical harmonics in \eqref{yconj} immediately implies an upper bound on the count of nodal components  whose diameter is \textit{large} and hence \textit{most} of the nodal components  have \textit{moderate} diameter. 

	\item 
	For the case of arithmetic random waves on $\mathbb{T}^d$ ($d \geq 2$), the specific structure of the Laplace eigenfunctions -- that they are trigonometric polynomials -- is used to conclude semi-locality.  
\end{enumerate}

\subsection{The Challenge} \label{chal} The main challenge in our proof is in establishing semi-locality of the nodal component count for the random functions on growing balls in the plane and geodesic balls  in  $\sph$ and $\mathbb{T}^2$, considered in Theorems \ref{thmrpw}, \ref{thmsph} and \ref{thmarw} respectively. We explain this only for the case of $\re^2$, the situation  for $\sph$ and $\tor$ is similar. The ideas used to establish semi-locality in the two instances discussed in Section \ref{semiloc} cannot be directly used in our situation for the following reasons.

\begin{itemize}[ align=left,leftmargin=*,widest={7}]
	\item \textit{No analogous statement of Yau's conjecture.} A statement about nodal length for plane waves which is  analogous to Yau's conjecture \eqref{yausconj} and one which would have sufficed to establish semi-locality is as follows: there is a constant $C >0$ such that nodal length  of every plane wave in $B(R)$ does not exceed $CR^2.$
	But such an estimate is \textbf{not} true\footnote{Thanks to Igor Wigman for pointing this out.} for the following reason. For every $n \in \nat$, consider the plane wave defined in polar coordinates by $h_n(r,\theta) := J_n(r) \cos n\theta$, where $J_n$ is the Bessel function. The $2n$ many angularly equi-spaced lines emanating from the origin which are solutions of $\cos n\theta =0$ are in the nodal set of $h_n$ and hence 
	\begin{align*}
	\textup{Nodal length of $h_n$ in $B(R)$} \geq 2nR. 
	\end{align*} 
	\item \textit{Plane waves do not have any nice algebraic structure.} While Laplace eigenfunctions on $\sph$  and $\tor$  are polynomials, plane waves do not have any such nice algebraic feature.  
\end{itemize}

\subsection{Idea of the proof} 
 The proof of our concentration results (Theorems \ref{thmrpw}, \ref{thmsph} and \ref{thmarw})  roughly consist of  the following steps \ref{step2}--\ref{step4}, which are  essentially the steps  in the proof of Result \ref{res1}. This is how the steps in our proof compare with the corresponding ones in the proof of Result \ref{res1}:  \ref{step4} is an easy adaptation, establishing \ref{step2} and \ref{step3} require much more effort.
 Step \ref{step2} which involves obtaining a relation between the Cameron--Martin norm and the $L^2$ norm is quite technical for $\re^2$ (Lemma \ref{lemnorm}) and $\sph$ (Lemmas \ref{lemsph1} and \ref{lemmamain}), whereas on $\tor$ it very easily follows from Result \ref{granwig}. Step \ref{step3} is about establishing semi-locality of the nodal component count and this is done in Lemma \ref{lengthlemma}.

We now sketch the key ideas in the proof of the concentration result for the  random plane wave $F_{\nu_0}$, all other cases are similar.
The concentration result \eqref{concrpw} for the random plane wave reads as follows, there are constants $a, A >0$ such that for every $\ep >0$ and every $R>0$ we have
\begin{align*}
	\p\left(\left| \frac{N_R (F_{\nu_0})}{ \pi R^2} - c_{NS}(\nu_0) \right| > \ep \right) \leq A  e^{-a\ep^{16}R }.
\end{align*}
In what follows, $\mathcal{H}$ denotes the Cameron-Martin space of the Gaussian process $F_{\nu_0}|_{[-R,R]^2}$ and $\|\cdot\|_{\mathcal{H}}$ denotes the norm in $\mathcal{H}$ (see Appendix \ref{app1} for relevant definitions). 
\begingroup
\renewcommand\thesubsubsection{\roman{subsubsection}}
\subsubsection{Main tool} \label{step1}
The main tool used to prove concentration is Lemma \ref{gaussianconct}, which is essentially the  Gaussian concentration result (\cite{Bogachev}, Theorem 4.5.6, p.176). To apply this, it suffices to show that $N_R(\cdot)/\pi R^2$ is uniformly lower semi-continuous except possibly on an exceptional set $\mathscr{E}$, i.e., 

\begin{equation}\label{texteq1}
\parbox{.85\textwidth}{Given $\ep >0$, there is $\rho >0$ and an exceptional set of plane waves $\mathscr{E}$, with $\p(\mathscr{E}) \leq e^{-cR}$, such that for every plane wave $f \notin \mathscr{E}$ and every $h \in \mathcal{H}$ with $\|h\|_{\mathcal{H}} \leq \rho \sqrt{R}$, the following holds
\begin{equation*}
\frac{N_R(f+h)}{\pi R^2} - \frac{N_R(f)}{\pi R^2} \geq -\ep.
\end{equation*} }
\end{equation}

\subsubsection{Relation between the Cameron--Martin norm and the $L^{2}(B(R))$ norm.}\label{step2}
We now explain the significance of this step. 
Lemma \ref{lengthlemma1} and Result \ref{NDcount} imply that for a smooth function $f$ defined on a disc $\mathcal{D}$, some knowledge of the values of $f$ and  its higher derivatives give information about the nodal set of $f$. If $f$ is a plane wave, then it follows from the regularity estimates \eqref{regpr} that the values of $f$ and its higher derivatives are controlled by the $L^2$ norm of $f$.  Hence the $L^2(B(R))$ norm of $f$ contains some information about the nodal set of $f$ in $B(R)$. 

We will be required to show in the course of the proof that certain events (concerning the nodal set of $F_{\nu_0}$ in $B(R)$) expressed in terms of the values of the field and  its higher derivatives have probability smaller than $e^{-cR}$ and to show this, we use Results \ref{thmisop} and  \ref{corisop}. Hence for this purpose and to establish \eqref{texteq1}, we are required to know the relation between the Cameron-Martin norm and the $L^2(B(R))$ norm. For the random plane wave the relation is as follows, there is $C>0$ such that for every $h \in \mathcal{H}$ and  every $R>0$, we have
\begin{align*}
\|h\|^{2}_{L^{2}(B(R))} \leq CR~ \|h\|^{2}_{\mathcal{H}}.
\end{align*}
Such a relation for other random monochromatic waves is established in Lemma \ref{lemnorm}. 

\subsubsection{Establishing semi-locality of the nodal component count.}\label{step3}
	\begin{itemize}[wide=0pt]
	\item Like in the proof of Result \ref{res1}, we also establish semi-locality of the nodal component count by obtaining a bound on the nodal length. But as was observed in Section \ref{chal}, it is not possible to get a deterministic bound on the nodal length.  We shall instead prove a \textit{probabilistic length bound} which is as follows. For $\delta>0$ (small), there is an exceptional set of plane waves $E^{*}$, with $\p(\ee) \leq e^{-cR}$, such that for $f \notin \ee$ we have
	\begin{align} \label{nodlen}
	\textup{Nodal length of $f$ in ($B(R) \setminus$ a \textit{negligible region}) } \leq C_{\delta}R^2,
	\end{align}
	where $C_{\delta} >0$. The \textit{negligible region} mentioned above has negligible area  (not exceeding $\delta R^2$) and hence \eqref{fk}  implies that the nodal component count in this \textit{negligible region} is also negligible. Hence \eqref{nodlen} provides a nodal length bound for $f$ outside  this \textit{negligible region} and this establishes semi-locality of the nodal component count of $f$. 
	\vskip .2cm
	\item We now sketch the idea for \textit{constructing} $\ee$. The first step is to identify regions in the domain of a function where its nodal length is possible large. The key result used for this is Corollary \ref{lengthcor}; for a function $f$ defined on a unit disc, this result gives a bound for its nodal length in terms of its derivatives. For a plane wave $f$, a point $p \in \re^2$ is called an  \textit{\uns point} if both $d_1 (f,p)$ and $d_2 (f,p)$, defined as in \eqref{ddef}, are \textit{small}. Corollary \ref{lengthcor} indicates that it is only in the neighbourhoods of these \uns points, the nodal length is possibly large and this motivates the following defintion of $\ee$.  We cover $B(R)$ with a minimal collection of unit discs $\{B_j\}_{j \in \mathcal{J}}$ and define $\ee$ to be the collection of plane waves $f$  given by
	\begin{align*}
	\ee := \{f: \textup{a \textit{sizeable} proportion of $\{B_j\}_{j \in \mathcal{J}}$ contains an \uns point of $f$}\}.
	\end{align*}
	For $f \notin \ee$, by definition, only a \textit{negligible} proportion of $\{B_j\}_{j \in \mathcal{J}}$ contains an \uns point of $f$, call such a disc an \textit{\uns disc}. It is only in these \uns discs, the nodal length is possibly \textit{large} and the role of the \textit{negligible region} in \eqref{nodlen} is played by the union of the \uns discs. Result \ref{corisop} is then used to show that $\ee$ has probability less than $e^{-cR}$.
\end{itemize}
\subsubsection{Identifying the exceptional set and concluding concentration.} \label{step4}
\begin{itemize}[wide=0pt]
	\item In order to establish \eqref{texteq1}, there is a need to understand how the nodal set of a function changes upon perturbing the function and Result \ref{NDcount} serves this purpose. Let $f$ be a $C^1$ function on a domain $U \subseteq \re^2$, call a point $p \in U$ an \textit{unstable point} of $f$ if both $|f(p)|$ and $| \nabla f(p)|$ are simultaneously \textit{small}. Using Result \ref{NDcount}, we  identify neighbourhoods of unstable points as the only regions in $U$ where the nodal set can possibly undergo a significant change upon perturbing $f$. Hence we  conclude that if $U=D$, a disc and it does not contain any unstable point of $f$, then its nodal components which are  contained well within $D$ are \textit{preserved} even after perturbing $f$.
	\vskip .2cm
	\item We now sketch the idea for identifying the set $\mathscr{E}$ in \eqref{texteq1}. Cover $B(R)$ with a minimal collection of  \textit{moderate} sized discs $\{D_i\}_{i \in \mathcal{I}}$.  $E$ is the collection of plane waves $f$ defined by
	\begin{align*}
	E := \{\mbox{$f :$ a \textit{sizeable} proportion of $\{D_i\}_{i \in \mathcal{I}}$ contain an unstable point of $f$}\}. 
	\end{align*}
	
	Define the exceptional set $\mathscr{E} := E \cup \ee$ and let $f \notin \mathscr{E}$. Since $f \notin \ee$, it follows by semi-locality that \textit{most} of the nodal components of $f$ are contained well within one of the discs $D_i$. Since $f \notin E$, only a \textit{negligible} proportion of the discs $\{D_i\}_{i \in \mathcal{I}}$ contain an unstable point of $f$ and the nodal components contained well within all the other discs are \textit{preserved} even after perturbing $f$. By \eqref{fk}, the number of nodal components contained in this negligible proportion of discs is also negligible. Result \ref{corisop}  is used to show that $\p(E) \leq e^{-c'R}$.
	\vskip .2cm
	\item We finally see how to conclude concentration. Let $f \notin \mathscr{E}$. Observe that if $h \in \mathcal{H}$ with $\|h\|_{\mathcal{H}} \leq \rho \sqrt{R}$ (and hence $\|h\|^{2}_{L^{2}(B(R))} \lesssim \rho^2 R^2$) with $\rho$ small enough, then by the regularity estimates \eqref{regpr}, in a \textit{large} proportion of the discs $D_i$, the $L^{\infty}$ norm of $h|_{D_i}$ is \textit{small}. Hence for such $D_i$ and those which do not contain an unstable point of $f$, the discussion above can be used to conclude that
	\begin{align*}
	N_R(f+h) \geq N_R(f) - \ep R^2,
	\end{align*}
	where $\ep R^2$ is an upper bound for the count of nodal components which are contained in the discs $D_i$ which contain an unstable point and those nodal components whose diameter is \textit{large}. This establishes \eqref{texteq1} and hence the concentration of $N_R(F_{\nu_0})/\pi R^2$. 
\end{itemize}

\endgroup

\section{Preliminaries} \label{secprelim}
In this section, we first present some deterministic results on nodal length, nodal component/domain count and some basic properties of Laplace eigenfunctions which will be used to prove our results. We end this section with the main probabilistic ingredient, namely the Gaussian concentration result, used to establish the concentration results.  
\begin{Notation} For $r>0$, $r\mathbb{D}$ denotes the open ball in $\re^2$  centered at  the origin and of radius $r$. When there is no ambiguity about the underlying metric, $B(p,r)$ will denote the geodesic ball centered at $p$ with radius $r$. 
\end{Notation}
\subsection{Nodal length of a smooth function} The following results about nodal length are taken from the work of Donnelly-Fefferman \cite{DF} and modified a bit to better suit our needs. 

\begin{Lemma}[\cite{DF}, Lemma 5.8] \label{lem00} Let $f: (-1,1) \ra \re$ be a smooth function and  assume that there are constants $M >A >0$ and $n \in \mathbb{N}$ satisfying
	\begin{align*}
	\max_{0 \leq j\leq n} |f^{j}(0)| \geq A~\text{and}~\max_{0 \leq j \leq n+1} \Vert f^{j}\Vert_{L^{\infty}(-1,1)} \leq M.
	\end{align*}
Then $f$ has at most $n$ distinct zeros in $\mathcal{I} := [-A/2M, A/2M]$. 
\end{Lemma}
\begin{proof}
Assume to the contrary that $f$ has at least $(n+1)$ distinct zeros in $\mathcal{I}$. Then for every $ 0 \leq j \leq n$, $f^{j}$ has at least $(n-j+1)$ distinct zeros in $\mathcal{I}$. 
Let $0 \leq k \leq n$ be such that $|f^{k}(0)| \geq A$ and let $t \in \mathcal{I}$ be such that $f^{k}(t) =0$. Then we have
\begin{align*}
\int_{0}^{t} f^{k+1}(x)dx & = f^{k}(t) - f^{(k)}(0) = - f^{k}(0).
\end{align*}
From the bound on the $L^{\infty}$ norms of the derivatives we have
\begin{align*}
\left|\int_{0}^{t} f^{k+1}(x)dx\right| &\leq Mt \leq M\cdot \frac{A}{2M} = \frac{A}{2}, 
\end{align*}
and hence 
$ |f^{k}(0)|  \leq A/2$, which contradicts our choice of $k$ and proves our claim.
\end{proof}
\noin The following result appears in the proof of Lemma 5.11 in \cite{DF}.
\begin{Lemma}
	 \label{lem11} Let $f:(a_1,a_2)\times (b_1,b_2) \rightarrow \re$ be a smooth function. Let $N_1(x)$ and $N_2(y)$ denote the number of intersections of the nodal set $\mathcal{Z}(f)$ with the lines 
$\{x\} \times (b_1,b_2)$ and $(a_1,a_2) \times \{y\}$  respectively. Then,
\begin{align*}
{\rm length} \{z \in [a_1,a_2]\times [b_1,b_2] : f(z)=0,~\nabla f(z)\ne 0\} \leq \sqrt{2}\left(\int_{a_1}^{a_2} N_1(x) dx + \int_{b_1}^{b_2} N_2(y) dy \right).
\end{align*} 
\end{Lemma}
\noin The following result is  inspired by Lemma 5.11 in \cite{DF}.
\begin{Lemma}
	 \label{lengthlemma1} Let $n \in \nat$ and $f: \dis \ra \re$  be a smooth function. For $p \in \dis$ and  $\ell = 1,2$ define
\begin{align}\label{definitiond}
d_{\ell}(f,p) &:= \max_{0 \leq j \leq n} |\partial_{\ell}^{j}f(p)|.
\end{align}
Assume that there are constants $M >A >0$ such that 
\begin{align*}
\min_{\ell =1,2}d_{\ell}(f,(0,0)) \geq A ~\text{and}~\max_{0 \leq j \leq n+1}  \Vert \nabla^{j} f \Vert_{L^{\infty}(\mathbb{D})}  \leq M.
\end{align*}
Then we have the following upper bound on the nodal length
\begin{align*}
\textup{length}\left\{z: |z|< A/4M, f(z)=0, \nabla f \neq 0\right\} \leq \sqrt{2}nA/M.
\end{align*}
\end{Lemma} 

\begin{proof} Since $d_1(f,(0,0)) \geq A$, there is a $j$ such that  $0 \leq  j \leq n$ and $|\partial_{1}^{j}f(0,0)| \geq A$. For every $y$ such that $|y| \leq A/2M$,  we have
\begin{equation*}
\begin{gathered}
\partial_{1}^{j}f(0,y) - \partial_{1}^{j}f(0,0) = \int_{0}^{y} \partial_2 \partial_{1}^{j}f(0,s) ds,\\
\text{hence }|\partial_{1}^{j}f(0,y) - \partial_{1}^{j}f(0,0)| \leq M \cdot \frac{A}{2M} = \frac{A}{2}, \\
\text{hence }|\partial_{1}^{j}f(0,y)|  \geq |\partial_{1}^{j}f(0,0)| - \frac{A}{2} \geq \frac{A}{2}.
\end{gathered}
\end{equation*}
Hence for every  $y$ such that $|y| \leq A/2M$, we conclude using Lemma \ref{lem00} that the number of zeros of $f(\cdot, y)$ in $\left[-A/4M, A/4M \right]$ is at most $n$. Using a similar argument we conclude that for every $x$ satisfying $|x| \leq A/2M$, the number of zeros of $f(x,\cdot)$ in $\left[-A/4M, A/4M \right]$ is at most $n$. It now follows from Lemma \ref{lem11} that 
\begin{align*}
\mbox{length}\left\{z \in \left(-\frac{A}{4M}, \frac{A}{4 M} \right)^2 : f(z)=0, \nabla f \neq 0\right\} \leq \sqrt{2}n \left( \frac{A}{2M} + \frac{A}{2M} \right) = \sqrt{2}n\frac{A}{M}.
\end{align*}
\end{proof}
\begin{Cor}\label{lengthcor} Let $n \in \nat$ and  $f: 2\mathbb{D} \ra \re$ be a smooth function. Assume that there are constants $M >A >0$ such that for every $p \in 2\mathbb{D}$, 
	\begin{align*}
	d_1(f,p) \wedge d_2(f,p) \geq A~\text{and}~\max_{ 0 \leq j \leq n+1} \|\nabla^j f\|_{L^{\infty}(2\mathbb{D})} \leq M,
	\end{align*}
	where $d_{\ell}(f,p)$ is as in \eqref{definitiond}.
	Then we have the following bound for the nodal length
	\begin{align*}
	\textup{length}\left\{z \in \mathbb{D} : f(z)=0, \nabla f \neq 0\right\}
	 \leq (64\sqrt{2}n)\frac{M}{A}.
	\end{align*}
\end{Cor}
\begin{proof}
	We can cover $\mathbb{D}$ by  $64 M^2/A^2$ many  balls of the form {$B(p,A/4M)$}, where $p \in \mathbb{D}$. We conclude using Lemma \ref{lengthlemma1} that the nodal length in every such ball $B(p,A/4M)$ does not exceed $\sqrt{2}nA/M$ and hence the nodal length in $\mathbb{D}$ does not exceed $(\sqrt{2}nA/M) \cdot (64M^2/A^2) \leq 64\sqrt{2}n M/A $.
	\end{proof}

\subsection{Counting nodal components} 

The following result  which goes by the name of \textit{barrier method} or \textit{shell lemma} is a deterministic result which is useful in understanding how the count of nodal components of a function changes upon perturbing the function. 
\begin{Result}[\cite{NS}, Claim 4.2, Corollary 4.3 and \cite{Rozenshein}, Propositions 4.3, 4.4, Lemma B.1]\label{NDcount}
	Let $\alpha, \beta >0$, $U \subseteq \re^2$ an open connected set. Let $f: U \ra \re$ be a $C^1$ function such that for every $z \in U$, either $|f(z)| > \alpha$ or $|\nabla f(z)| > \beta$. Let $\Gamma$ be any component of $\mathcal{Z}(f)$ which satisfies $d(\Gamma, \partial U) > \alpha/\beta$, denote by $S_{\Gamma}$ the component of $|f| < \alpha$ which contains $\Gamma$. Then $S_{\Gamma}$, called a \textit{shell}, is diffeomorphic to $\Gamma \times (-1,1)$ and is a subset of 	$\Gamma_{+\alpha/\beta}$. Hence $S_{\Gamma}$ has exactly two boundary components, one of which  satisfies  $f= \alpha$ and the other satisfies $f = -\alpha$. Moreover, for two such distinct components $\Gamma_1$ and $\Gamma_2$ of $\mathcal{Z}(f)$, the corresponding shells satisfy $S_{\Gamma_1} \cap S_{\Gamma_2} = \phi$. As a consequence of this, the following hold. 
	\begin{itemize}[align=left,leftmargin=*,widest={10}]
		\item If $h: U \ra \re$ is a continuous function such that $|h| < \alpha$ on $U$, then the zero set $\mathcal{Z}(f+h) \subseteq \{|f|<\alpha\}$. For $\Gamma_1, \Gamma_2$ as above, in each of the shells $S_{\Gamma_1}$ and $S_{\Gamma_2}$, there is at least one nodal component of $\mathcal{Z}(f+h)$ and these components are distinct.
		\item If $h: U \ra \re$ is a $C^1$ function satisfying $|h| < \alpha/2$ and $|\nabla h|< \beta/2$ on $U$, then there is a unique component $\widetilde{\Gamma}$ of $f+h$ in $S_{\Gamma}$.
	\end{itemize}
Figure \ref{fig:fig1} is a pictorial representation of this result. 
	\end{Result}
\subsection{Properties of Laplace eigenfunctions} \label{reglap}
Let $X=\re^2$, $\sph$ or $\mathbb{T}^2$ equipped with the Euclidean, spherical or the flat metric respectively.  Let $f$ be an eigenfunction of the Laplacian on $X$ satisfying $\Delta f + \lambda f =0$, for every such $\lambda$ and $f$ the following estimates hold.  
\subsubsection{Estimates on the wavelength scale} 
For every $r>0$ and every $j \in \nat$, there are constants $C_r, C_{r,j} >0$ such that for every $p \in X$ we have
\begin{equation}\label{regpr}
\begin{aligned}
|f(p)|^2 &\leq C_r~ \lambda \int_{B(p,r/\sqrt{\lambda})}f(x)^2 dV(x), \\
|\nabla^{j} f(p)|^2 & \leq C_{r,j}~ \lambda^{j+1} \int_{B(p,r/\sqrt{\lambda})} f(x)^2 dV(x).
\end{aligned}
\end{equation}
\subsubsection{Lower bound on the area of a nodal domain} 
It follows from the Faber-Krahn inequality (\cite{Chavel}, Chapter 4, p.86) that there is a constant $C>0$ satisfying
\begin{align}\label{fk}
\textup{Area of every nodal domain of $f$} \geq C/\lambda .
\end{align}
\subsubsection{Singular zeros of $f$} It follows from the proof of Lemma 3.1 in \cite{DF} that the set $S$ defined below is a discrete set in $X$
\begin{align*}
S := \{z \in X: f(z) = 0\text{ and }\nabla f(z) = 0\}. 
\end{align*}
This fact is useful for the following reason. For every $U \subseteq X$, it follows that the length of $(\mathcal{Z}(f) \setminus S) \cap U$ equals the length of $\mathcal{Z}(f) \cap U$ and hence Lemma \ref{lengthlemma1} and Corollary \ref{lengthcor} can be used to get bounds for the nodal length of $f$.

\subsection{Yau's Conjecture} \label{secyauconj} In \cite{Yau1} and \cite{Yau2}, Yau proposed the following conjecture about the nodal volume of Laplace eigenfunctions on any smooth, closed Riemannian manifold $(M,g)$. There are constants $c_{M},C_{M} >0$ such that for every $h: M \ra \re$ satisfying $\Delta_{g} h + \lambda h =0$, the following holds
\begin{align} \label{yausconj}
c_{M}  \sqrt{\lambda} \leq  \textup{Vol}(h^{-1}(0)) \leq C_{M} \sqrt{\lambda}.
\end{align}
This conjecture was resolved by Donnelly-Fefferman in \cite{DFF} when the metric $g$ is real analytic. Hence for $M = \sph$ or $\mathbb{T}^2$ and for every $f$ satisfying $\Delta f+ \lambda f=0$ on $M$, we have
\begin{align} \label{yconj}
c_{M}  \sqrt{\lambda} \leq  \textup{Nodal length of $f$ in $M$} \leq C_{M} \sqrt{\lambda}.
\end{align}
\subsection{Isoperimetric inequality and the concentration result}\label{isoconc}

In Appendix \ref{app1}, a notion of Gaussian measure on an infinite dimensional space is defined. This can be viewed as a generalization of the finite dimensional Gaussian distributions. We also show how one can view a Gaussian process on $\re^2$ as a Gaussian measure. Just like their finite dimensional counterparts, the infinite dimensional Gaussian measures also satisfy analogous isoperimetric inequalities and concentration results. The results below and the discussion in Appendix \ref{app1} are mainly based on \cite{Bogachev}.

\begin{Result} [\cite{Bogachev}, Theorem 4.3.3] \label{thmisop} Let $\gamma$ be a Radon  Gaussian measure on a locally convex space $X$ and let $A \subseteq X$ be measurable. Let $U_{\mathcal{H}_{\ga}}$ be the closed unit ball in the Cameron-Martin space $\mathcal{H}_{\gamma}$. Then for every $t \geq 0$,
\begin{equation*}
\gamma(A+tU_{\mathcal{H}_{\ga}}) \geq \Phi(c+t),
\end{equation*}
where $\Phi$ is the distribution function of the standard Gaussian, $\Phi(s):= \int_{-\infty}^{s}\frac{1}{\sqrt{2\pi}} e^{-x^2/2} dx$ and $c \in \re$ is such that $\gamma(A)=\Phi(c)$.
\end{Result} 
\noin The next result is an easy consequence of Result \ref{thmisop}. 
\begin{Result} \label{corisop} In the setting of Result \ref{thmisop},  there is a constant $b_0>0$ such that whenever $\gamma(A+tU_{\mathcal{H}_{\ga}}) \leq 3/4$, we have $\ga(A) \leq e^{-b_0 t^2}$.
\end{Result}
The next lemma is a generalization of the well known Gaussian concentration result and  is essentially the argument following Lemma 4.1 in \cite{NS}.
\begin{Lemma}\label{gaussianconct} 
	In the setting of Result \ref{thmisop}, let $F: X \ra \re$ be a measurable function. Suppose that for every $\epsilon >0$, there is   $\eta >0$ and a measurable set $\E \subset X$ with $\gamma(\E) \leq 1/4$ such that whenever $x \in X \setminus \E$ and $u \in \mathcal{H}_{\gamma}$ with $\|u\|_{\mathcal{H}_{\gamma}} \leq \eta$, the following holds
	\begin{align*}
	F(x+ u) \geq F(x) - \ep,
	\end{align*}
	then $F$ concentrates around its median, that is there exists $c>0$ such that
	\begin{align*}
	\p(|F-\textup{Med}(F)| >\ep) \leq   e^{-c\eta^2} + \gamma(\E).
	\end{align*}
\end{Lemma}
\begin{proof}
	Let $\mathcal{M} = \textup{Med}(F)$ and $\mathcal{C} := \{F \leq \mc\}$. We can now write
	\begin{align*}
	\{|F-\mc| > \ep\} = \underbrace{\{F > \mc +\ep\}}_{=: S_1} \cup \underbrace{\{F < \mc - \ep\}}_{=: S_2}.
	\end{align*}
	Suppose $y \in (\mathcal{C}+\eta U_{\mathcal{H}}) \cap \E^{c}$. Then $y = x+ u$, where $x \in \mathcal{C}$, $u \in \eta U_{\mathcal{H}}$. Thus we have
	\begin{align*}
	\mc \geq F(x) = F(y-u) \geq F(y) - \ep.
	\end{align*}
	Hence $F(y) \leq \mc + \ep$. Thus $S_1 \subseteq (\mathcal{C}+\eta U_{\mathcal{H}})^{c} \cup \E$ and hence it follows from Result \ref{thmisop} that 
	\begin{align*}
\p(S_1) \leq \frac{1}{\sqrt{2\pi}}e^{-\eta^2/2} + \gamma(\E).
	\end{align*}

	 If $z \in (S_2  + \eta U_{\mathcal{H}}) \cap \mathscr{E}^{c}$, then $z = x + u$ with $x \in S_2$ and $u \in \eta U_{\mathcal{H}}$. Hence we have
	\begin{align*}
	\mc - \ep > F(x) = F(z - u) \geq F(z) - \ep,
	\end{align*}
	hence $F(z) < \mc$. Thus $(S_2  + \eta U_{\mathcal{H}}) \subseteq \E \cup \{F < \mc\}$ and hence we have
	$$
	\p(S_2 + \eta U_{\mathcal{H}}) \leq \frac{1}{2} + \gamma(\E) \leq \frac{3}{4},
	$$
and  we conclude using Result \ref{corisop} that $\p(S_2) \leq e^{-b_0 \eta^2}$. Hence  there is  $c >0$ such that
	\begin{align*}
		\p(|F-\textup{Med}(F)| >\ep) \leq \frac{1}{\sqrt{2\pi}}e^{-\eta^2/2} + \gamma(\E) + e^{-b_0 \eta^2} \leq e^{-c\eta^2} + \gamma(\E).
	\end{align*}
\end{proof}

\section{Common Proof}\label{commonproof}
For $n\in\nat$, let $X_n := ([-n,n]^2, g_n)$, where $g_n$ is a smooth metric.
In this section, we consider ensembles of smooth centered Gaussian processes $\{F_n\}_{n \geq 1}$, where each $F_n : X_n \ra \re$ almost surely satisfies $\Delta_{g_n} F_n +F_n =0$ on $\nsq$. We prove concentration results for the count of nodal components of such Gaussian processes using ideas developed by Nazarov--Sodin in \cite{NS}. The setup is quite general and it allows us to tackle in one go all  cases of interest to us, namely the random monochromatic waves, random spherical harmonics and arithmetic random waves. 
We start by listing the assumptions which the field $F_n$ and the metric $g_n$ must satisfy,  we then state and prove the concentration results.

\begin{Notations}\hfill
	\begin{itemize}[align=left,leftmargin=*,widest={10}]
		\item We let $g$ denote the Euclidean metric on $\re^2$. 
		\item 
		$\mathbb{D}$ denotes the unit ball in $(\re^2,g)$ centered at the origin, while $\mathbb{D}_p$  denotes the unit ball in $(\re^2,g)$ centered at $p$. For $r>0$, $B_{g_n}(r)$  denotes the open ball in $X_n$ of radius $r$ centered at the origin, while $B_{g_n}(p,r)$ denotes the open ball in $X_n$ with radius $r$ and center $p$. For $D = B_{g_n}(p,r)$ and $a>0$, $aD$ denotes $B_{\gn}(p,ar)$. We use $B(r)$ and $B(p,r)$ to denote the corresponding balls in $(\re^2,g)$.
		\item For $\Gamma$  a piecewise smooth curve in $\nsq$,  $\mathcal{L}_{g_n}(\Gamma)$ denotes the length of  $\Gamma$ in $X_n$.
		\item  $\Delta_{g_n}$, $\nabla_{g_n}$, Vol$_{g_n}$ and $\mathcal{L}_{g_n}$ denote the Laplace-Beltrami operator, gradient, volume and length in $X_n$. 
		$\Delta$, $\nabla$, Vol and $\mathcal{L}$ denote the same operators/quantities in $(\re^2,g)$.
		\item For $r>0$, a continuous function $f:X_n \ra \re$ and $\mathscr{T} \in \mathcal{T}$, we define the following
		\begin{equation}\label{defcompcount}
		\begin{aligned}
		N_r(f) &:= \sharp\{\mbox{nodal components of $f$ contained in $B_{g_n}(r)$}\},\\
		N_r(f,\mathscr{T})&:= \sharp\{\mbox{nodal components of $f$ contained in $B_{g_n}(r)$ whose tree end is $\mathscr{T}$}\}.
		\end{aligned}
		\end{equation}
		\item We let $\mathcal{H}_n$ denote the Cameron-Martin space of $F_n$, the norm in this space is denoted by $\|\cdot\|_{\mathcal{H}_n}$ and $U_n$ denotes the closed unit ball in $\mathcal{H}_n$.
	\end{itemize}
	\end{Notations}

\begin{Assumptions} \label{assump1} Assume that $F_n$, $g_n$ and $R=R_n$ satisfy the following conditions.
\begin{enumerate}[align=left,leftmargin=*,widest={9},label=({$A$\arabic*})]
	\item\label{as1} There are constants $~m_{+} \geq m_{-} >0$ such that for every $n \in \nat$, every $z \in (-n,n)^2$ and every $v \in T_{z}(X_n)$, we have
	\begin{align*}
	m_{-}^{2} \|v\|^2 \leq g_n (z)(v , v) \leq m_{+}^{2} \|v\|^2.
	\end{align*}
	\item \label{as2} For every $n \in \nat$, almost surely $F_n$ satisfies $\Delta_{g_n} F_n + F_n =0$ on $(-n,n)^2$.
	\item \label{as3} For every  $n \in \nat$ and every $z \in [-n,n]^2$, $\e[F_{n}^{2}(z)]=1$.
	\item \label{as4} For every positive integer $j \leq 106$,  there is $\kappa_j >0$ such that for every $n \in \nat$, $\ell=1,2$ and every $z \in B_{g_n}(R)$, the Gaussian vector $(F_{n}(z),\partial_{\ell} F_{n}(z), \partial_{\ell}^{2}F_{n}(z),\ldots,\partial_{\ell}^{j}F_{n}(z))$ is non-degenerate and its density is bounded above by $\kappa_j$.
\item[$(A4')$] There is $b>0$ such that for every $n \in \nat$ and every $z \in \bgr$, the Gaussian vector $(F_n(z), \partial_1 F_n(z), \partial_2 F_n(z))$ is non-degenerate and its density is bounded above by $b$. 

	\item  \label{as5} For every $n \in \nat$, there are functions $\psi_{n}:(0,\infty) \ra (0,\infty)$ such that for every $h \in \mathcal{H}_n$, 
	the following relation between the $L^2$ norm and the Cameron-Martin norm holds,
	\begin{align}\label{a5}
	\int_{B_{g_n}(R)} h^2(z) dz \leq R^2 \psi_{n}(R)\|h\|^{2}_{\mathcal{H}_n}.
	\end{align}
		\item \label{as6} \textit{Uniform regularity estimates.} For every $r>0$ and  every $j \in \mathbb{N}$, there are constants $\widetilde{C}_r , \widetilde{C}_{r,j} >0$ such that   for every $f: X_n \ra \re$ satisfying $\Delta_{\gn} f+f =0$ and every $p \in \nsq$ satisfying $d(p,\partial X_n) >r/\mmin$, the following estimates hold
	\begin{align*}
	f^2(p) & \leq \widetilde{C}_r \int_{B_{g_n}(p,r)} f^2(z)~ dV_{\gn}(z),\\
	|\nabla^{j} f(p)|^2 & \leq \widetilde{C}_{r,j} \int_{B_{g_n}(p,r)} f^2(z)~ dV_{\gn}(z),
	\end{align*}
	where $d(\cdot,\cdot)$ denotes the distance in $(\re^2,g)$.
	
	
	\end{enumerate}
\end{Assumptions}
\begin{Consequences} Here are some easy consequences of the assumptions above. 
\begin{enumerate}[align=left,leftmargin=*,widest={9}, label=({$C$\arabic*})]
	\item \label{con1} For every $p \in \nsq$ and every $r>0$, we have
	\begin{align*}
	B\left(p,\frac{r}{\mpl}\right) \cap X_n \subseteq B_{g_n}(p,r) \subseteq B\left(p,\frac{r}{\mmin}\right).
	\end{align*}
	\item \label{con2} For every measurable $S \subseteq \nsq$, we have
	\begin{align*}
	\mmin^2\mbox{Vol}_{g}(S) \leq \mbox{Vol}_{g_n}(S) \leq \mpl^2 \mbox{Vol}_{g}(S). 
	\end{align*}
	\item \label{con3} \textit{Uniform lower bound for length of curves enclosing domains of a fixed area}. Given $a_0 >0$, there exists $\ell_0 >0$ such that if $U \subseteq \nsq$ is an open set with piecewise smooth boundary and  $\mbox{Vol}_{B_{g_n}}(U) \geq a_0$, then $\mathcal{L}(\partial U) \geq \ell_0$. 
	\item \label{con4} \textit{Uniform regularity estimates.} For every $r>0$ and every $j \in \mathbb{N}$, there are constants $C_r , C_{r,j} >0$ such that for every $f: \nsq \ra \re$ satisfying $\Delta_{\gn} f+f =0$ and every $p \in \nsq$ satisfying $d(p, \partial X_n)>r/\mmin$, the following estimates hold
	\begin{align*}
	|f(p)|^2 & \leq C_r \int_{B(p,r)} |f(x)|^2 dx,\\
	|\nabla^{j} f(p)|^2 & \leq C_{r,j} \int_{B(p,r)} |f(x)|^2 dx.
	\end{align*}
	\item \label{con5}  \textit{Uniform lower bound for the volume of a nodal domain.} There is $A_0 >0$ such that for every $n \in \nat$ and for every $f:\nsq \ra \re$ satisfying $\Delta_{\gn} f+ f=0$, we have
	\begin{align*}
	\mbox{Vol of every bounded nodal domain of $f$} \geq A_0.
	\end{align*} 
	\textit{Justification.} Let $\Omega \subset \nsq$ be a bounded nodal domain of $f$. Then we have
	\begin{align*}
	1 = \lambda_{1,g_n}(\Omega) &= \frac{\int_{\Omega} |\nabla_{g_n}f (x)|^2 dV_{g_n}(x)}{\int_{\Omega} |f(x)|^2 dV_{g_n}(x)} \geq \kappa \frac{\int_{\Omega} |\nabla f (x)|^2 dx}{\int_{\Omega} |f(x)|^2 dx}, \\
	&\geq \kappa \lambda_{1,g}(\Omega) \geq \kappa \lambda_{1,g}(B),
	\end{align*}
	where $\kappa$ is a constant depending only on $\mmin, \mpl$ and $B$ is the Euclidean ball with $\mbox{Vol}_{g}(\Omega) = \mbox{Vol}_{g}(B)$. Since $\lambda_{1,g}(B) \leq 1/\kappa$, it follows that
	\begin{align*}
	\mbox{Vol}_{g}(\Omega) = \mbox{Vol}_{g}(B) \geq \pi \kappa \lambda_{1,g}(B_{g}(0,1)) =: A_0.
	\end{align*}
\end{enumerate}
\end{Consequences}

\begin{thm} \label{commonthm} Suppose that $F_n$, $g_n$ and $R$ satisfy the assumptions \ref{as1}--\ref{as6}. Then there are constants $c,C>0$ such that for every $\ep >0$, the following holds for every $n \in \nat$ and for $n_R(\cdot) = N_R(\cdot)$ and $N_{R}(\cdot,\mathscr{T})$ 
	\begin{align*}
	\p\left( \left|\frac{n_{R}(F_n)}{\textup{Vol}_{g_n} [B_{g_n}(R)]} - \emph{Median}\left(\frac{n_{R}(F_n)}{\textup{Vol}_{g_n} [B_{g_n}(R)]}\right)\right| \geq \ep \right) \leq Ce^{-c\ep^{16}/\psi_{n}(R)}.
	\end{align*}
\end{thm}

\begin{proof}
	\vskip .3cm
	The proof involves the following steps.
	
	\subsection*{Step 1: Probabilistic bound on the $L^2$ norm of $F_n$.}
	
	\begin{Lemma}\label{lemintg}
	There are constants $m_0, b_1 >0$ such that for every $L>0$, the following holds with probability greater than $ 1- e^{-b_1/\psi_n(L)}$
	\begin{align*}
	\int_{\bgrl} F_{n}^{2}(z)~dz \leq m_0 L^2. 
	\end{align*}
	\end{Lemma}
	\begin{proof}
		It follows from \ref{as3}, \ref{con1} and \ref{con2} that there is $m>0$ such that for every $n \in \nat$,
		\begin{align*}
\e \left[	\int_{\bgrl} F_{n}^{2}(z)~dV_{g_n}(z)\right] = \int_{\bgrl} \e[F_{n}^2(z)]~ dV_{g_n}(z) = \mbox{Vol}_{g_n}(\bgrl) \leq m L^2.
		\end{align*}
		 Consider the event $\mathcal{E} :=\{\int_{\bgrl} F_{n}^{2}(z)~dV_{g_n}(z) < 3mL^2\}$. Markov's inequality gives
		\begin{align*}
		\p \left(\int_{\bgrl} F_{n}^{2}(z)~dV_{g_n}(z) \geq 3mL^2 \right) \leq \frac{\e \left[\int_{\bgrl} F_{n}^{2}(z)~dV_{g_n}(z)\right]}{3 m L^2} \leq  \frac{mL^2}{3m L^2} =\frac{1}{3},
		\end{align*}
		and hence $\p(\mathcal{E}) \geq 2/3$. Let $U_n$ be the unit ball in the Cameron-Martin space $\mathcal{H}_n$ and $t >0$. It now follows from Result \ref{thmisop} that $\p(\mathcal{E}+tU_n) \geq 1- e^{-t^2/2}$. Every element of $\mathcal{E}+tU_{n}$ is of the form $f+th$, with $f \in \mathcal{E}$ and $h \in U_n$. 
		\begin{align*}
		\int_{\bgrl} (f+th)^2 (z)~ dV_{g_n}(z) & \leq 2 \left( \int_{\bgrl} f^2(z)~ dV_{g_n}(z) + t^2 \int_{\bgrl} h^2(z)~ dV_{g_n}(z) \right), \\
		& \leq_{\text{\ref{as5}}} 2 \left[3mL^2 + t^2 \cdot  L^2 \psigl\right], \\
		& = 8mL^2,~\mbox{for $t=\sqrt{m/\psi_n(L)}$}. 
		\end{align*}
		Hence $\{\int_{\bgrl} F_{n}^{2}(z)~dV_{g_n}(z) > 8mL^2\} \subseteq \left(\mathcal{E}+\sqrt{m/\psi_n(L)} U_n \right)^c$. Hence we conclude that 
		\begin{align*}
		\p \left[ \int_{\bgrl} F_{n}^{2}(z)~dV_{g_n}(z) > 8mL^2 \right]  \leq \p \left[\left(\mathcal{E}+\sqrt{\frac{m}{\psi_{n}(L)}}U\right)^c \right] \leq e^{-\frac{m}{2 \psi_{n}(L)}}.
		\end{align*}
		The result now follows from \ref{as1} by letting {$m_0 := 8m/\mmin^2$}.
	\end{proof}
\noindent \textit{Notation.} We shall henceforth denote by $\mathcal{E}_L$ the event $\mathcal{E}_L := \{\int_{\bgrl} F_{n}^{2}(z)~dz \leq m_0 L^2\}$.\\

	\subsection*{Step 2: Construction of the exceptional set $E$.} \hfill

	 Fix $\epsilon>0$. Let $\alpha$, $\beta$, $\delta$, $\tau >0$, $\gamma \in (0,1)$ and $r >1$. The values of  these  parameters in terms of $\ep$ is given below 
		\begin{align}
	\alpha & \simeq \epsilon^{25/4},~\beta\simeq \epsilon^{75/32},~ \gamma\simeq \ep^{125/32},~ \delta \simeq \ep^{50/16},~r\simeq\ep^{-51/48},~\tau \simeq \ep^{125/16},\label{param}
	\end{align}
	where by $X \simeq Y$ we mean $X = C Y$, for some constant $C>0$.
	Cover $\bgr$ by a collection of Euclidean discs $\mathscr{C}_1 = \{D_j\}_{j \in \mathcal{J}_1}$, where $D_j = B(p_j,r)$ and each $p_j \in \bgr$. The collection  $\mathscr{C}_1$ is such that every $D_j$ intersects at most 10  other discs in $\mathscr{C}_1$. For a function $f$ satisfying $\Delta_{\gn} f+f=0$, we make the following definitions. A disc $D_j \in \mathscr{C}_1$ is called an \textit{unstable disc} of $f$ if there exists $y_j \in 3D_j$ such that both $|f(y_j)| < \alpha$ and $|\nabla f(y_j)|< \beta$, otherwise it is called a \textit{stable disc} of $f$. Such a point $y_j$ is called an \textit{unstable point} of $f$. We call $f$ an \textit{unstable function} if the number of unstable discs exceeds $\delta R^2$.  Let $E$ be the collection of unstable functions, we show below that $E$ is indeed exceptional. 
	\begin{Claim} There is $b_2 >0$ such that $\p(E) \leq e^{-b_2 \tau^2/\psig}$. 
		\end{Claim}
	\begin{proof}
	If $f\in E$, then a sizeable proportion of its unstable discs are well-separated. That is, there is  $c_1 >0$ such that for every $f \in E$, there is a subset $\mathcal{J}_2 \subseteq \mathcal{J}_1$ such that $|\mathcal{J}_2| \geq c_1\delta R^2$, every disc in the collection $\mathscr{C}_2 := \{D_j\}_{j \in \mathcal{J}_2}$ is an unstable disc of $f$ and  {$d(3D_{j}, 3D_{j'}) \geq 10$} for different indices $j,j' \in \mathcal{J}_2$.

	  For $f \in E \cap \mathcal{E}_R$ and $j \in \mathcal{J}_2$, let $y_j \in 3D_j$ be an unstable point of $f$ and define $M_j$ by  $M_j := \mbox{max}_{B(y_j,\gamma)}|\nabla^{2} f|$. Since $f \in \mathcal{E}_R$ and for different indices $j \in \mathcal{J}_2$, the discs $B(y_j,\ga+1)$ are disjoint, it follows from \ref{con4} that 
	\begin{align*}
	M_j ^2 &\leq C_{1,2} \int_{B(y_j, \gamma+1)} f^2(z)~ dz,\\
	\sum_{j \in \mathcal{J}_2} M_j ^2 & \leq C_{1,2} \int_{\bgr} f^2(z)~ dz \lesssim R^2.
	\end{align*}
	Hence in at least half of the discs in $\mathscr{C}_2$ (denote this collection by $\mathscr{C}_3 = \{D_j\}_{j \in \mathcal{J}_3}$), we have 
	\begin{align*}
	M_j ^2 & \lesssim \frac{1}{|\mathcal{J}_2|}R^2  =: a_{1}^{2} ~\delta^{-1}, 
	\end{align*}
	and hence $M_j  \leq a_1 \delta^{-1/2}.$ For $j \in \mathcal{J}_3$ and $v \in B(\ga)$,  Taylor expansion of $f$ at $y_j$ gives 
	\begin{equation*}
	\begin{gathered}
	|f(y_j + v)| \leq |f(y_j)| + |\nabla f(y_j)| \gamma + M_j \frac{\ga^2}{2} \leq (\alpha + \beta \ga + a_1 \delta^{-1/2} \ga^2) =: \tilde{\alpha},\\
	|\nabla f(\yj+v)|  \leq |\nabla f(\yj)| + M_j \gamma \leq (\beta + a_1 \delta^{-1/2}\gamma) =: \tilde{\beta}.
	\end{gathered}
	\end{equation*}
For $f: X_n \ra \re$ satisfying $\Delta_{\gn}f+f =0$, define
	\begin{align*}
	\mathcal{V}(f):= \mbox{Vol}\{z \in  \bgr : |f(z)|\leq \tilde{\alpha}~\mbox{and}~|\nabla f(z)|\leq \tilde{\beta}\}.
	\end{align*}
	The above discussion implies that if $f \in E \cap \mathcal{E}_R$,  then 
	\begin{align*}
	\mathcal{V}(f) \geq |\mathcal{J}_3| \cdot \pi \ga^2 \geq   \frac{c_1\delta R^2}{2} \cdot \pi \ga^2.
	\end{align*}
Consider $(E\cap \mathcal{E}_R) +t U_n$, we now show for an appropriate choice of $t$ that $\p((E\cap \mathcal{E}_R)+tU_n) \leq 3/4$. Result \ref{corisop} will then imply that  $\p(E\cap \mathcal{E}_R) \leq e^{-b_0 t^2}$. Let $f \in (E\cap \mathcal{E}_R)$ and $h \in \mathcal{H}_n$ with $\left\|h\right\|_{\mathcal{H}_n} \leq t$, then for every $j \in \mathcal{J}_3$ and for $\ell =0$ and $1$, it follows from \ref{as5} and \ref{con4} that 
	\begin{align*}
 \max_{B(\yj, \ga)} |\nabla^{\ell} h|^2 & \leq \overline{C}_1 \int_{B(\yj, \ga +1)} |h(z)|^2 ~dz,\\
 \sum_{j \in \mathcal{J}_3} \max_{B(\yj, \ga)}|\nabla^{\ell} h|^2  & \leq \overline{C}_1 \int_{\bgr} |h(z)|^2~ dz \lesssim t^2 R^2 \psig,
	\end{align*}  
	where $\overline{C}_1 = \max\{C_1, C_{1,1}\}$.
 Hence for least a quarter of the $j \in \mathcal{J}_3$ (denote this collection of indices by $\mathcal{J}_4$, note that $|\mathcal{J}_4|\geq c_1 \delta R^2/8$)  and for $\ell =0$ and $1$, we have
 \begin{equation}\label{eqearlier}
	\begin{aligned}
	 \max_{B(\yj, \ga)} |\nabla^{\ell} h|^2 & \lesssim \frac{t^2 R^2 \psig}{|\mathcal{J}_3|}  =:a_{2}^{2} ~t^2  \psig~ \delta^{-1} , \\
\max_{B(\yj, \ga)} |\nabla^{\ell} h| & \leq a_2~ t \sqrt{\psig}~ \delta^{-\frac{1}{2}}.
	\end{aligned}
	\end{equation}
	If we choose $t = \tau/ \sqrt{\psig}$, then the above inequality becomes
	\begin{align*}
 \max_{B(\yj, \ga)} |\nabla^{\ell} h| & \leq a_2 \tau \delta^{-1/2}.
	\end{align*}
Let $a_3 := \max\{a_1,a_2\}$, then for every $j \in \mathcal{J}_4$ and every $v \in B(\ga)$ we have
	\begin{align*}
	|(f+h)(y_j + v)|&\leq (\alpha + \beta \ga + a_3 \delta^{-1/2} (\ga^2 + \tau)) =: \alpha',\\
	|\nabla (f+h)(\yj+v)| & \leq  (\beta + a_3 \delta^{-1/2}(\gamma + \tau)) =: \beta'.
	\end{align*}
	Now define $\widetilde{\mathcal{V}}$ by
	\begin{align*}
	\widetilde{\mathcal{V}}(f):= \mbox{Vol}\{z \in \bgr : |f(z)|\leq \alpha' ~\mbox{and}~|\nabla f(z)|\leq \beta'\}.
	\end{align*}
Then on $((E \cap \mathcal{E}_R) +(\tau/\sqrt{\psig}) U_n)$, we have
\begin{align*}
\widetilde{\mathcal{V}}  \geq \pi \ga^2 \cdot |\mathcal{J}_4| \geq \pi \ga^2 \cdot \frac{c_1 \delta R^2}{8},
\end{align*}
 and hence we have
	\begin{align*}
	\p\lb(E \cap \mathcal{E}_R) +\frac{\tau}{\sqrt{\psig}} U_n\rb &\leq \p \left(\widetilde{\mathcal{V}} \geq \frac{c_1 \delta R^2}{8} \cdot \pi \ga^2 \right) \leq \frac{8}{c_1 \pi \delta R^2 \ga^2} \cdot \e[\widetilde{\mathcal{V}}] , \\
	& \lesssim \frac{1}{\delta R^2 \ga^2} \int_{\bgr} \p(|F_n(z)| \leq \alpha' , |\nabla F_n(z)|\leq \beta' )~ dz,\\
	&\lesssim_{(A4')} \frac{1}{\delta R^2 \ga^2}   \cdot  \alpha'\beta'^2 R^2 =: a_4 \frac{\alpha'\beta'^2}{\delta \ga^2}. 
	\end{align*}
	Our choice of  parameters in \eqref{param} is such that the r.h.s. of the above inequality is less than $3/4$. That is,
	\begin{align} \label{Con1}
	(\alpha + \beta \ga + a_3 \delta^{-1/2} (\ga^2 + \tau)) (\beta + a_3 \delta^{-1/2}(\gamma + \tau))^2 \leq \frac{3}{4a_4} \delta \ga^2. 
	\end{align}
	 It now follows from the generalised isoperimetric inequality and its consequence, namely Results \ref{thmisop} and \ref{corisop} that $\p(E \cap \mathcal{E}_R) \leq e^{-b_0 \tau^2/\psig}$ and hence  Lemma \ref{lemintg} implies that
	 
	\begin{align*}
	\p(E) \leq \p(E \cap \mathcal{E}_R) + \p(\mathcal{E}_{R}^{c})  \leq e^{-b_0 \tau^2/\psig} + e^{-b_1/\psig}.
	\end{align*}
	Since $\tau = \ep^{125/16}$, for $\ep$ small enough we have $\p(E) \leq e^{-b_0 \tau^2/2\psig}$. Taking $b_2 := b_0/2$ establishes our claim. 
\end{proof}

\subsection*{Step 3: A probabilitistic bound on the nodal length.}

\begin{Lemma} \label{lengthlemma}  There are constants $a_{8},~ b_3 >0$ such that for every $ \deltas >0$ and every $n\in \nat$, the following holds with probability greater than $1 - e^{-b_3 \deltas^{2}/\psig}$: there exists a (random) subset $\mathscr{D} \subseteq \bgr$ which is a union of at most $\deltas R^2$ many Euclidean unit discs such that
	\begin{align*}
	\mathcal{L}(\mathcal{Z}(F_n) \cap (\bgr \setminus \mathscr{D})) \leq \frac{a_{8}}{\deltas^{1/50}}R^2.
	\end{align*}
\end{Lemma}

\begin{proof}
We first construct the exceptional set $E^{*}$	satisfying $\p(E^{*}) \leq e^{-b_3 \deltas^2/\psig}$ and which is such that if $f \notin \ee$, an upper bound for its nodal length as claimed in the statement of the lemma holds. 
\vskip .3cm
\noin \textbf{Construction of $\ee$.}
\vskip .2cm
\noin Let $U \subseteq \re^2$ be an open set, $f: U \ra \re$ a smooth function and $p \in U$. For $\ell = 1$ and $2$, define  $d_{\ell}(f,p)$  by
\begin{equation}\label{ddef}
\begin{aligned}
d_{\ell}(f,p) &:= \max_{1 \leq j \leq 105} |\partial_{\ell}^{j}f(p)|.
\end{aligned}
\end{equation}
 Cover $\bgr$ by a collection of Euclidean discs $\col_1 = \{B_i\}_{i \in \mi_1}$, where $B_i = B(q_i,1)$ and $q_i \in \bgr$. We choose $\col_1$  such that for every $B_i$, there are at most 10 other discs in $\col_1$ which intersect $B_i$ non trivially.  Fix $\deltas \in (0,1)$ and let $A$,  $\gamma_{*}$, $\tau_{*} >0$ be parameters whose dependence on $\deltas$ is given by
\begin{align*}
A \simeq \deltas^{1/50},~\ga_{*} \simeq \deltas^{13/25},~\tau_{*} \simeq \deltas^{13/25}. 
\end{align*}
The exact values of these parameters will be chosen towards the end of the proof. For a  function $f: X_n \ra \re$ satisfying $\Delta_{\gn} f+f=0$,  a disc $B_i \in \col_1$ is called an \textit{unstable$^{*}$ disc} of $f$ if there is a point $p_i \in 2B_i$ such that 
\begin{align*}
d_1 (f, p_i) \wedge d_2(f, p_i) \leq A.
\end{align*}
Call such a point an \textit{\uns  point} of $f$. Call $f$ an \textit{\uns function} if the number of \uns discs of $f$ exceeds $ \deltas R^2$. The collection of \uns functions is denoted by $E^*$. Let $\col_2 = \{B_i\}_{i \in \mi_2}$ be the collection of \uns discs of $f$. There is a constant $c_2 >0$ such that whenever $f \in \ee$, there is a subset $\mathcal{I}_3 \subseteq \mathcal{I}_2$ such that $|\mathcal{I}_3| \geq c_2 \deltas R^2$ and for distinct $i,i' \in \mathcal{I}_3$, we have $d(3B_i, 3B_{i'}) \geq 5$. For $f \in \ee \cap \mathcal{E}_R$, using the regularity estimates \ref{con4} in each of the discs in $\col_3$, we have for every $\ell \leq 105$,
\begin{align*}
\max_{3B_i} |\nabla^{\ell} f|^2 & \leq \widetilde{C} \int_{4B_i} f^2(z)~ dz ,\\
\sum_{i \in \mi_3} \max_{3B_i} |\nabla^{\ell} f|^2 & \leq \widetilde{C} \int_{\bgr} f^2(z)~ dz \lesssim  R^2,
\end{align*}   
where $\widetilde{C}:= \max\{C_1, C_{1,\ell}: 1 \leq \ell \leq 105\}$.  Hence there is $c_3 >0$ such that  in at least $c_3 \deltas R^2$ many discs   in $\col_3$ (denote this collection by $\col_4 = \{B_i\}_{i \in \mi_4}$),  we have  for every $\ell \leq 105$
\begin{align*}
\max_{3B_i}  |\nabla^{\ell}f|^2 \lesssim \frac{1}{\deltas R^2} R^2 
=: \frac{a_{5}^{2}}{\deltas},
\end{align*}
and hence $\max_{3B_i}  |\nabla^{\ell}f| \leq a_5/\sqrt{\deltas}$. For $i \in \mi_{4}$, let $p_i \in 2B_i$ be an \uns point. Without loss of generality,  assume that $d_1 (f, p_i) \leq A$. For every $\ell \leq 104$, by Taylor expansion of $\partial_{1}^{\ell}f$ at $p_i$, we have for every $v \in B(\ga_{*})$
\begin{align}\label{eqa1}
|\partial_{1}^{\ell}f(p_i+v)| \leq |\partial_{1}^{\ell}f(p_i)| + \max_{B(p_i,\gamma_{*})}|\nabla^{\ell+1}f|\cdot \gamma _{*} \leq
A + \frac{a_5 \gamma_{*}}{\sqrt{\deltas}}.
\end{align}
If we choose $A = a_5\gamma_{*}/\sqrt{\deltas}$  in \eqref{eqa1}, then 
\begin{align}\label{texteqA}
d_1 (f,q) \wedge d_2(f,q) \leq 2A,\text{ for every $q \in B(p_i, \ga_{*})$.}
\end{align}
Define $\mathcal{V}^{*}$ as follows
\begin{align*}
\mathcal{V}^{*}(f) := \mbox{Vol}\{z \in \bgr: d_1(f,z) \wedge d_2(f, z) \leq 2A\},
\end{align*}
then for every $f \in E^*$, we have 
$$\mathcal{V}^{*}(f) \geq |\col_4|\cdot {\pi\gamma_{*}}^2 \geq c_3 \pi \deltas R^2   \gamma_{*}^2.$$
 We now show that $\p(\ee) \leq e^{-b_3 \deltas^2/\psig}$ for some $b_3 >0$. For this consider $(\ee \cap \mathcal{E}_R) +t_{*} U_n$, we now show for an appropriate choice of $t_{*}$ that $\p((\ee \cap \mathcal{E}_R) + t_{*} U_n) \leq 3/4$. 
Let $f \in \ee$ and $h \in \mathcal{H}_n$ with $\left\|h\right\|_{\mathcal{H}_n} \leq t_{*}$, then for every $i \in \mi_{4}$ and $0 \leq j \leq 105$ we have
\begin{align*}
\max_{B(p_i, \ga_{*})} |\nabla^{j} h|^2 & \leq \widetilde{C} \int_{B(p_i, \ga_{*} +1)} h(z)^2~dz,\\
 \sum_{i \in \mi_{4}} \max_{B(p_i, \ga_{*})}|\nabla^{j} h|^2 & \leq \widetilde{C} \int_{\bgr}  h(z)^2~dz \leq \widetilde{C} {t_{*}}^2 R^2 \psig.
\end{align*}
The above sums have at least $c_3 \deltas R^2$ many terms. Hence by an argument similar to the one used to get \eqref{eqearlier}, we conclude the existence of $c_4 >0$ and  $\overline{\mi} \subseteq \mi_{4}$ such that $|\overline{\mi}| \geq c_4\deltas R^2$ and for every $0\leq j \leq 105$, 
\begin{align}
 \max_{B(p_i, \ga_{*})} |\nabla^{j} h|^2  &\lesssim \frac{1}{c_4 \deltas R^2} \cdot  {t_{*}}^2 R^2 \psig =:  a_{6}^{2} {t_{*}}^2 \deltas^{-1} \psig, \nonumber \\
 \max_{B(p_i, \ga_{*})} |\nabla^{j} h| &\leq a_6 t_{*} \deltas^{-\frac{1}{2}} \sqrt{\psig}.\label{maxh}
\end{align}
If we choose $t_{*} = \tau_{*}/\sqrt{\psig}$, then \eqref{maxh} becomes
\begin{align} \label{eqder}
\max_{B_{\eg}(p_i, \ga_{*})} |\nabla^{j} h| & \leq a_6 \tau_{*} \deltas^{-1/2},\text{ for every $0 \leq j \leq 105$}.
\end{align}
 Fix $i \in \overline{\mi}$ and let $p_i$ be an \uns point of $f$ and assume without loss of generality that $d_1 (f, p_i) \leq A$, then it follows from \eqref{texteqA} and \eqref{eqder} that for every $ q \in B(p_i, \ga_{*})$ we have
 
\begin{align}\label{eqd1}
d_1 (f+h,q) \leq 2A + a_6 \tau_{*} \deltas^{-1/2}. 
\end{align}
\begin{equation} \label{texteqB}
\parbox{.7\textwidth}{Choosing $a_6 \tau_{*} = a_5 \gamma_{*}$ in \eqref{eqd1} gives $d_1 (f+h,q) \leq 3A$.}
\end{equation}
Define $\widetilde{\mathcal{V}^{*}}$ by
\begin{align*}
\widetilde{\mathcal{V}^{*}}(f):= \mbox{Vol}\{z \in \bgr : d_1(f,z) \wedge d_2(f,z) \leq 3A\}.
\end{align*}
Then on $((\ee \cap \mathcal{E}_R) +\tau_{*} {[\psig]}^{-1/2} U_n)$, we have $\widetilde{\mathcal{V}^{*}} \geq c_4\deltas R^2 \cdot \pi{ \ga_{*}}^2 $. So we have,
\begin{align}
\p((\ee \cap \mathcal{E}_R) +\tau_{*} {[\psig]}^{-1/2} U_n) &\leq \p \left(\widetilde{\mathcal{V}^{*}} \geq c_4 \pi \deltas R^2 {\ga_{*}}^2 \right) \leq \frac{\e[\widetilde{\mathcal{V}^{*}}]}{c_4\pi\deltas R^2 {\ga_{*}}^2}, \nono \\
&\lesssim \frac{1}{\deltas R^2 {\ga_{*}}^2} \int_{\bgr} \p(d_1(z) \wedge d_2(z) \leq 3A)~ dz, \nono\\
&\lesssim \frac{1}{\deltas R^2 {\ga_{*}}^2} \int_{\bgr} \p(d_1 (z) \leq 3A)+ \p(d_2(z) \leq 3A) ~ dz , \nono\\
& \lesssim_{\text{\ref{as4}}} \frac{R^2}{\deltas {\ga_{*}}^2 R^2}\cdot  A^{105} =: \frac{a_{7}^{2}}{\deltas {\ga_{*}}^2}   A^{105} = a_{7}^{2} A^3 \frac{A^{100}}{\deltas^2}.  \label{eqabove}
\end{align}
If we choose $\deltas = 2a_7 A^{50}$, then the r.h.s. of \eqref{eqabove} is less than $3/4$. It now follows from Results \ref{thmisop} and \ref{corisop} that
\begin{align*}
 \p(\ee \cap \mathcal{E}_R) \leq e^{-b_0 {\tau^{*}}^2/\psig}.
\end{align*}
Recall that the parameters $A, \ga_{*}, \tau_{*}$ were chosen to satisfy the following relations
\begin{align*}
a_6 \tau_{*}=\ga_{*},~\deltas = 2a_7 A^{50},~A=a_5 \frac{\ga_{*}}{\sqrt{\deltas}}.
\end{align*}
Hence ${\tau_{*}}^2 = ((2a_7)^{1/25}(a_6 a_5)^2)^{-1} \deltas^{1+\frac{1}{25}}$ and thus we have $\p(\ee) \leq e^{-b_3 \deltas^2/\psig}$ for some $b_3 >0$.\\

\noin\textbf{Estimating nodal length for stable$^{*}$ functions outside of its \uns set.}
\vskip .2cm

\noin We start with a few definitions. A function $f: X_n \ra \re$ satisfying $\Delta_{\gn} f+f =0$ is called a \textit{stable$^{*}$ function} if $f \notin \ee$. The \textit{\uns set} $\mathscr{D}$ of $f$ is defined by
\begin{align*}
\mathscr{D} := \bigcup_{i \in \mi_1}\{B_i: B_i~\mbox{is an \uns disc of $f$}\}.
\end{align*} 	
Let $f \in (\ee)^c \cap \mathcal{E}_R$. For  $i \in \mi_1$ and $s \leq 105$, we define $M_{i}^{s}$ by
\begin{align*}
M_{i}^{s} := \max_{2B_i}|\nabla^{s} f|.
\end{align*}
It follows from the regularity estimates \ref{con4} that for every $ s \leq 105$ and for every $ i \in \mi_1$
\begin{align*}
{M_i^{s}}^2 \leq \widetilde{C} \int_{3B_i} f^2(z)~ dz.
\end{align*}
Since every $B_i \in \col_1$ intersects at most 10 other discs in $\col_1$, we have for every $s \leq 105$

\begin{align}
\sum_{i \in \mi_1} {M_i^{s}}^2  \leq 10 \widetilde{C} \int_{\bgr} f^2(z)~ dz \lesssim R^2. \label{eq300}
\end{align}
For $i\in \mi_1$, define $M_i:= \max_{1 \leq s \leq 105} M_{i}^{s}$. It now follows from \eqref{eq300} that
\begin{align*}
\sum_{i \in \mi_1} {M_i}^2  \lesssim R^2. 
\end{align*} 
Let $\widetilde{S} := \bgr \setminus \mathscr{D}$. We now get an estimate for the nodal length of  $f$ in $\widetilde{S}$. Let $B_{i}$ be a stable$^{*}$ disc of $f$, then for every $p \in B_i$ we have 
\begin{equation*}
\begin{gathered}
d_1 (f,p) \wedge d_2 (f,p) \geq A,\\
\max_{s \leq 105} \{\Vert f \Vert_{L^{\infty}(2\mathbb{D}_p)}, \Vert \nabla^{s} f\Vert_{L^{\infty}(2\mathbb{D}_p)}\} \leq M_i.
\end{gathered}
\end{equation*}
Now appealing to Corollary \ref{lengthcor}, 
  we  conclude that 
\begin{align*}
\mathcal{L}(\zf \cap B_i) \lesssim \frac{M_i}{A}.
\end{align*}
Recall that each disc $B_i$ has radius $1$ and intersects at most 10 other discs in $\col_1$. Hence $|\mi_1| \lesssim  R^2$ and we have
\begin{align}
\mathcal{L}(\mathcal{Z}(f) \cap \widetilde{S}) & \leq \sum_{i:B_i~\mbox{\scriptsize{stable$^{*}$disc}}} \mathcal{L}(\zf \cap B_i), \nonumber \\
& \lesssim \frac{1}{A} \sum_{i:B_i~\mbox{\scriptsize{stable$^{*}$disc}}} M_i \lesssim \frac{1}{A} \sum_{i \in \mi_1} M_i, \nonumber \\
& \lesssim \frac{1}{A} \sqrt{\sum_{i \in \mi_1} M_{i}^{2}} \cdot \sqrt{|\mi_1|}, \nonumber \\
& \lesssim \frac{1}{A} \sqrt{R^2} \sqrt{R^2} =: a_{8} \frac{R^2}{A} = a_{8} \frac{R^2}{\deltas^{1/50}}.  \nonumber
\end{align}
\end{proof}
\subsection*{Step 4: Concluding the concentration result.} \hfill
\vskip .3cm
\noin For $n_R(\cdot) = N_R(\cdot)$ and $N_R(\cdot, \mathscr{T})$, we prove concentration of $n_R (F_n)/\text{Vol}_{\gn}[\bgr]$ around its median using  Lemma \ref{gaussianconct}. The role of $\E$ in Lemma \ref{gaussianconct} is played by $(E \cup E^{*} \cup \mathcal{E}_{R}^{c})$, with $\deltas = \delta$ in the definition of $\ee$. We show that for every $\epsilon >0$, there is $ \rho >0$ such that whenever $f \in E^{c} \cap {E^{*}}^c \cap \mathcal{E}_R$ and $h \in \mathcal{H}_n$  with $\Vert h \Vert_{\mathcal{H}_n} \leq \rho/\sqrt{\psig}$, we have
\begin{align} \label{lsc}
\frac{n_R (f+h)}{\text{Vol}_{\gn}[\bgr]} - \frac{n_R (f)}{\text{Vol}_{\gn}[\bgr]} \geq -\epsilon.
\end{align}
It will then follow from Lemma \ref{gaussianconct} that 
\begin{align}\label{eqq}
\p\left(\left| \frac{n_R (f)}{\text{Vol}_{\gn}[\bgr]} - \textup{Med} \left( \frac{n_R (f)}{\text{Vol}_{\gn}[\bgr]}\right)\right| > \ep \right) \leq \p(E \cup E^{*} \cup \mathcal{E}_{R}^{c}) + e^{-\tilde{c}\rho^2/\psig}.
\end{align}
\vskip .2cm
Fix $f \in (E^{c} \cap {E^{*}}^c \cap \mathcal{E}_R)$ and let $\mathcal{C}_R (f)$ denote the collection of nodal components of $f$ which are contained entirely in $\bgr$ and let $\mathcal{C}_R (f,\mathscr{T})$ be the collection of all those nodal components in $\mathcal{C}_R (f)$ whose tree end is $\mathscr{T}$. Let $\mathcal{U} := \{\mbox{unstable discs}\} \cup \{\mbox{\uns discs}\}$ and $S := \bgr \setminus  \U$.  Define the following collections of nodal components
\begin{align*}
\mathcal{C}_{R,1}(f) &:= \{\Gamma \in  \mathcal{C}_R (f): \Gamma \cap \U = \phi~\text{and diam}(\Gamma) \leq r\},
\end{align*}
\begin{align*}
\mathcal{C}_{R,2}(f) &:= \{\Gamma \in  \mathcal{C}_R (f): \Gamma \cap \U = \phi~\text{and diam}(\Gamma) > r\},\\
\mathcal{C}_{R,3}(f) &:= \{\Gamma \in  \mathcal{C}_R (f): \Gamma \cap \U \neq \phi~\text{and }\mathcal{L}(\Gamma|_{S}) \leq  r\},\\
\mathcal{C}_{R,4}(f) &:= \{\Gamma \in  \mathcal{C}_R (f): \Gamma \cap \U \neq \phi~\text{and }\mathcal{L}(\Gamma|_{S}) >  r\}.
\end{align*}
For $i \leq 4$, define $\mathcal{C}_{R,i}(f,\mathscr{T}) := \mathcal{C}_{R,i}(f) \cap \mathcal{C}_{R}(f,\mathscr{T})$.
We can write $N_{R}(f)$ and $N_{R}(f,\mathscr{T})$ as the following sums
\begin{align*}
N_R(f) = \sum_{i=1}^{4} \sharp~\mathcal{C}_{R,i}(f),~N_R(f,\mathscr{T}) = \sum_{i=1}^{4} \sharp~ \mathcal{C}_{R,i}(f,\mathscr{T}).
\end{align*}
We now show that $\sharp~\mathcal{C}_{R,2}(f),~\sharp~ \mathcal{C}_{R,3}(f) ~\text{and} ~\sharp~ \mathcal{C}_{R,4}(f)$ are all negligible in comparison to $N_{R}(f)$ and  then use Result \ref{NDcount} to conclude that if $h \in \mathcal{H}_n$ has sufficiently small Cameron-Martin norm, then 
\begin{align*}
 \text{$N_R(f+h) \geq \sharp~ \mathcal{C}_{R,1}(f) - \ep R^2$ and $ N_{R}(f+h,\mathscr{T}) \geq \sharp~ \mathcal{C}_{R,1}(f,\mathscr{T}) - \ep R^2. $}
 \end{align*}
Let  $\Gamma \in \mathcal{C}_{R,2}(f) \cup \mathcal{C}_{R,4}(f)$, then necessarily $\mathcal{L}( \Gamma|_{S}) \geq r$.  It now follows from the nodal length estimate in Lemma \ref{lengthlemma} and its proof that 
	\begin{align} \label{cons1}
	\sharp~\mathcal{C}_{R,2}(f) + \sharp~\mathcal{C}_{R,4}(f) \leq \frac{a_{8} }{r \delta^{1/50}} R^2.
	\end{align}
 If $\Gamma \in \mathcal{C}_{R,3}(f)$, then we have
	\begin{align*}
	\Gamma \subseteq \cup \{ \mbox{(unstable discs)$_{+r}$} \cup  \mbox{(\uns discs)$_{+r}$}\}. 
	\end{align*}
Note that (\uns disc)$_{+r}$ and (unstable disc)$_{+r}$ are discs of radius $r+1$  and $2r$ respectively. Hence we have 
	\begin{align*}
	\mbox{Vol}(\cup \{\mbox{(unstable disc)$_{+r}$}\cup  \mbox{(\uns discs)$_{+r}$} \}) \leq 8\delta r^2 R^2.
	\end{align*}
	So it follows from \ref{con5} that 
	\begin{align}\label{cons3}
\sharp~	\mathcal{C}_{R,3}(f) \leq \frac{8\delta r^2}{A_0} R^2.
	\end{align}
 We now compare $\sharp~ \mathcal{C}_{R,1}(f)$ with $ N_{R}(f+h)$ and $\sharp~ \mathcal{C}_{R,1}(f,\mathscr{T})$ with $ N_{R}(f+h,\mathscr{T})$ . Suppose $h \in \mathcal{H}_n$ is such that $\left\|h\right\|_{\mathcal{H}_n} \leq \rho/\sqrt{\psig}$. Then we have from \ref{as5} that
	\begin{align*}
	\int_{\bgr} h^2(z)~ dz \leq \rho^2 R^2.
	\end{align*}
	By the regularity estimates \ref{con4} we have for every $ j \in \mathcal{J}_1$,
	\begin{align*}
	\max_{3D_j} h^2,~  	\max_{3D_j} |\nabla h|^2\leq \overline{C_1} \int_{B(p_j, 3r+1)} h^2(z)~ dz.
	\end{align*}
	 Since the collection of discs $\{D_j\}_{j \in \mathcal{J}_1}$  is such that every $D_j$ is of radius $r$ and intersects at most 10 other discs in the collection, there is $N_0 \in \mathbb{N}$ such that for every $j \in \mathcal{J}_1$, there are at most $N_0$ many indices $j' \in \mathcal{J}_1$ for which $B(p_j, 3r+1) \cap B(p_{j'}, 3r+1) \ne \phi$. Hence we have
	 	\begin{align*}
	\sum_{j \in \mathcal{J}_1} \max_{3D_j} h^2,~\sum_{j \in \mathcal{J}_1} \max_{3D_j} |\nabla h|^2 &\leq C_1 \sum_{j \in \mathcal{J}_1} \int_{B(q_j, 3r+1)} h^2(z)~dz,\\
	& \leq C_1 N_0 \int_{\bgr} h^2(z)~ dz, \\
	& \leq (C_1 N_0) \rho^2 R^2 \lesssim \rho^2 R^2.
	\end{align*}
Let $\mathcal{J}_0 \subseteq \mathcal{J}_1$ be the collection of indices for which either  $\max_{3D_j} |h| \geq \alpha/2$ or $\max_{3D_j} |\nabla h| \geq \beta/2$. Then $|\mathcal{J}_0|  \lesssim \rho^2R^2( \alpha^{-2} +  \beta^{-2})$ and hence
\begin{equation*}
\begin{gathered}
\text{Vol}\left(\cup_{j \in \mathcal{J}_0} 3D_j\right)  \lesssim \rho^2r^2 R^2 (\alpha^{-2} + \beta^{-2}),\\
\sharp \{\Gamma \in \mathcal{C}_{R,1}(f) : \Gamma \subset \left(\cup_{j \in \mathcal{J}_0} 3D_j\right)\}  \leq_{\text{\ref{con5}}} a_{9}\rho^2r^2 R^2 (\alpha^{-2} + \beta^{-2}).
\end{gathered}
\end{equation*}
We choose parameters so that 
	\begin{align}\label{Cons4}
 a_9 \rho^2 r^2(\alpha^{-2} + \beta^{-2}) \lesssim \epsilon,
	\end{align}
	for some $a_9 >0$. We make the following conclusions using Result \ref{NDcount}.
	\begin{itemize}[ align=left,leftmargin=*,widest={8}]
\item Let $\Gamma \in \mathcal{C}_{R,1}(f)$, then it necessarily intersects a stable disc $D_j$ and $\Gamma \subset 2D_j$.  Suppose that $|h| < \alpha/2$, $|\nabla h| \leq \beta/2$ on $3D_j$, $3D_j \subset \bgr$ and $\alpha/\beta < r$. Then it follows from Result \ref{NDcount} that corresponding to $\Gamma$, there is a nodal component $\widetilde{\Gamma} \in \mathcal{C}_{R}(f+h)$ and different such nodal components of $f$ correspond to different nodal components of $f+h$ and hence
\begin{align}\label{eq900}
N_{R}(f+h) \geq \sharp~\mathcal{C}_{R,1}(f) - a_{9} \rho^2 r^2 R^2(\alpha^{-2} + \beta^{-2}) - a_{10} \frac{R}{r}, 
\end{align}
where $a_{10}R/r$ is an upper bound for the number of discs $D_j$ for which $3D_j \nsubseteq \bgr$.
\vskip .2cm
\item 
In the above argument, suppose that $\Gamma \in \mathcal{C}_{R,1}(f,\mathscr{T})$. Then for $\Gamma$ and every other nodal component $\Gamma_i$ of $f$ which  lies in $\text{Int}(\Gamma)$, their corresponding shells (components of $|f| < \alpha$) $S_{\Gamma}, S_{\Gamma_i}$ are all disjoint and there is a unique nodal component of $f+h$ in each of these shells and no other nodal component of $f+h$ in Int$(\Gamma$). Also  $\widetilde{\Gamma}_i $ is homotopic to $\Gamma_i$ in $S_{\Gamma_i}$ since otherwise $\Gamma_i$ will be contractible in $S_{\Gamma_i}$ and this will imply the existence of a point in Int$(\Gamma_i)$ where $\nabla(f+h) = 0$, which is not possible since $|\nabla(f+h)| \geq \beta/2$ in $S_{\Gamma_i}$. This implies that the tree end of $\widetilde{\Gamma}$ is also $\mathscr{T}$ (See Figure \ref{fig:fig1} for a pictorial representation of this argument). Hence we conclude that 
\begin{align}\label{eq901}
N_R (f+h,\mathscr{T}) \geq \sharp~\mathcal{C}_{R,1}(f,\mathscr{T}) - a_{9} \rho^2 r^2 R^2(\alpha^{-2} + \beta^{-2}) -a_{10} \frac{R}{r}. 
\end{align}
\end{itemize}
If the parameters are chosen so that the r.h.s. of \eqref{cons1}, \eqref{cons3} and the l.h.s. in \eqref{Cons4} are all $\lesssim \epsilon$, then we can conclude that \eqref{lsc} holds. We now show that we can indeed make a consistent choice of parameters so that all the required conditions are met. 
	\vskip .2cm
	  
	 \textbf{Choosing the parameters.} We jot down all the constraints on the parameters. 
	\begin{enumerate}
		\item $(\alpha + \beta \ga + a_3 \delta^{-1/2} (\ga^2 + \tau)) (\beta + a_3 \delta^{-1/2}(\gamma + \tau))^2 \leq  3\delta \ga^2/4a_4$,
		\vskip .2cm
		\item $1 / r \delta^{1/50} \lesssim \ep$,
			\vskip .2cm
		\item $\delta r^2  \lesssim \ep$,
		\vskip .2cm
		\item $ \rho^2 r^2(\alpha^{-2} + \beta^{-2}) \lesssim \epsilon$,
		\vskip .2cm
		\item $\alpha/\beta \leq r$. 
	\end{enumerate}
	
	\noin With the following choice of parameters, all the above constraints are satisfied 
	\begin{align*}
	\alpha &=\alpha_0~\epsilon^{25/4},~\beta = \beta_0~ \epsilon^{75/32},~ \gamma = \gamma_0~ \ep^{125/32},~ \tau = \tau_0 ~\ep^{125/16},~ \delta = \delta_0~ \ep^{50/16}\\
	r&= r_0~\ep^{-51/48},~\rho = \rho_0~ \ep^{125/16},
	\end{align*}
	where $\alpha_0$, $\beta_0$, $\gamma_0$, $\tau_0$, $\delta_0$, $r_0$ and $\rho_0$ are positive constants chosen in the following manner. First $r_0$ and $\delta_0$ are chosen so that (2) and (3) hold. With this choice of $\delta_0$, we can choose $\alpha_0$, $\beta_0$ and $\gamma_0$ satisfying $0 < \alpha_0,~\beta_0 \ll \gamma_0 \ll 1$ so that (1) holds. With the above choice, (5) holds whenever $\ep >0$ is small enough. Finally we can choose $ \rho \ll 1$ so that (4) holds. 
\vskip .1cm
	Recall from Lemma \ref{lengthlemma} that $\p(\ee) \leq e^{-b_3 \delta^{2}/\psig}$ and $\p(E) \leq e^{-b_{0}\tau^2/2\psig}$.
	Notice that
	\begin{align*}
	\min \{\tau^2, {\delta}^2, \rho^2\} = \ep^{15.625},
	\end{align*}
	and hence it follows from \eqref{eqq} that there is a constant $c>0$ satisfying
	\begin{align*}
	\p\left(\left| \frac{n_R (f)}{\text{Vol}_{\gn}[\bgr]} - \textup{Median} \left( \frac{n_R (f)}{\text{Vol}_{\gn}[\bgr]}\right)\right| > \ep \right) \leq  4e^{-c\ep^{15.625}/\psig}.
	\end{align*}
	\end{proof}
\section{Proof of Theorem \ref{thmrpw}} \label{pf1}
We now show how to deduce Theorem \ref{thmrpw} from Theorem \ref{commonthm}. Let $F_{\nu}$ be as in Theorem \ref{thmrpw}. For $n \in \nat$, define $F_n := F_{\nu}|_{[-n,n]^2}$, $g_n = g$ and {$R <n/2$}. It follows from the discussion in Appendix \ref{app1} that $F_n$ defines a Radon Gaussian measure $\ga_n$ on the locally convex space $(C[-n,n]^2, \|\cdot\|_{L^{\infty}[-n,n]^2})$ and the Cameron-Martin space $\mathcal{H}_{n}$ of $\ga_n$ is a subset of the restriction to $[-n,n]^2$ of functions in $ \mathcal{F}L^{2}_{\text{symm}}(\nu)$, where 
\begin{equation} \label{l2symm}
\begin{aligned}
L^{2}_{\text{symm}}(\nu) & := \{f: \mathbb{R}^2 \ra \mathbb{C}~|~ f \in L^2(\nu)\text{ and } f(-z) = \overline{f(z)},\text{ for every $z \in \mathbb{R}^2$}\},\\
\mathcal{F}L^2_{\text{symm}}(\nu) &:= \left\lbrace \widehat{f}: \re^2 \ra \re~|~ \widehat{f}(z) =\int_{0}^{2\pi} e^{-iz\cdot v_{\theta}} f(\theta)~ d\nu(\theta),~f \in L^2_{\text{symm}}(\nu)\right\rbrace,
\end{aligned}
\end{equation}
where $v_{\theta} := (\cos \theta, \sin \theta)$ and $\cdot$ is the standard inner product in $\re^2$. 
The norm in $\mathcal{H}_{n}$ is given by $\|\widehat{f}\|_{\mathcal{H}_{n}} = \left\|f \right\|_{L^2 (\nu)}$. 
We need  to check that the assumptions   \ref{as1}--\ref{as6}   hold for this choice of $F_n$, $g_n$ and $R$. That \ref{as1}--\ref{as3}  hold is trivial. Because of stationarity of $F_n$, \ref{as4}  and $(A4')$ follow from Lemma \ref{lemnondeg}. \ref{as5}  is established in Lemma \ref{lemnorm}. 
\ref{as6} follows from \eqref{regpr}. With $\omega_{\nu}$ as in Notation preceding Theorem \ref{thmrpw}, we prove the following.  
\begin{Lemma} \label{lemnorm} There is $C>0$ such that for every $f \in L^{2}_{\textup{symm}}(\nu)$ and $T$ large enough, we have 
	\begin{equation*}
	\|\widehat{f}\|_{L^2 (B(T)} \leq C \|\widehat{f}\|_{\mathcal{H}_n} ^{2} \cdot T^2~ \omega_{\rhom} (1/T)  = C\|f \|_{L^2(\rhom)} ^{2} \cdot T^2~ \omega_{\rhom} \left(1/T \right).
	\end{equation*}
	
\end{Lemma}
\begin{proof}  
	
	For $T>0$, consider the Gaussians and the scaled Gaussians
		\begin{align*}
		\phi(x,y) &= e^{-(x^2+y^2)},~\widehat{\phi}(x,y) = \pi e^{-(x^2+y^2)/4}, \\
		\phi_T (x,y) &= \phi(Tx,Ty),~\widehat{\phi_T}(x,y) = \frac{1}{T^2} \hat{\phi}\left(\frac{x}{T}, \frac{y}{T}\right).
		\end{align*}
			 For $f \in L^2_{\text{symm}}(\nu)$ and $T >0$, we have
			 
		\begin{align}
		\int_{\ffT} \fhat^2 dx dy &\leq e^{\frac{T^2}{2T^2}} \int_{\ffT} \fhat^2 e^{-\frac{(x^2 +y^2)}{2 T^2}} dx dy, \nonumber\\
		&= \frac{\sqrt{e}}{\pi^2} \int_{\ffT} \fhat^2  \hat{\phi}\left(\frac{x}{T},\frac{y}{T}\right) ^2 dx dy, \nono\\
		&= T^4 \frac{\sqrt{e}}{\pi^2} \int_{\ffT} ((\widehat{f * \phi_T})(x,y))^2 dx dy, \nonumber \\
		& \lesssim T^4 \int_{\re^2} ((\widehat{f * \phi_T})(x,y))^2 dx dy, \nono \\
		& = T^4 \int_{\re^2} |(f * \phi_T) (x,y)|^2 dx dy, \label{eq0003}
		\end{align}
		where the equality in \eqref{eq0003} follows by Plancherel's theorem. For $z=(x,y)$ we have
		\begin{align*}
		(f * \phi_T)(z) &= \int_{0}^{2\pi} f(\theta)\phi_T (z-v_{\theta}) d\rhom(\theta),\\
		\int_{\re^2} |f * \phi_T (z)|^2 dz & = \intc \intc f(\theta) \overline{f(\alpha)} \left[ \intr \phi_T (z-v_{\theta}) \phi_T (z-v_{\alpha}) dz \right] d\rhom(\theta) d\rhom(\alpha).
		\end{align*}
		Let $v_{\alpha}-v_{\theta} = (q_1,q_2) $, then we have
		\begin{align*}
		\intr \phi_T (z-v_{\theta}) \phi_T (z-v_{\alpha}) dz & = \intr \phi_T (z) \phi_T (z-(q_1,q_2)) dz,\\
		& = \intr e^{-T^2(x^2 +y^2)} e^{-T^2[(x-q_1)^2 + (y-q_2)^2]} dx dy,\\
		& = \lb \int_{\re} e^{-T^2(x^2 + (x-q_1)^2)} dx \rb \lb\int_{\re} e^{-T^2(y^2 + (y-q_2)^2 )}dy\rb,\\
		& = \frac{\pi}{2T^2} ~e^{\frac{-T^2}{2}\Vert v_{\theta} - v_{\alpha}\Vert^2}.
		\end{align*}
		\noin Hence we have
		\begin{align}\label{exp1}
		\int_{\re^2} |f * \phi_T (z)|^2  dz & =\frac{\pi}{2T^2} \intc \intc f(\theta) \overline{f(\alpha)}  e^{\frac{-T^2}{2}\Vert v_{\theta} - v_{\alpha}\Vert^2} d\rhom(\theta) d\rhom(\alpha). 
		\end{align}
		Define $h$ and $I_{\nu}$ by
		\begin{align*}
		h(\theta, \alpha) := e^{\frac{-T^2}{2}\Vert v_{\theta} - v_{\alpha}\Vert^2}\text{ and } I_{\rhom}(\theta) := \intc  h(\theta, \alpha)~  d\rhom(\alpha).
		\end{align*}
\noin	Using Cauchy-Schwartz inequality for the integral in the r.h.s. of \eqref{exp1}, we get
	\begin{align}
		\int_{\re^2} |f * \phi_T (z)|^2  dz & =\frac{\pi}{2T^2} \intc \intc f(\theta) \overline{f(\alpha)}~  h(\theta, \alpha) d\rhom(\theta) d\rhom(\alpha), \nono\\
		& \leq \frac{\pi}{2T^2} \intc |f(\theta)| \sqrt{I_{\rhom}(\theta)} \left( \intc |f(\alpha)|^2 ~ h(\theta, \alpha)  d\rhom(\alpha) \right)^{\frac{1}{2}}   d\rhom(\theta), \nonumber\\
		& \leq \frac{\pi}{2T^2} \left(\intc |f(\theta)|^2 I_{\rhom}(\theta) d\rhom(\theta) \right)^{1/2} \left( \intc \intc |f(\alpha)|^2 ~ h(\theta, \alpha)   d\rhom(\alpha)   d\rhom(\theta) \right)^{\frac{1}{2}}, \nonumber \\
		& = \frac{\pi}{2T^2} \left(\intc |f(\theta)|^2 I_{\rhom}(\theta) d\rhom(\theta) \right)^{1/2} \left( \intc |f(\alpha)|^2 I_{\rhom}(\alpha) d\rhom(\alpha)    \right)^{\frac{1}{2}}, \nonumber\\
		& = \frac{\pi}{2T^2} \intc |f(\theta)|^2 I_{\rhom}(\theta) d\rhom(\theta) \leq \frac{\pi}{2T^2}~ \|f\|^{2}_{L^{2}(\rhom)}~ \|I_{\rhom}\|_{L^{\infty}[0,2\pi]} . \label{eq0001}
	\end{align}	
	
	\begin{Claim} \label{claimomega}There are constants $C, T_0 >0$ such that for every $T>T_0$  and every $\rhom \in L^{2}_{\textup{symm}}(\mathbb{S}^1)$, we have 
			\begin{align*}
			\|I_{\nu}\|_{L^{\infty}[0,2\pi]} \leq C \omega_{\rhom}\left(1/T\right).
			\end{align*}
		\end{Claim}
	\begin{proof}
		We can express $I_{\nu}(\theta)$ as
		\begin{align}
		I_{\nu}(\theta) &= \int_{0}^{2\pi} e^{-\frac{T^2}{2} \|v_{\theta} - v_{\alpha}\|^2 } d\nu(\alpha) = \int_{0}^{\infty} \nu \left\{ \alpha:e^{-\frac{T^2}{2} \|v_{\theta} - v_{\alpha}\|^2 } \geq t \right\} dt, \nonumber \\
		& = \int_{e^{-2T^2}}^{1} \nu \left\{ \alpha: \left| \sin\left( \frac{\theta - \alpha}{2}\right)\right| \leq \sqrt{\frac{\log (1/t)}{2}}\frac{1}{T}\right\} dt, \nonumber\\
		& = 4 \int_{0}^{T} \nu \left\{ \alpha: \left| \sin\left( \frac{\theta - \alpha}{2}\right)\right| \leq \frac{y}{T}\right\} y e^{-2y^2} dy, \nono\\
		& = 4 \left(\int_{0}^{\sqrt{\log T}} + \int_{\sqrt{\log T}}^{T} \right) (\cdots) dy. \label{intsplit}
		\end{align}
		We  estimate the two integrals in \eqref{intsplit} separately,
		\begin{align}
		\int_{\sqrt{\log T}}^{T} \cdots~ dy & \leq  \int_{\sqrt{\log T}}^{T}  ye^{-2y^2} dy \lesssim \frac{1}{T^2}. \label{i02}
		\end{align}
		Let $T$ be large enough so that for $ |x| \in [0,1/\sqrt{T}]$, $|\sin x| \geq 2|x|/3$. Then we have
		\begin{align}
		\int_{0}^{\sqrt{\log T}} \cdots~dy &\leq  \int_{0}^{\sqrt{\log T}} \nu \left\{ \alpha: \left| \theta - \alpha \right| \leq \frac{3y}{T}\right\} y e^{-2y^2} dy, \nonumber\\
		& \leq 3~\omega_{\nu}(1/T) + \sum_{n=1}^{\lceil T \rceil} \nu\left\{\alpha: |\theta - \alpha| \leq \frac{3(n+1)}{T}\right\} ne^{-2n^2}, \nonumber\\
		& \leq 3~\omega_{\nu}(1/T) +  \omega_{\nu}(1/T) \sum_{n=1}^{\infty} 3(n+1) ne^{-2n^2} \lesssim \omega_{\nu}(1/T). \label{i01}
		\end{align}
		The claim now follows from \eqref{intsplit}, \eqref{i02} and \eqref{i01}.
		\end{proof}
	\noin	It now follows from \eqref{eq0003}, \eqref{eq0001} and Claim \ref{claimomega} that
		\begin{equation*}
		\|\widehat{f}\|_{L^2(B(T))} ^{2} \lesssim   \|f \|_{L^2(\rhom)} ^{2} \cdot T^2~ \omega_{\rhom} \left( 1/T \right).
		\end{equation*}
	\end{proof}
\begin{Lemma}\label{lemnondeg}
	Let $\nu$ be a symmetric measure on $\mathbb{S}^1$ such that supp$(\nu)$ is not a finite set. Let $G$ be the centered, stationary Gaussian process on $\re^2$ whose spectral measure is $\nu$. Then for every $n \in \mathbb{N}$ and ${\ell}=1,2$, the  Gaussian vectors 
	\begin{align*} 
	(\partial_{\ell} G(0),\partial_{{\ell}}^{2}G(0),\ldots,\partial_{{\ell}}^{n}G(0))\text{ and }(G(0), \partial_{1}G(0), \partial_{2}G(0))
	\end{align*} 
	are non-degenerate. 
\end{Lemma}

\begin{proof}
	We prove the claim for ${\ell}=1$, a similar argument works for ${\ell}=2$ as well. For $z,w \in \re^2$, the covariance function of $G$ is given by
	\begin{align*}
	\e[G(z)G(w)]&= k(z-w) = \hat{\nu}(z-w).
	\end{align*}
	We also have for  $r,s \in \mathbb{N}$ (see \eqref{calc2}) that
	\begin{align*}
	\e[\partial_{1}^{r}G(0)\cdot \partial_{1}^{s}G(0)] = \begin{cases} 
	(-1)^{(r+s)/2} (-1)^{s}  \int_{0}^{2\pi} (\cos \theta)^{r+s} d\nu(\theta),&\textup{if $r+s$ is even},\\
	0,&\textup{if $r+s$ is odd}.
	\end{cases}
	\end{align*}
The Gaussian vector $(\partial_1 G(0),\ldots,\partial_{1}^{n}G(0))$ is non-degenerate iff for every $(a_1,a_2,\ldots,a_n)\in \re^n \setminus {0}$, the zero mean Gaussian $\sum_{i=1}^{n} a_i \partial_{1}^{i}G(0)$ has non-zero variance. With $\mathscr{I}_n := \{i \in 2\nat : i \leq n\}$ and $\mathscr{S}_n := \{s \in 2\nat +1 : s \leq n\}$, we have from \eqref{calc3} and \eqref{calc4} that
	\begin{align*}
	\mbox{Var}\left(\sum_{i=1}^{n} a_i \partial_{1}^{i}G(0)\right) &= \mbox{Var} \left(\sum_{i \in \mathscr{I}_n}a_i \partial_{1}^{i}G(0) \right)   + \mbox{Var} \left(\sum_{s \in \mathscr{S}_n} a_s \partial_{1}^{s}G(0) \right), \\
	& = \Vert \sum_{i \in \mathscr{I}_n} (-1)^{i/2} a_i (\cos\theta)^{i} \Vert_{L^2(\nu)}^2   + \Vert \sum_{s \in \mathscr{S}_n} (-1)^{(s-1)/2} a_s (\cos\theta)^{s} \Vert_{L^2(\nu)}^2.
	\end{align*}
Thus if Var$\left(\sum_{i=1}^{n} a_i \partial_{1}^{i}G(0)\right) = 0$, then each of the above two polynomials in $\cos \theta$ vanish almost surely w.r.t. $\nu$ and hence for infinitely many values of $\theta$,  but this is not possible unless $(a_1,\ldots,a_n)=0$.

Stationarity of $G$ implies that $G(0)$ is independent of $\nabla G(0)$. Hence in order to prove non-degeneracy of $(G(0), \partial_{1}G(0), \partial_{2}G(0))$, it suffices to show that  $(\partial_{1}G(0), \partial_{2}G(0))$ is non-degenerate. By a calculation similar to \eqref{calc2}, we have
\begin{align*}
\e[\partial_{i}G(0) \cdot \partial_{j}G(0)] = \begin{cases}
 \int_{0}^{2\pi} \cos(\theta) \sin(\theta) d\nu(\theta),\text{ for $i=1$, $j=2$},\\
 \int_{0}^{2\pi} \cos^2(\theta) d\nu(\theta), \text{ for $i=j=1$},\\
  \int_{0}^{2\pi} \sin^2(\theta) d\nu(\theta), \text{ for $i=j=2$}.
\end{cases}
\end{align*}
Hence $(\partial_{1}G(0), \partial_{2}G(0))$ is non-degenerate iff there is a constant $c \in \re$ such that $\cos (\theta) = c \sin (\theta)$ almost surely w.r.t. $\nu$, but this is not possible since supp$(\nu)$ is not finite. 
\end{proof}

\subsection*{Concluding concentration} Since all the assumptions \ref{as1}--\ref{as6} hold, we conclude using Theorem \ref{commonthm} that $N_{R}(F_{\nu})/\pi R^2$ and $N_{R}(F_{\nu},\mathscr{T})/\pi R^2$ concentrate around their respective medians and hence also around their respective means. It was shown in Results \ref{res2} and \ref{sw} that the means of the above quantities converge to $c_{NS}(\nu)$ and $c_{NS}(\nu)\mu(\mathscr{T})$ respectively and hence we conclude concentration of $N_{R}(F_{\nu})/\pi R^2$ and $N_{R}(F_{\nu},\mathscr{T})/\pi R^2$ around $c_{NS}(\nu)$ and $c_{NS}(\nu)\mu(\mathscr{T})$ respectively.

\section{Proof of Theorem \ref{thmlowerbound}}\label{seclowerbound}
We now prove Theorem \ref{thmlowerbound}. A few facts about Bessel functions which are used in the proof below are presented in Appendix \ref{appbessel}. Let $\nu$, $F_{\nu}$ and $\psi$ be as in the statement of Theorem \ref{thmlowerbound}. It follows from the symmetry of $\nu$ that $\psi(\cdot) = \psi(\cdot+ \pi)$ on $\sone$. In what follows, we use the definitions from  \eqref{l2symm}. An orthonormal basis for $L^{2}(\nu)$ is given by 
\begin{align*}
\{ e^{in\theta}/\sqrt{\psi(\theta)}: n \in \mathbb{Z}\},
\end{align*}
and from this we get the following orthonormal basis for $L^{2}_{\text{symm}}(\nu)$
\begin{align*}
\mathcal{B} = \left\{ \frac{1}{\sqrt{\psi(\theta)}}, \frac{e^{in\theta} + (-1)^n e^{-in\theta}}{\sqrt{2} \sqrt{\psi(\theta)}}, i~ \frac{e^{in\theta} + (-1)^{n+1} e^{-in\theta}}{\sqrt{2} \sqrt{\psi(\theta)}} \right\}_{n \in \nat}.
\end{align*}
We use the basis $\mathcal{B}$ to get the following series representation of $F_{\nu}$ 
\begin{align}\label{fnuseries}
F_{\nu}(z) & = \xi_0 \left( \frac{1}{\sqrt{\psi(\theta)}} \right)^{\wedge}(z)+  \sum_{n \in \mathbb{Z} \setminus \{0\}} \xi_n \left( \frac{e^{in\theta}}{\sqrt{\psi(\theta)}}\right)^{\wedge} (z), \nonumber \\
& = \xi_0  \int_{0}^{2\pi} e^{-iz \cdot v_{\theta}} \sqrt{\psi(\theta)}~d\theta + \sum_{n \in \mathbb{Z} \setminus \{0\}}  \xi_n \int_{0}^{2\pi} e^{-iz \cdot v_{\theta}} e^{in\theta} \sqrt{\psi(\theta)}~d\theta, \nonumber\\
& =: \xi_0 f_0(z) + \sum_{n \in \mathbb{Z} \setminus \{0\}}  \xi_n  f_n(z),
\end{align}
where $\{\xi_n : n \in \nat_0\}$ are independent Gaussian random variables such that  for every $n>0$, $\xi_{n} \sim \mathbb{C}\mathcal{N}(0,1)$,  $\xi_{-n} = (-1)^n \overline{\xi_{n}}$ and $\xi_0 \sim \mathcal{N}(0,1)$. Let $\{a_k\}_{k \in \mathbb{Z}}$ be the Fourier coefficients of $\sqrt{\psi}$ and hence
\begin{align}\label{fs}
\sqrt{\psi (\theta)} = \sum_{k =-\infty}^{\infty} a_k e^{ik\theta},
\end{align}
and our assumption that $\sqrt{\psi} \in C^5(\sone)$ implies that $|a_k| \lesssim 1/k^5$. If we let $v_{\theta} := (\cos \theta, \sin \theta)$ and  $z=(r,\alpha)$ in polar coordinates, then the functions $f_n$ defined in \eqref{fnuseries} can be written as follows
\begin{align}\label{gn}
f_n(z) &= \sum_{k =-\infty}^{\infty} a_k \int_{0}^{2\pi} e^{i (n+k)\theta}~ e^{-i z\cdot v_{\theta}}~ d\theta, \nonumber \\
&= \sum_{k =-\infty}^{\infty} a_k~ i^{-(n+k)} e^{i(n+k)\alpha} J_{|n+k|}(r),
\end{align}
where $J_k$ is the Bessel function of order $k$.
\begin{Claim}\label{fnestimate}
	There is $C>0$ such that for every $R\in \mathbb{N}$, every $n\geq 10R$ and every $z \in B(R)$, we have	$|f_n(z)| \leq C/n^3$. 
	\end{Claim}
\begin{proof}
	Let $r = |z|$ and $A := \max_{k} |a_k|$.  Using the fact that $|J_k| \leq 1$, we conclude the following from \eqref{gn}
	\begin{align}
|f_n(z)| &\leq \sum_{k =-\infty}^{\infty} |a_k|~ |J_{|n+k|}(r)|  \leq \sum_{k =-3n/2}^{-n/2} |a_k| +  A \sum_{|m|=n/2}^{\infty} |J_m(r)|, \nono\\
&\lesssim_{\eqref{besselesti}} \frac{1}{n^3} +  \sum_{|m|=n/2}^{\infty} \left(\frac{2R}{m}\right)^{m} \lesssim \frac{1}{n^3} + \frac{1}{2^{n/2}} \lesssim \frac{1}{n^3}.\label{gesti}
\end{align}
\end{proof}
With notations as in \eqref{fnuseries} and for  $N \in \mathbb{N}$, define $h_N$ and $G_N$ as follows
\begin{align}\label{hngndef}
h_N(z)  := \sum_{ 1\leq |n| \leq 10N} \xi_n f_n(z)\text{ and } G_N(z) := \sum_{ |n| > 10N} \xi_n f_n(z).
\end{align}
\begin{Claim}\label{claim1} There are constants $c,C>0$ such that for every  $R\in \nat$ large enough, there is an event $\Omega_R$ with $\p(\Omega_R) \geq (1-e^{-cR^2})$ on which the following holds
	\begin{align*}
	|G_R(z)| \leq  \frac{C}{R},\text{ for every $z \in B(R).$}
	\end{align*}
\end{Claim}
\begin{proof}
	For $n \geq 1$, let $\xi_n := X_n + i Y_n$, where $X_n,Y_n$ are i.i.d. $\mathcal{N}(0,1/\sqrt{2})$ random variables. For $R \in \nat$, define 
	\begin{align}\label{defomegan}
	\Omega_{R} := \{|\xi_n|\leq n,~\forall n\geq 10 R\}.
	\end{align}
	We first get a lower bound for $\p(\Omega_R)$. Since $\{|X_n|\leq n/\sqrt{2}\} \cap \{|Y_n|\leq n/\sqrt{2}\} \subseteq \{|\xi_n|\leq n\}$, there is a constant $c_1 >0$ such that 
	\begin{align*}
	\p(|\xi_n|\leq n) &\geq 1- e^{-c_1 n^2}.
	\end{align*}
	Since $\{\xi_n: n\in \nat\}$ are i.i.d., it follows that there exist constants $c, c_2>0$ such that for large enough $R$,
	\begin{align*}
	\p(\Omega_{R}) & \geq \prod_{n=10 R}^{\infty} (1- e^{-c_1 n^2}) \geq \prod_{n=10 R}^{\infty} \exp(-2e^{-c_1 n^2}),\\
	& \geq \exp(-c_2 e^{-c_1 R^2}) \geq 1 - c_2e^{-c_1 R^2} \geq 1 - e^{-c R^2}.
	\end{align*}
	On $\Omega_R$,	we get an estimate for $G_R(z)$ when $z\in B(R)$. It follows from \eqref{gesti} that 
	\begin{align*}
	|G_R(z)| \leq \sum_{|n|>10R} |\xi_n| |f_n(z)| \lesssim \sum_{|n|>10R} n \cdot \frac{1}{n^3} \leq \frac{C}{R}.  
	\end{align*}
\end{proof}
We now present the proof of Theorem \ref{thmlowerbound}. 
\begin{proof}
	We first give a sketch of the proof. For $R \in \nat$, it follows from \eqref{fnuseries} and \eqref{hngndef} that we can write $F_{\nu}$ as the following sum
	\begin{align*}
	F_{\nu}(r,\theta) = \xi_0 f_0 (r,\theta) + h_{R}(r,\theta) +  G_{R}(r,\theta) =:  \xi_0 f_0 (r,\theta) + H_{R}(r,\theta). 
	\end{align*}
	For $\rho$ satisfying $0<\rho \ll 1$ to be chosen eventually, define the event $\tom_R$ as follows
	\begin{align*}
	\tom_R := \{\sqrt{R} \leq |\xi_0|\leq 10\sqrt{R},~ \sum_{n=1}^{10R} |\xi_n|^2 \leq \rho^2 R \}.
	\end{align*}
	It follows from the independence of  $\{\xi_n\}_{n\geq 0}$ that there is a constant $c_{\rho} >0$ such that $\p(\tom_{R}) \geq e^{-c_{\rho}R}$. Let $\Omega_{R}$ be as in Claim \ref{claim1}, then for large enough $R$ we have
	\begin{align} \label{prob}
	\p(\Omega_R \cap \tom_R) \geq (e^{-c_{\rho}R} - e^{-CR^2}) \geq e^{-2c_{\rho}R}. 
	\end{align}	
We show that on $\tom_{R} \cap \Omega_{R}$, the values of  $H_R$ and $|\nabla H_R|$ are \textit{small} on most parts of $B(R)$, hence $F_{\nu}$ can be viewed as a $C^1$ perturbation of $\xi_0 f_0$ and so the nodal sets of $F_{\nu}$ and $f_0$ are \textit{similar}. We also show that $f_0$ is an outlier in terms of the nodal component count, that is, $N_R(f_0) \ll \e[N_R(F_{\nu})]$. Hence on $\Omega_R \cap \tom_R$, we expect $N_R \ll \e[N_R(F_{\nu})]$ and this is what is shown below.

\noin \textit{An estimate for the $L^2$ norm of $H_R$ on $B(R)$.} It follows from Lemma \ref{lemnorm} that 
\begin{align}
\int_{B(R)} |\sum_{1 \leq |n|\leq 10R} \xi_n f_{n}(z)|^2 dz &\lesssim  R \int_{0}^{2\pi} \frac{1}{\psi(\alpha)}|\sum_{1 \leq |n|\leq 10R} \xi_n  e^{in\alpha}|^2 \psi(\alpha) d\alpha, \nonumber\\
\int_{B(R)} |h_{R}|^2  & \lesssim  R \sum_{n=1}^{10R} |\xi_n|^2 \lesssim \rho^2 R^2~\text{on $\tom_R$.} \label{ltwonorm}
\end{align}
It follows from Claim \ref{claim1} and \eqref{ltwonorm} that on $\Omega_{R} \cap \tom_{R}$, we have
\begin{align} \label{ltwonorm2}
\int_{B(R)} |H_{R}|^2  & \lesssim \rho^2 R^2.
\end{align} 

\noin \textit{Nodal set of $f_0$.} Let $z=(r,\alpha)$, we have from the series defining $f_0$ in \eqref{gn} that
\begin{align}\label{fnotseries}
f_0(r,\alpha) & = \sum_{k \in \mathbb{Z}} a_k~ e^{ik(\alpha - \frac{\pi}{2})} J_{|k|}(r).
\end{align}
Using the asymptotic expression for Bessel functions \eqref{besasym}  in \eqref{fnotseries}, we get
\begin{align}
f_0(r,\alpha)&= \sqrt{\frac{2}{\pi r}}  \sum_{k \in \mathbb{Z}} a_k~ e^{ik(\alpha - \frac{\pi}{2})} \left[ \cos \left( r - \frac{k \pi}{2} - \frac{\pi}{4}\right)  + O\left(\frac{k^2}{r}\right) \right], \label{221}\\
& =  \sqrt{\frac{2}{\pi r}} ~ \left[ \sqrt{\psi(\alpha)} \cos \left( r - \frac{\pi}{4}\right)    + O\left(\frac{1}{r}\right) \right]. \label{230}
\end{align}
Since the error term in \eqref{221} is a function of $r$ and $|a_k| \lesssim 1/k^5$, we have the following expression for $\partial_{\alpha}f_0$
\begin{align}\label{partialfnot}
\partial_{\alpha}f_0(r,\alpha) &= \sqrt{\frac{2}{\pi r}} ~ \left[ \frac{d}{d \alpha} (\sqrt{\psi(\alpha)}) \cos \left( r - \frac{\pi}{4}\right)    + O\left(\frac{1}{r}\right) \right].
\end{align}
We get the following expression for $\partial_r f_0$ from \eqref{besdif}
\begin{align}
\partial_r f_0(r,\alpha)& =  -\sqrt{\frac{2}{\pi r}}\sum_{k \in \mathbb{Z}} a_k~ e^{ik(\alpha - \frac{\pi}{2})}   \left[ \sin \left( r - \frac{k \pi}{2} - \frac{\pi}{4} \right) + O\left(\frac{k^2}{r}\right)\right], \nonumber\\
& = -\sqrt{\frac{2}{\pi r}} ~  \left[\sqrt{\psi(\alpha)}  \sin \left( r  - \frac{\pi}{4}\right)  + O\left(\frac{1}{r}\right)\right]. \label{222}
\end{align}
We conclude using \eqref{230} and \eqref{222} that there is $c_0>0$ such that if $r>0$ is large enough, then
\begin{align}\label{stabb}
|f_0(z)| + |\nabla f_0(z)| \geq c_0/\sqrt{r}.
\end{align}
Consider $\tilde{f_0}$ defined as follows
\begin{align}\label{f0tilde}
\tilde{f_0}(r,\alpha) :=  \sqrt{\frac{2}{\pi r}} ~ \sqrt{\psi(\alpha)} \cos \left( r - \frac{\pi}{4}\right),
\end{align}
then $f_0$ is a $C^1$ perturbation of $\tilde{f_0}$, that is, we have  from \eqref{230}, \eqref{partialfnot} and \eqref{222} that
\begin{align*}
|(f_0 - \tilde{f}_0)(r,\alpha)| + |\nabla (f_0 - \tilde{f}_0)(r,\alpha)| \lesssim \frac{1}{r^{3/2}}.
\end{align*}
We note that the nodal set of $\tilde{f_0}$ consists of concentric circles of radii $(2n+1)\pi/2 + \pi/4,\text{ for } n \in \nat_0$ and hence $N_R(\tilde{f_0}) \lesssim R$.\\ 

The following analysis holds on the event $\Omega_R \cap \tom_R$.

\noin \textit{An upper bound for the nodal length of $F_{\nu}$.}
We estimate the $L^2$ norm of $F_{\nu}$ on $B(R)$,
\begin{align}\label{l2bound}
\int_{B(R)} F_{\nu}^2 \leq 2 \left(|\xi_0|^{2} \int_{B(R)}  f_{0}^{2} + \int_{B(R)} |H_R|^2 \right) \lesssim R^2. 
\end{align}

Cover $B(R)$ with a minimal collection of Euclidean unit discs all of whose centers lie in $B(R)$, denote this collection by $\mathcal{C}_1 = \{D_i\}_{i \in \mi}$ and note that $ |\mathcal{I}| \simeq R^2$. Let $\widetilde{\mathcal{I}} \subset \mathcal{I}$ be the set of indices $i$ such that there is a point in $2D_i$ where $|H_R|+|\nabla H_R| \geq c_0/10$. It follows from the regularity estimates \eqref{regpr} that $|\widetilde{\mathcal{I}}| \lesssim \rho^2 R^2$. Let $j \notin \widetilde{\mathcal{I}}$, we have as a consequence of \eqref{stabb} that $|F_{\nu}| + |\nabla F_{\nu}| \geq c_0/2$ on $2D_j$.
 For $i \in \mi$, let $M_i := \max_{2D_i}\{|F_{\nu}|,|\nabla F_{\nu}|, |\nabla^2 F_{\nu}|\}$. As a consequence of the regularity estimates \eqref{regpr} and \eqref{l2bound}, we conclude that $\sum_{i \in \mi} M_{i}^{2} \lesssim R^2$. Now using Corollary \ref{lengthcor}, 
  we get the following bound for the nodal length of $F_{\nu}$ in $D_j$ and hence on $B(R) \setminus \{\cup_{i \in \widetilde{\mathcal{I}}}~D_i\}$,
\begin{equation}
\begin{gathered}
\mathcal{L}(\mathcal{Z}(F_{\nu}) \cap D_j) \lesssim M_j,\\
\mathcal{L}(\mathcal{Z}(F_{\nu}) \cap (B(R) \setminus  \{\cup_{i \in \widetilde{\mathcal{I}}}~D_i\})) \lesssim \sum_{i \in \mi} M_i	 \lesssim \sqrt{ |\mi|^2} \sqrt{\sum_{i \in \mi} M_{i}^{2}} \lesssim R^2. \label{nlbd}
\end{gathered}
\end{equation}
\textit{Counting nodal components of $F_{\nu}$.} Cover $B(R) \setminus B(\sqrt{R})$ with a minimal collection of Euclidean discs of radius $\ell$, all of whose centers lie in $B(R)$ and denote  this collection by $\mathcal{C}_2 = \{B_j\}_{j \in \mathcal{J}}$. The exact value of $\ell \gg 1$ will be chosen eventually. A disc $B_j \in \mathcal{C}_2$ is called \textit{good} if at every point in $3B_j$, atleast one of $|F_{\nu}| > c_0/4$ or $|\nabla F_{\nu}| > c_0/4$ holds. As a consequence of \eqref{regpr} and  \eqref{ltwonorm2}, the number of indices $j \in \mathcal{J}$ for which there is a point in $3B_j$ where $|H_R| \geq c_0/10$ or  $|\nabla H_R| \geq c_0/10$ holds is $\lesssim \rho^2 R^2$. For all other discs in $\mathcal{C}_2$, we have
\begin{align*}
|F_{\nu}| + |\nabla F_{\nu}| \geq |\xi_0 f_0| + |\nabla \xi_0 f_0| - |H_R| - |\nabla H_R| \geq_{\eqref{stabb}} \frac{4 c_0}{5}, 
\end{align*}
and hence they are good. For a nodal component $\Gamma$ of $F_{\nu}$ intersecting a good disc $B_j$ and whose diameter does not exceed $\ell$, it follows from Result \ref{NDcount} (with $f=F_{\nu}$ and $h=-H_R + \xi_0(\tilde{f_0} - f_0)$) that there is a corresponding nodal component of $F_{\nu} - H_R + \xi_0(\tilde{f_0} - f_0) = \xi_0 \tilde{f_0}$ which is contained in $3B_j$. Now using the nodal length bound established in \eqref{nlbd} and counting nodal components of $F_{\nu}$ in a similar manner as was done in the proof of Theorem \ref{commonthm}, we get
\begin{align}
N_{R}(F_{\nu})&\lesssim \left(\rho^2 \ell^{2} + \frac{1}{\ell}\right) R^2 + N_{\sqrt{R}}(F_{\nu}) + N_{R}(\xi_0 \tilde{f_0}) \lesssim \left(\rho^2 \ell^{2} + \frac{1}{\ell}\right) R^2 + R.
\end{align}
 Let $\kappa >0$ be small enough, choosing $\ell \simeq \kappa^{-1}$ and $\rho = \kappa^2$, we get $N_{R}(F_{\nu})/R^2 \leq 3\kappa$ on $\Omega_R \cap \tom_R$ and the result follows from \eqref{prob}.
	\end{proof}

\section{Proof of Theorem \ref{thmsph}} \label{pf2}


\noin In this section, we deduce Theorem \ref{thmsph} from Theorem \ref{commonthm}.  We first introduce some notations which will be used throughout this section. 
\begin{Notations}\hfill
	\begin{itemize}[align=left,leftmargin=*,widest={9}]
		\item $\textup{exp}_{p}$ will denote the exponential map at $p =(0,0,1)$ on the sphere $\sph$. Denote the metric on $\sph$ by $g_{\sph}$. Define $\sph_{+} := \{(x,y,z) \in \sph: z>0\}$. 
		\item For any Riemannian manifold $(M,g_M)$ and $q \in M$, $B_{g_M}(q,r)$ will denote the geodesic ball of radius $r$ and center $q$, while $B(q,r)$ and $\mathcal{D}(q,r)$ will denote analogous discs in $\re^2$ and $\sph$ respectively. $\mathcal{D}(\ep)$ will denote $\mathcal{D}(p,\ep)$.
		\item 
		For $u=(u_1,u_2)$, the usual coordinates on $\re^2$ and $i=(i_1,i_2) \in \nat_{0}^{2}$,  we define $\partial_{u}^{i} := \partial_{u_1}^{i_1} \partial_{u_2}^{i_2}$ and $|i|:= i_1 + i_2$.
		\item  For $n \in \nat$, the function $\mathscr{B}_n :\re^2 \ra \re^2$ is defined by $\bup(x,y):= \sqrt{n(n+1)}(x,y)$. 
	\end{itemize}
	\vskip .3cm
	Consider $\re^2$ with the Euclidean metric and let $I_p : \re^2 \ra T_{p}\sph$ be an isometry. 
	Then the restricted map $\ex_{p} \circ I_p : B(0,\pi/2) \ra \sph_{+}$ is a diffeomorphism. For $n \in \nat$, define the random function $F_n$ on $(-n,n)^2$ by
	\begin{align} \label{fndefn}
	F_n (x,y) = (\mathscr{F}_n \circ \ex_p \circ I_p \circ \bupin)(x,y),
	\end{align}
	where $\mathscr{F}_n$ is the degree $n$ random spherical harmonic defined in \eqref{rsh}. Let $\tilde{g}$ be the metric on $B(0,\pi/2)$ which is the pullback of $g_{\sph}$ via the map $\textup{exp}_{p} \circ I_p$. Define the metric $g_n$  by 
	\begin{align*}
	g_n(x,y)(\cdot,\cdot) := \tilde{g} (\bupin(x,y))(\cdot, \cdot).
	\end{align*}
	We show that $F_n$ and $g_n$ defined above satisfy the assumptions \ref{as1}--\ref{as6} of Theorem \ref{commonthm}. Since $\mathscr{F}_n$ is a smooth and centered Gaussian process on $\sph$, it follows that  $F_n$ is also a smooth and centered Gaussian process on $(-n,n)^2$. Checking that assumptions \ref{as1}--\ref{as3} hold for this choice of $F_n$ and $g_n$ is easy. \ref{as6} follows directly from \eqref{regpr}. The rest of this section is devoted to checking that assumptions \ref{as4} and \ref{as5} hold. We omit checking $(A4')$ since it can be established by arguments similar to the ones in Lemma \ref{sphnondeg}.
\end{Notations}
\subsection{Checking \ref{as4}}  Below, we collect a few facts and results which are used to establish \ref{as4} for our choice of $F_n$ and $g_n$.

\begin{enumerate}[label={\arabic*.},align=left,leftmargin=*,widest={8}]
	\item Let $\mathscr{F}$ be a smooth Gaussian process on $\re^2$ whose covariance kernel is $\mathscr{K}$. For $u,v \in \re^2$ and $i,j \in \natz^2$, we have 
	\begin{align}\label{dercovkernel}
	\e[\partial_{u}^{i} \scf(u) \partial_{v}^{j} \scf(v)] = \partial_{u}^{i} \partial_{v}^{j} \mathscr{K}(u,v). 
	\end{align}
	\item It is shown in \cite{Zel} that the random plane wave is the scaling limit of the random spherical harmonics. We now present this result as it appears in Section 2.5.2 of \cite{MS}. 
	\begin{Result} \label{scalinglimit} Let $K_n$ be the covariance kernel of $F_n$ defined in \eqref{fndefn} and let $K$ be the covariance kernel of the random plane wave $F$. Then for every $i,j \in \nat_{0}^{2}$, the following convergence happens locally uniformly in $u,v$ as $n \ra \infty$
		\begin{align*}
		\partial_{u}^{i} \partial_{v}^{j} K_n(u,v) \longrightarrow \partial_{u}^{i} \partial_{v}^{j} K(u,v). 
		\end{align*}
	\end{Result}
	\item From \eqref{dercovkernel} and Result \ref{scalinglimit}, we conclude the following convergence of the Gaussian vectors for $\ell =1,2$, as $n \ra \infty$
	\begin{align} \label{gaussconvg}
	(F_{n}(0),\partial_{\ell} F_{n}(0), \ldots,\partial_{\ell}^{106}F_{n}(0)) \Longrightarrow (F(0),\partial_{\ell} F(0),\ldots,\partial_{\ell}^{106}F(0)).
	\end{align}  
	\item \ref{as4} was established for the random plane wave $F$ in Lemma \ref{lemnondeg} and let $\tilde{\kappa}_j >0$ be an upper bound for the density of the Gaussian vector  $(F(0),\partial_{\ell} F(0),\ldots,\partial_{\ell}^{j}F(0))$. Hence it follows from \eqref{gaussconvg} that for large enough $n$ and every $j \leq 106$,
	\begin{equation} \label{texteq}
	\parbox{.8\textwidth}{$(F_n(0),\partial_{\ell} F_n(0), \ldots,\partial_{\ell}^{j}F_n(0))$ is a non-degenerate Gaussian whose density is bounded above by $2\tilde{\kappa}_j$.
	}
	\end{equation}
\end{enumerate}
\begin{Lemma} \label{sphnondeg}There is a $\delta >0$ such that for every $j \leq 106$ , there exists $\kappa_j >0$ such that for  large enough  $n\in\nat$ and every $z \in B_{g_{n}}(\delta n)$, the Gaussian vector $(F_{n}(z),\partial_{\ell} F_{n}(z), \ldots,\partial_{\ell}^{j}F_{n}(z))$ is non-degenerate and its density is bounded above by $\kappa_j$.
\end{Lemma}
\begin{proof} The value of $\delta$ will be chosen towards the end of this proof. We show that for  large enough $n \in \nat$ and every $z \in B_{g_n}(\delta n)$, the Gaussian vector $(F_{n}(z),\partial_{\ell} F_{n}(z), \ldots,$ $\partial_{\ell}^{j}F_{n}(z))$ is uniformly close in distribution to that of $(F_n(0),\partial_{\ell} F_n(0),$ $\ldots,\partial_{\ell}^{j}F_n(0))$ and the desired result then follows from \eqref{texteq}. In order to do this, for every $z \in B_{g_n}(\delta n)$, we introduce another Gaussian process $G_n$  on $\nsq$ which is coupled with $F_n$, has the same distribution as $F_n$ on $\nsq$ and is such that for every $j \leq 106$,  $\partial_{\ell}^{j}F_{n}(z)$ can be expressed as a linear combination of the Gaussian random variables $\{\partial_{1}^{i_1}\partial_{2}^{i_2}G_n(0): 1 \leq i_1 + i_2 \leq 106\}$.
	
	For $q \in \sph$, let $\mathscr{R}_q : \sph \ra \sph$ denote the rotation which maps $p$ to $q$. The map $\mathscr{R}_q $ is a self-isometry of $\sph$ and hence the derivative map $d\mathscr{R}_q|_{T_{p}(\sph)} : T_{p}(\sph) \ra T_{q}(\sph)$ is an isometry. We shall henceforth denote $d\mathscr{R}_q|_{T_{p}(\sph)} $ by ${\mathscr{R}_{q_{*}}}$.  
	Let $G_{n,q}$ be the smooth Gaussian process on $\nsq$ defined by
	\begin{equation}\label{eq3000}
	\begin{aligned}
	G_{n,q} (x,y) &:= (\scf_n \circ \ex_q \circ {\mathscr{R}_{q_{*}}}  \circ I_p  \circ \bup^{-1})(x,y),\\
	& = (F_n \circ \Psi_{q}^{-1})(x,y),\\
	\Psi_{q} & := (\bup \circ I_{p}^{-1} \circ {\mathscr{R}_{q_{*}}^{-1}} \circ \ex_{q}^{-1} \circ \ex_p   \circ I_p  \circ \bup^{-1}).
	\end{aligned}
	\end{equation}
	When $q$ is \textit{close} to $p$, the map $\Psi_q$ in \eqref{eq3000} is \textit{close} to the identity map and following is a quantitative justificaion for this. Rotation invariance of $\scf_n$ implies that $F_n \overset{d}{=} G_n$ on $\nsq$ (see Claim \ref{claimcalc2}) and hence for every $j \in \nat$ and $\ell =1,2$ we have
	\begin{align*}
	(F_n(0),\partial_{\ell} F_n(0), \ldots,\partial_{\ell}^{j}F_n(0)) \overset{d}{=} (G_n(0),\partial_{\ell} G_n(0), \ldots,\partial_{\ell}^{j}G_n(0)).
	\end{align*}
	Consider the smooth map $\psi: \sph_{+} \times \overline{B(0,1)} \ra \re^2$ defined by
	\begin{align*}
	\psi(q,v) := (I_{p}^{-1} \circ {\mathscr{R}_{q_{*}}^{-1}} \circ \ex_{q}^{-1} \circ \ex_p   \circ I_p ) (v).
	\end{align*}
	Note that $\psi(p,v) = v$ for every $v \in \overline{B(0,1)}$. Let $\psi = (\psi_1, \psi_2)$, then we have for every $v \in \overline{B(0,1)}$ and $i \in \natz^2$ with $|i| \geq 2$
		\begin{equation}\label{derH}
	\begin{aligned}
	\partial_{v_1}\psi_1 (p,v) = \partial_{v_2}\psi_2 (p,v) = 1,\\
	\partial_{v_1}\psi_2 (p,v) = \partial_{v_2}\psi_1 (p,v) = 0,\\
	\partial_{v}^{i}\psi_1(p,v) = \partial_{v}^{i}\psi_2(p,v) =0.
	\end{aligned}
	\end{equation}
	It follows by continuity that given $\epsilon >0$, there is $\delta >0$ such that for every $v \in \overline{B(0,1)}$, every $\ell \in \{1,2\}$, every $q \in \sph$ satisfying $d_{\sph}(p,q) \leq \delta$ and every $|i| \leq 106$ we have 
	\begin{equation}\label{derH2}
	\begin{aligned}
	| \partial_{v}^{i}\psi_{\ell}(q,v) - \partial_{v}^{i}\psi_{\ell}(p,v)| \leq \ep.
	\end{aligned}
	\end{equation}
	We will make a choice of $\epsilon >0$ soon, let $\delta >0$ be as in \eqref{derH2} for this choice of $\ep$. 

	Fix $z \in B_{\gn}(\delta n)$, let $q_z \in \sph$ be such that $ q_z = ( \ex_{p}\circ I_{p} \circ\bupin)(z)$. Note that  $q_z  \in \mathcal{D}(p,\delta)$. Define $\Psi := \Psi_{q_z}$ and let $\Psi = (\Psi_1, \Psi_2)$. Since $\Psi = \bup \circ \psi (q_z,\cdot) \circ \bupin$, we have
	\begin{equation}\label{psider}
	\begin{aligned}
	\Psi(x,y) &= \sqrt{n(n+1)}~\psi\left(q_z,\left(\frac{x}{\sqrt{n(n+1)}}, \frac{y}{\sqrt{n(n+1)}}\right)\right),\\
	\partial^{i} \Psi_{\ell} (x,y) &= \frac{1}{(\sqrt{n(n+1)})^{|i|- 1}}\partial^{i} \psi_{\ell} (q_z,\bupin (x,y)). 
	\end{aligned}
	\end{equation}
	We now conclude from \eqref{derH}, \eqref{derH2} and \eqref{psider} that for every $(x,y) \in \nsq$,
	\begin{equation}\label{derpsi2}
	\begin{gathered}
	|\partial_1 \Psi_1 (x,y) -1|,  |\partial_2 \Psi_2 (x,y) -1| \leq \ep, \\
	|\partial^{i} \Psi_{\ell}(x,y)| \leq \frac{\ep}{(\sqrt{n(n+1)})^{|i|- 1}},~\text{for all other $(i, \ell)$}. 
	\end{gathered}
	\end{equation}
	Define $G_n := G_{n,q_{z}}$., we  then have $F_n = G_n \circ \Psi$ and  $\Psi(z) =0$.  It now follows from \eqref{derpsi2}, Result \ref{faa} and the remark following it that for every $j \leq 106$,
	\begin{align*}
	\partial_{1}^{j}F_{n}(z) & = (\partial_{1} \Psi_1 (z))^j ~ \partial_{1}^{j}G_{n}(0) + \sum_{\substack{|\lambda| = 1 \\ \lambda \neq (j,0)}}^{j} c_{\lambda} \partial^{\lambda} G_n(0), \\
	& = (1+O(\ep))^j ~\partial_{1}^{j}G_{n}(0) + \sum_{\substack{|\lambda| = 1 \\ \lambda \neq (j,0)}}^{j} O(\ep) \partial^{\lambda} G_n(0).
	\end{align*}
	This shows that for $\ep>0$ small enough, the Gaussian vector $(F_{n}(z),\partial_{\ell} F_{n}(z), \ldots,$ $\partial_{\ell}^{106}F_{n}(z))$ is a slight perturbation of $(G_{n}(0),\partial_{\ell} G_{n}(0), \ldots,$ $\partial_{\ell}^{106}G_{n}(0))$. Choose $\ep>0$ small enough so that for every $j \leq 106$ and  large enough $n$,  $(F_{n}(z),\partial_{\ell} F_{n}(z), \ldots,$ $\partial_{\ell}^{j}F_{n}(z))$ is non-degenerate and its density is bounded above by $4 \tilde{\kappa}_j =: \kappa_j$.
\end{proof}

\subsection{Relation between the Cameron-Martin norm and the $L^2$ norm.}\hfill

\noin 
Let ${\mathscr{V}}_n$ be the space of degree $n$ spherical harmonics introduced in Section \ref{RSH}. Let $\{f_m: -n \leq m \leq n\}$ be the ultraspherical harmonics with respect to $p = (0,0,1)$, then $f_m$ can be expressed in spherical coordinates as follows  
\begin{align} 
f_m (\theta, \phi) = \sqrt{\frac{2n+1}{2\pi} \frac{(n-|m|)!}{(n+|m|)!}}~ P_{n}^{|m|} (\cos \theta) T(m\phi) =: \sqrt{2n+1}~h_m(\theta, \phi),\label{eq300a}
\end{align}
where $T(m\phi) = \cos(m \phi)$ when $m > 0$, $T(m\phi) = \sin(|m| \phi)$ when $m <0$, $T(0) = 1/\sqrt{2}$ and $P_{n}^{m}$ are the associated Legendre polynomials. The collection $\{f_m : -n \leq m \leq n\}$ has the property that it is an orthonormal basis for $\mathscr{V}_n$ w.r.t. the  $L^2(\sph)$ inner product and is an orthogonal set in $L^2(\mathcal{D}(r))$, 
 for every $r \in (0,\pi)$.
   Random spherical harmonics of degree $n$, $\scf_n$ can be expressed in this basis as
\begin{align*}
\scf_n = \frac{1}{\sqrt{2n+1}} \sum_{m=-n}^{n} \xi_{m} f_m,~\mbox{where $\xi_m \overset{\textup{i.i.d.}}{\sim} \mathcal{N}(0,1)$}.
\end{align*}
With this representation of $\scf_n$, $F_n$ can be expressed as follows for $(x,y) \in \nsq$
\begin{align*}
F_n(x,y) &= \frac{1}{\sqrt{2n+1}} \sum_{m=-n}^{n} \xi_{m} \widetilde{f_m}(x,y),
\end{align*}
where $\widetilde{f_m} = (f_m \circ \ex_{p} \circ I_p \circ \bupin)$. Consider the Hilbert space $W_n$ defined by $W_n := \text{span}\{\widetilde{f_m}: -n \leq m \leq n\}$, with inner product given by $\langle \widetilde{f_i},\widetilde{f_j}\rangle_{W_n} := \delta_{ij}$.  $F_n$ defines a Gaussian  measure on $W_n$ which we denote by $\ga_n$. The Cameron-Martin space $\mathcal{H}_n$ of $(W_n,\gamma_n)$ is given by $\mathcal{H}_n = (W_n, \langle \cdot , \cdot \rangle_{\mathcal{H}_n})$, where $\langle \cdot , \cdot \rangle_{\mathcal{H}_n}= (2n+1) \langle \cdot , \cdot \rangle_{W_n}.$ The following lemma gives the relation between the Cameron-Martin norm and the $L^2$ norm, more specifically \ref{as5} holds with $\psi_n$ satisfying $\psi_n(T) \lesssim 1/T$. 
\begin{Lemma}\label{lemsph1}
	There is a constant $C>0$ such that for every $h \in W_n$ and every $R=R_n$ such that $ R_n \leq \pi n/2\sqrt{2}$, we have
	\begin{align*}
	\int_{B_{g_n}(R)} h^2(x,y)~ dV_{g_n}(x,y) \leq C R \|h\|_{\mathcal{H}_n}^{2}.
	\end{align*}
\end{Lemma}
\begin{proof}
	Let $h = \sum_{m=-n}^{n} a_m \widetilde{f_m}$, then its Cameron-Martin norm is 
	\begin{align*}
	\|h\|_{\mathcal{H}_n}^{2} = (2n+1) \sum_{m=-n}^{n} a_{m}^{2}.
	\end{align*}
	By change of variables, we  get
	\begin{align}
	\int_{B_{g_n}(R)} h^2(x,y) dV_{g_n}(x,y) &= n(n+1) \sum_{m=-n}^{n} a_{m}^{2} \int_{\mathcal{D}(R/\sqrt{n(n+1)})} f_{m}^{2}(z) dV(z),\nonumber\\ 
	& \leq 2n^2 \left(\sum_{m=-n}^{n} a_{m}^{2}\right) \max_{m} \left( \int_{\mathcal{D}(R/n)} f_{m}^{2}(z) dV(z) \right). \label{eq300A}
	\end{align}
	It follows from \eqref{eq300a} and \eqref{eq300A} that it suffices to show the following for every integer $|m| \leq n$,
	\begin{align*}
	\int_{\mathcal{D}(R/n)} h_{m}^{2}(z) dV(z) \lesssim  \frac{R}{n^2}.
	\end{align*}
	We show this using Lemma \ref{lemmamain} which is proved in Section \ref{seclemmaproof}. For every $|m| \leq n$, we have
	\begin{align*}
	\int_{\mathcal{D}(R/n)} h_{m}^{2}(z) dV(z)  & \leq \frac{(n-|m|)!}{(n+|m|)!}  \int_{0}^{R/n} P_{n}^{|m|}(\cos \theta)^2 \sin\theta d\theta.
	\end{align*}
	Let $r=R/n$, then by a change of variable and the fact that $1-(r^2/2) \leq \cos r$, we have
	\begin{align*}
	\frac{(n-|m|)!}{(n+|m|)!}  \int_{0}^{r} P_{n}^{|m|}(\cos \theta)^2 \sin\theta d\theta & = \frac{(n-|m|)!}{(n+|m|)!}  \int_{\cos r}^{1} P_{n}^{|m|}(x)^2  dx,\\
	& \leq \frac{(n-|m|)!}{(n+|m|)!} \int_{\lb 1-\frac{r^2}{2} \rb}^{1} P_{n}^{|m|}(x)^2 dx,\\
	& \leq_{\eqref{want}} C \frac{r}{n} = C \frac{R}{n^2}. 
	\end{align*}
\end{proof}

\begin{Lemma}\label{lemmamain}
	There is a constant $C>0$ such that for every $m,n \in \nat_{0}$ satisfying $m \leq n$ and every $z \in [0,1]$ we have
	\begin{align}\label{want}
	\frac{(n-m)!}{(n+m)!}	\int_{z}^{1} \pmn(x)^2 dx \leq C \frac{\sqrt{1-z}}{n}.
	\end{align}
\end{Lemma}

\subsection{Concluding concentration} Since all the assumptions \ref{as1}--\ref{as6} hold for $F_n, \gn$ and $R$, it follows from Theorem \ref{commonthm} that $N_{R}(F_n)/\textup{Vol}_{g_n}[B_{\gn}(R)]$ concentrates around its median and hence around its mean also. In order to conclude concentration around the constant $c_{NS}$ as claimed in Theorem \ref{thmsph}, it suffices to show that the corresponding sequence of means converge. We establish this convergence in the following claim. The proof is essentially what appears in Section 5 of \cite{NS}, we prove it here for the sake of completeness. 


\begin{Claim} If $R=R_n \ra \infty$ as $n \ra \infty$, then we have
	\begin{align*} 
	\e \left[ \frac{N_{R}(F_n)}{\textup{Vol}_{g_n}[B_{\gn}(R)]}\right] \ra c_{NS}\text{ as $n \ra \infty$}.
	\end{align*}
\end{Claim}
\begin{proof} 
	Let $\Gamma$ be a closed set which is a union of piecewise smooth curves in $\sph$.  Let $\Gamma = \cup_{i} \gamma_i$, where each $\gamma_i$ is a connected component of $\Gamma$. For $U \subseteq \sph$, define the following quantities
	\begin{align*}
	N_{*}(\Gamma,U) := \sharp\{i: \gamma_i \subseteq U\}~\text{ and}~ N^{*}(\Gamma,U) := \sharp\{i: \gamma_i \cap U \neq \phi\}.
	\end{align*}
	The following integral geometric sandwich for the quantities defined above was established in Claim 5.1, \cite{NS}  and Lemma 1, \cite{NS2}. For every $\ep,T >0$ such that $T+\ep < \pi$, we have \vskip -.2cm
	\begin{align} \label{ig}
	\int_{\disc(T-\ep)} \frac{N_{*} (\Gamma, \disc(x,\ep) ) }{\textup{Vol}~\disc(\ep)} ~  dV(x)  \leq N_{*}(\Gamma, \disc(T)) \leq  	\int_{\disc(T+\ep)} \frac{N^{*} (\Gamma, \disc(x,\ep) ) }{\textup{Vol}~\disc(\ep)}~  dV(x).
	\end{align}
	\vskip .2cm
\noin	Let $f$ be a smooth function, we will use \eqref{ig} with $\Gamma = \mathcal{Z}(f)$. We let $N_{*}(f, U):= N_{*}(\mathcal{Z}(f), U)$ and $N^{*}(f, U):= N^{*}(\mathcal{Z}(f), U)$. For $d>0$, a component of $\mathcal{Z}(\scf_n)$ is said to be \textit{d-normal} if its diameter is less than $d/n$. Let $N_{d*}(\scf_n,U) $ and  $N_{d}^{*}(\scf_n,U)$ denote the count of the $d$-normal components of $\mathcal{Z}(\scf_n)$ which are entirely contained in $U$ and which intersect $U$ respectively.
	 In what follows we fix $r>0$ and let $\disc_x := \disc(x,r/n)$. We can write $N_{*}(\scf_n, \disc(\R))$ as follows
	\begin{align}
	N_{*}(\scf_n, \disc(R/n)) &= N_{d*}(\scf_n, \disc(R/n)) + \sharp\{\text{Components of $\mathcal{Z}(\scf_n)$ in $\disc(R/n)$} \nono\\
		& \hskip 7cm \text{with diameter $\geq d/n$}\}, \nonumber \\
		& \leq_{\eqref{ig}} \int_{\disc(\frac{R+r}{n})} \frac{N^{*}_{d} (\scf_n, \disc_{x} ) }{\textup{Vol}~\disc(r/n)}~ dV(x) + \frac{n}{d}~  \mathcal{L}_{g_{\sph}}(\mathcal{Z}(\scf_n) \cap \disc(R/n)). \label{ig0}	
	\end{align}
	Again from \eqref{ig}, we have
	\begin{align}\label{ig3}
	\int_{\disc(\frac{R-r}{n})} \frac{N_{*} (\scf_n, \disc_{x} ) }{\textup{Vol}~\disc(r/n)}~  dV(x) \leq N_{*}(\scf_n, \disc(R/n)).
	\end{align}
	Taking expectation and using the rotation invariance of $\scf_n$ in \eqref{ig0} and \eqref{ig3}, we get
	\begin{align}\label{ig2}
	\e[N_{*}(\scf_n, \disc(R/n))] &  \leq \frac{\e[N^{*}_{d} (\scf_n, \disc_{p} )] }{\textup{Vol}~\disc(r/n)}~ \textup{Vol}~\disc((R+r)/n)) +  \frac{Cn^2}{d}~  \textup{Vol}~\disc(R/n),
	\end{align}
	\begin{align} \label{ig4}
	\frac{\vol(\disc((R-r)/n))}{\textup{Vol}~\disc(r/n)} \e[N_{*} (\scf_n, \disc_{p} )]   \leq \e[N_{*}(\scf_n, \disc(R/n))].
	\end{align}
	We choose $r,R$ such that $1 \ll r \ll R \leq n$. With such a choice and the fact that $\vol~ \disc(\ep) = 4\pi \sin^2(\ep/2)$, we conclude the following from \eqref{ig2} and \eqref{ig4}
	\begin{align} \label{ig10}
	\frac{\e[N_{*} (F_n, B_{\gn}(r) )] }{\pi r^2} & \leq \frac{\e[N_{*}(F_n,\bgr)]}{\text{Vol}_{g_n}(\bgr)}\lb 1 + \frac{C'r}{R} \rb, \nono\\ 
	&\leq  \frac{\e[N^{*}_{d}(F_n, B_{\gn}(r))]}{\pi r^2} + \frac{Cr}{R} + \frac{C}{d}, \nonumber\\
	& \leq  \frac{\e[N_{*}(F_n, B_{\gn}(r))]}{\pi r^2} + C \frac{d}{r} + \frac{Cr}{R} + \frac{C}{d}. 
	\end{align}
	It follows from Theorem 5, \cite{MS} that for every $\ep >0$,
	\begin{align} \label{doublelimit}
	\lim_{r\ra \infty} \lim_{n\ra \infty} \p\left(\left|\frac{N_{*} (F_n, B_{\gn}(r) ) }{\pi r^2} -c_{NS} \right| > \ep\right) =0. 
	\end{align}
	Taking $d = \sqrt{r}$ in \eqref{ig10}, it now follows from \eqref{doublelimit} that 
	\begin{align*}
	\text{$\frac{\e[N_{*}(F_n,\bgr)]}{\text{Vol}_{g_n}(\bgr)} \ra c_{NS}$ as $n \ra \infty$.}
	\end{align*}
\end{proof}
\section{Proof of Lemma \ref{lemmamain}}\label{seclemmaproof}
The associated Legendre polynomials $\pmn$ satisfy the following second-order differential equation on $(-1,1)$
\begin{align}\label{pmndiffeq}
\frac{d}{dx}\lb(1-x^2)\frac{d}{dx}\pmn(x)\rb+\left(n(n+1) - \frac{m^2}{1-x^2}\right)\pmn(x) = 0,
\end{align}
and they are normalized so that 
\begin{align}\label{l2norm}
\frac{(n-m)!}{(n+m)!}	\int_{0}^{1} \pmn(x)^2 dx = \frac{1}{2n+1}.
\end{align}

\subsection{Preliminaries}
We first introduce a few notations and some useful results known about associated Legendre polynomials. Throughout this section, we fix $m,n \in \nat$ such that $m \leq n$. Define $\dnot$ and $x_k$, for $k \in \nat_0$, as follows

\begin{align}\label{xkdef}
\delta_0 := \lb 1 - \frac{m^2 - (1/4)}{n(n+1) - (3/4)}\rb^{1/2}\text{ and } x_{k} := \lb 1- \frac{m^2}{(k+1)n(n+1)} \rb^{1/2}.
\end{align}
Let $f :[-1,1] \ra \re$ denote the normalized associated Legendre polynomial defined by
\begin{align*}
f(x) := \sqrt{\frac{(n-m)!}{(n+m)!}}~\pmn (x). 
\end{align*}
For $x \in [0,1]$, we define $w(x)$ as follows
\begin{align*}
w(x) := \frac{m^2}{1-x^2} - n(n+1).
\end{align*}
$w(\cdot)$ appears in the differential equation \eqref{pmndiffeq} which $f$ satisfies, we note that $w$ is positive on $(\xnot,1)$, negative on $(0,\xnot)$ and vanishes at $\xnot$. 
\begin{Result}[\cite{alpzeros}, Corollary 1]\label{alpcp} $\pmn$  does not have any critical point in $[\xnot,1]$ and $\pmn(\xnot) \neq 0$.
\end{Result}
We denote by $\overline{x}_1$ the critical point of $f$ in $[0,1)$ which is closest to $1$, the following result gives an approximate location of $\xb$.
\begin{Result}[\cite{loh}, Lemma 2] The following estimate for $\overline{x}_1$ indicates that it is \textit{close} to $x_0$,
	\begin{align}\label{cpestimate}
	\frac{m^2}{(n+\frac{1}{2})^2} \leq 1-(\overline{x}_1)^2 \leq \frac{(1.11(m+1))^2}{n(n+1)}. 
	\end{align}
\end{Result}
\begin{Result}[\cite{loh}, Corollary 3] \label{maxestimate} For every $x\in [0,1]$ , we have
	\begin{align}\label{max}
	|f(x)| \leq C_0 \frac{1}{m^{1/4}},\text{ where $C_0 = 2^{5/4}/\pi^{3/4}$}.
	\end{align}
\end{Result}
The behaviour of $\pmn$ in $[0,\dnot)$ is \textit{oscillatory} in nature  and it exhibits \textit{exponential} behaviour in most part of $[\dnot,1]$, hence there is a need to analyse $\pmn$  in these two intervals separately. The oscillatory behaviour in $[0,\dnot]$ is captured by the following result which can be thought of as a generalization of the classical Bernstein's inequality for Legendre polynomials. 

\begin{Result}[\cite{krasibern}, Theorem 3] \label{krasbern} For every $x \in [0,\delta_0]$, we have
	\begin{align}\label{berntype}
	(\delta_{0}^{2} - x^2)^{1/4} \sqrt{\frac{(n-m)!}{(n+m)!}}~|\pmn (x)| \leq \sqrt{\frac{230}{\pi}}. 
	\end{align}
\end{Result}
The points $\dnot$ and $\xnot$ are special for $\pmn$ because of Result \ref{krasbern} and the fact that $w(\xnot) = 0$. The following comparison of these points will be useful later. 
\begin{align}\label{xnotdnot}
|\dnot^2 - \xnot^2| \leq \frac{1}{n(n+1)}. 
\end{align}

In order to prove the $L^2$ estimates for $\pmn$ claimed in Lemma \ref{lemmamain}, we need to understand the behaviour of $\pmn$ in $[\dnot,1]$ and  the next section is devoted to this. 

\begin{Rem}\label{remark1} It follows from \eqref{l2norm} that $\int_{[0,1]}f^2(x) dx =1/(2n+1)$ and hence in Lemma \ref{lemmamain} it suffices to establish the claim for $z \geq \sqrt{3}/2$. 
\end{Rem}

\subsection{Analysing $f$ in the non-oscillatory region} \label{secnonosc} 

A change of variable or certain other transformations are useful in analysing the solution of a second-order differential equation, and in particular $f$, in the \textit{monotone} region. Such a transformation renders the differential equation in a more suitable form which is easier to analyse and this is what we do below.
We first rewrite \eqref{pmndiffeq} as follows
\begin{align*}
f''(x)+ \left( \frac{-2x}{1-x^2}\right)f'(x) + \frac{n(n+1)}{(1-x^2)^2}(x_{0}^{2} - x^2)f(x) = 0,
\end{align*}
and define $p$ and $q$ by
\begin{align*}
p(x) :=  \frac{-2x}{1-x^2} \text{ and } q(x):= \frac{n(n+1)}{(1-x^2)^2}(x_{0}^{2} - x^2).
\end{align*}
We use the following transformation which appears  in  \cite{DEbook} (p.3,4). Define $u$ as follows
\begin{align}\label{ufrel}
u(x) &:= f(x) \exp\left(\frac{1}{2} \int_{0}^{x} p(v) dv \right) = f(x) \sqrt{1-x^2}. 
\end{align}
Then $u$ satisfies the following differential equation
\begin{align}\label{simpeq}
u''(x) - q_2(x) u(x) = 0,
\end{align}
where $q_2$ is given by
\begin{align}\label{qone}
- q_2(x) &:= q(x)  -\frac{1}{4} p(x)^2 - \frac{1}{2} p'(x) = \frac{n(n+1)(x_{0}^{2} - x^2) + 1}{(1-x^2)^2}.
\end{align}
We observe that 
\begin{align}\label{qonepos}
q_2(x) \geq 0 \Leftrightarrow x^2 \geq \xnot^2 + \frac{1}{n(n+1)} =: s^{2}. 
\end{align}
Note that for every $m\geq 1$, $q_2$ is increasing on $[s,1]$ (see  \ref{appq2inc}). The idea now is to compare the solution of \eqref{simpeq} with  the solution of $y'' = c^2 y$, where $c \in \re$. The following claim shows how this comparison can be done. 
\begin{Claim}\label{claim2} Let $[a,b] \subset \re$. Let $Q_1, Q_2$ be real valued functions on $[a,b]$ satisfying $Q_2 \geq Q_1 > 0$. Let $y_1$ and $y_2$ be smooth functions on $[a,b]$ satisfying the following conditions.
	\begin{enumerate}
		\item $y_1 ''(x) = Q_1(x)y_1(x)\text{ and }y_2 ''(x) = Q_2(x)y_2(x)$ on $[a,b]$.
		\item For $v=a\text{ and }b$, $(y_1 - y_2)(v) \geq 0$.
		\item $y_2 \geq 0$ on $[a,b]$.
	\end{enumerate}
	Then $(y_1 - y_2) \geq 0$ on the entire interval $[a,b]$. 
\end{Claim}
\begin{proof}Let $c \in [a,b]$ be a point where $(y_1 - y_2)$ attains its global minimum. If $c = a \text{ or }b$,  then our assumption that $(y_1 - y_2)$ is non-negative at the end points implies that $y_1 - y_2 \geq 0$ on $[a,b]$. If $c \in (a,b)$, then we have
		\begin{align*}
	0 \leq (y_1 - y_2)''(c) &= Q_1(c) y_1(c) - Q_2(c)y_2(c),\\
	& = Q_1(c) [y_1(c) - y_2(c)] - [Q_2(c) - Q_1(c)] y_2(c),
	\end{align*}
and hence 
\begin{align*}
y_1(c) - y_2(c) = \frac{1}{Q_1(c)} (y_1 - y_2)''(c)  + \lb \frac{Q_2(c)}{Q_1(c)} - 1\rb y_2(c) \geq 0,
\end{align*}	
and hence in this case also $(y_1 - y_2) \geq 0$ on $[a,b]$.
\end{proof}
Consider the point $t \in [0,1]$ defined by 
\begin{align}\label{zdef}
t^{2} := \xnot^2 + \frac{m+1}{n(n+1)}. 
\end{align}
When $m \geq 3$, it is easy to see that $t \in [s,x_1]$ and hence we have
\begin{align*}
x_0, \dnot \leq_{\eqref{xnotdnot}} s \leq  t \leq x_1 \leq x_2 \leq \cdots \leq 1.
\end{align*}
We now use Claim \ref{claim2} to get an upper bound for $f$ on $[t,1]$. 
\begin{Claim}\label{expcomp} If $C_0$ is as in \eqref{max} and $m \geq 3$, then for every $x \in [t,1]$ we have
	\begin{align}\label{fcomp}
	|f(x)| \leq C_0~ \frac{1}{m^{1/4}} \lb \frac{1-t^2}{1-x^2}\rb^{1/2} e^{-\sqrt{n^4/m^3}(x-t)}. 
	\end{align}
\end{Claim}
\begin{proof}
	Since $q_2$ is an increasing function on $[s,1]$, we conclude that 
	\begin{equation*}
	q_2(x) \geq q_2(t)~\text{ for every $x \in [t,1]$},
	\end{equation*}
	and the value of $q_2$ at $t$ is 
	\begin{align*}
	q_2(t) = \frac{n(n+1)(t^2 - \xnot^2) - 1}{(1-t^2)^2} \geq \frac{m}{(1-\xnot^2)^2} \geq \frac{n^4}{m^3} =: q_1(x). 
	\end{align*}
	We now use Claim \ref{claim2}, with $Q_i = q_i$, to compare $u$  and $h$ defined below.
	\begin{align*}
	h(x) := \lb \frac{C_0}{m^{1/4}} \sqrt{1-t^2}\rb  e^{-\sqrt{n^4/m^3}(x-t)},\text{ for $x \in [t,1].$}
	\end{align*}
	It follows from Result \ref{alpcp} that $f$ is monotone on $[\xnot,1]$, $f(\xnot) \neq 0$ and $f(1)=0$. Hence $f$ does not change sign on $[\xnot,1)$, we assume without loss of generality that $f >0$  on $[x_0,1)$ and hence $u >0$ on $[x_0,1)$ (otherwise we can work with $-f$). Note that $h$ and $u$ satisfy the following on  $[t,1]$ 
	\begin{align*}
	h''(x)   = q_1(x) h(x), \text{ } h(t) = \frac{C_0}{m^{1/4}} \sqrt{1-t^2} \text{ and $h(1)>0$},\\
	u''(x) = q_2(x) u(x),\text{ }u(t) = f(t)\sqrt{1-t^2}\text{ and $u(1)=0$}. 
	\end{align*}
	Hence it now follows from Claim \ref{claim2} that $u(x) \leq h(x)\text{ for $x \in [t,1]$}$  and our claim immediately follows from this. 
\end{proof}

\subsection{Proof of Lemma \ref{lemmamain}}\label{sec3proof} We first prove Lemma \ref{lemmamain} for the case when $z \in [x_2,1]$ and this is the content of the following claim. 
\begin{Claim}\label{cll}
	For every $n \in \nat$,  every $ 1 \leq m \leq n$  and every $z \in [x_2,1]$, \eqref{want} holds. 
\end{Claim}

\begin{proof} We first observe that since for every $k$ we have $\sqrt{1-x_k}/\sqrt{1-x_{k+1}} \leq 2$, it suffices to prove \eqref{want} for $z = x_k$, $k \geq 2$. Since $f$ satisfies \eqref{pmndiffeq}, we have
	\begin{align}\label{inteq}
0=	\int_{\overline{x}_1}^{1} ((1-x^2)f'(x))' dx = \int_{\overline{x}_1}^{1}\left(\frac{m^2}{1-x^2} - n(n+1)\right)f(x) dx,
	\end{align}
	and since $f$ does not change sign in $[\xb,1]$, we have
	\begin{align}\label{intinfo}
	\int_{\overline{x}_1}^{x_0} \left(n(n+1) -\frac{m^2}{1-x^2}\right)|f(x)| dx = \int_{x_0}^{1} \left(\frac{m^2}{1-x^2} - n(n+1)\right)|f(x)| dx.
	\end{align}
	Result \ref{krasbern} contains information about the behaviour of  $f$ in $[\overline{x}_1,x_0]$ and \eqref{intinfo} relates the behaviour of $f$ in $[x_0,1]$ to that in $[\overline{x}_1,x_0]$. Let us first estimate the l.h.s. in \eqref{intinfo}. 
	\begin{align}\label{eq1}
	\int_{\overline{x}_1}^{x_0} \left(n(n+1) -\frac{m^2}{1-x^2}\right)|f(x)| dx	& \leq \int_{\overline{x}_1}^{x_0} \frac{n(n+1)}{1-x_{0}^{2}} \left(x_{0}^{2} - x^2 \right) |f(x)| dx, \nono \\
	& \leq \frac{4n^4}{m^2} \int_{\overline{x}_1}^{x_0}  \left(x_{0}^{2} - x^2 \right) |f(x)| dx.
	\end{align}
	We now estimate the integral in \eqref{eq1}. We know that $\xb \leq \xnot$ and hence there are three possibilities for $\dnot$ : $\dnot \leq \xb$, $\xb \leq \dnot \leq \xnot$ and $\xnot \leq \dnot$. We get a common bound for the integral in all the three cases. If $\dnot \leq \xb$, then 
	\begin{align}
	\int_{\overline{x}_1}^{x_0}  \left(x_{0}^{2} - x^2 \right) |f(x)| dx &\leq_{\eqref{max}} C_0 \int_{\dnot}^{\xnot} |\xnot^2 - \dnot^2|\cdot \frac{1}{m^{1/4}} dx, \nono \\
	& \lesssim_{\eqref{xnotdnot}, \eqref{caldis}}  \frac{1}{n^2}\cdot \frac{m^2}{n^2}\cdot \frac{1}{m^{1/4}} =  \frac{m^{7/4}}{n^4}.  \label{case1}
	\end{align}
	If $\xb \leq \dnot \leq \xnot$, then we have
	\begin{align}
		\int_{\overline{x}_1}^{x_0}  \left(x_{0}^{2} - x^2 \right)~ |f(x)| dx & \leq \int_{\xb}^{\dnot} (\xnot^2 - \dnot^2)~ |f(x)|~ dx  + \int_{\xb}^{\dnot} (\dnot^2 - x^2)~ |f(x)|~ dx \nono\\
		& \hspace{5.5cm}+ \int_{\dnot}^{\xnot}  (\xnot^2 - \dnot^2) |f(x)|~ dx,\nono \\
		& =  \int_{\xb}^{\xnot} (\xnot^2 - \dnot^2)~ |f(x)|~ dx  + \int_{\xb}^{\dnot} (\dnot^2 - x^2)~ |f(x)|~ dx, \nono\\
		& \leq \frac{1}{n^2} \cdot \frac{C_0}{m^{1/4}}\cdot |\xnot - \xb|  +  (\dnot^2 - \xb^2)^{3/4} \int_{\xb}^{\dnot} (\dnot^2 - x^2)^{1/4}~ |f(x)|~ dx, \nono\\
		& \lesssim_{\eqref{berntype}} \frac{m^{7/4}}{n^4} + \lb\frac{m^2}{n^2}\rb^{3/4} \cdot \frac{1}{\sqrt{n}} \cdot \frac{m^2}{n^2} \lesssim \frac{m^{7/2}}{n^4}. \label{case2}
	\end{align}
	If $\xnot \leq \dnot$, then we have
	\begin{align}
		\int_{\overline{x}_1}^{x_0}  \left(x_{0}^{2} - x^2 \right)~ |f(x)| dx & \leq (\xnot^2 - \xb^2)^{3/4} \int_{\xb}^{\xnot} (\dnot^2 - x^2)^{1/4}~|f(x)|~dx, \nono\\
		& \lesssim_{\eqref{berntype}} \lb \frac{m^2}{n^2} \rb^{3/4} \cdot \frac{1}{\sqrt{n}} \cdot |\xnot - \xb| \lesssim \frac{m^{7/2}}{n^4}. \label{case3}
	\end{align}
We now conclude from \eqref{eq1}, \eqref{case1}, \eqref{case2} and \eqref{case3} that
	\begin{align}\label{eq2}
	\int_{x_0}^{1} \left(\frac{m^2}{1-x^2} - n(n+1)\right) |f(x)| dx \lesssim m^{3/2}.
	\end{align}
	Note that the weight $w(x)$ appearing in the integral in \eqref{eq2} is such that it vanishes at $x_0$ and it increases to infinity as $x \ra 1$.  Hence for $z$ away from $\xnot$, \eqref{eq2} can be used to get a good bound of $\int_{[z,1]}f(x) dx$ and we do this below. 
	\begin{equation} \label{eq3}
	\begin{gathered}
	\int_{x_k}^{1} \lb \frac{m^2}{1-x_{k}^{2}} -n(n+1) \rb |f(x)| dx  \leq \int_{x_k}^{1} \lb \frac{m^2}{1-x^2} -n(n+1) \rb |f(x)| dx \lesssim m^{3/2},\\
	kn(n+1) \int_{x_k}^{1} |f(x)| dx \lesssim m^{3/2}.
	\end{gathered}
	\end{equation}
	We now conclude using \eqref{fcomp} and \eqref{caldis} that for every $m \geq 12$, we have
	\begin{align}\label{fx2estimate}
	|f(x_2)| \leq \frac{C_0}{m^{1/4}}~ e^{-\sqrt{m}/24} \lesssim \frac{1}{m^2}.
	\end{align}
	We first use \eqref{caldis} to estimate $\sqrt{1-x_{k}}$, then use \eqref{eq3}, \eqref{fx2estimate} and the fact that $|f|$ is decreasing on $[\xnot,1]$  to get the following $L^2$ estimates for $f$.
	\begin{align}
	1-x_{k} = 1 -\lb 1 - \frac{m^2}{(k+1)n(n+1)} \rb^{1/2} \gtrsim_{\eqref{caldis}} \frac{m^2}{kn^2}, \label{eq009} \\
	\int_{x_k}^{1} f(x)^2 dx \leq |f(x_2)| \int_{x_k}^{1} |f(x)| dx \lesssim 
	\frac{1}{kn^2} \lesssim_{\eqref{eq009}} \frac{\sqrt{1-x_{k}}}{n}, \nonumber
	\end{align}
	and this proves the claim when $m \geq 12$. We now consider the case when $1 \leq m < 12$. It follows from Result \ref{maxestimate} that for every $m \leq 12$ and every $k \geq 2$, we have
	\begin{align*}
	\int_{x_k}^{1} f(x)^2 dx \lesssim 1-x_k \lesssim  \frac{1}{kn^2} \lesssim_{\eqref{eq009}} \frac{\sqrt{1-x_{k}}}{n}.
	\end{align*}
	This completes the proof for every $1 \leq m \leq n$. We note that the values of the constants which were suppressed in all the inequalities above are independent of $m$ and $n$. 
\end{proof}
\begin{Rem}\label{remark2} We note that when $m > \sqrt{3n(n+1)}/2$, the value of $x_2$ is
	\begin{align*}
	x_{2} = \sqrt{1- \frac{m^2}{3n(n+1)}} \leq \frac{\sqrt{3}}{2},
	\end{align*}
	and hence it follows from Remark \ref{remark1} that Claim \ref{cll} implies  Lemma \ref{lemmamain}. Hence in what follows we assume that 
	\begin{align}\label{ass1}
	m \leq \sqrt{3n(n+1)}/2 \text{ and hence $\xnot \geq 1/2$}.
	\end{align}
\end{Rem}
We now estimate the $L^2$ norm of $f$ in $[x_0,x_2]$.
\begin{Claim}\label{claimmidpart} For every $m,n \in \nat$ such that $m\leq \sqrt{3n(n+1)}/2 $, we have
	\begin{align*}
	\int_{[\xnot,x_2]} f(x)^2 dx \lesssim \frac{m}{n^2}.
	\end{align*}
\end{Claim}
\begin{proof} For $m \geq 3$, we use Claim \ref{expcomp} to conclude the following
	\begin{align*}
	\int_{[\xnot,x_2]}f^{2}(x) dx & = \int_{[\xnot,t]} f^{2}(x) dx + \int_{[t,x_2]} f^{2}(x) dx,\\
	& \lesssim_{\eqref{max},\eqref{fcomp}} \frac{(t-\xnot)}{\sqrt{m}}  + \frac{1}{\sqrt{m}} \lb \frac{1-t^2}{1-x_{2}^{2}}\rb \int_{[t,x_2]}  \exp \lb-\frac{2n^2}{m^{3/2}}(x-t)\rb dx,\\
	& \lesssim \frac{(t-\xnot)}{\sqrt{m}}  + \frac{1}{\sqrt{m}}  \frac{m^{3/2}}{2n^2}
	\lesssim~ \frac{(t^2-\xnot^2)}{2\xnot \sqrt{m}} +  \frac{m}{n^2},\\
	& \lesssim_{\eqref{ass1}} \frac{m}{\sqrt{m}~n^2} +  \frac{m}{n^2} \lesssim \frac{m}{n^2}.
	\end{align*}
	The case $m \leq 3$ follows easily by using the bound in Result \ref{maxestimate},
	\begin{align*}
	\int_{[\xnot,x_2]}f^{2}(x) dx \lesssim x_2 - x_0 \lesssim \frac{1}{n^2}. 
	\end{align*}
\end{proof}
We are now ready to  prove Lemma \ref{lemmamain}.
\begin{proof} It follows from Remark \ref{remark2} that we only have to consider the case when $m \leq \sqrt{3n(n+1)}/2$.  For $z \in [x_2,1]$, the required conclusion follows from Claim \ref{cll}. For $z \in [\xnot, x_2]$, we observe that 
	\begin{equation}\label{eeqq1}
	\frac{m}{n} \lesssim 	\frac{\sqrt{1-z^2}}{2} \leq \sqrt{1-z} = \frac{\sqrt{1-z^2}}{1+z} \leq \sqrt{1-z^2} \lesssim \frac{m}{n}.
	\end{equation}
	In this case too, the required conclusion follows from \eqref{eeqq1}, Claims \ref{cll} and \ref{claimmidpart}.  For $z \in [0,x_0]$, we have
	\begin{align}\label{final}
	\int_{z}^{1} f(x)^2 dx & \leq  \int_{z}^{\dnot} f(x)^2 dx + 1_{\dnot \leq \xnot} \int_{\dnot}^{\xnot} f(x)^2 dx + \int_{\xnot}^{1} f(x)^2 dx.
	\end{align}
	The first integral in the r.h.s. of \eqref{final} is estimated using \eqref{berntype}, while  \eqref{max} is used to estimate the second integral. Doing so we get
	\begin{align*}
	\int_{z}^{1} f(x)^2 dx & \lesssim \frac{1}{n} \int_{z}^{\dnot} \frac{1}{(\dnot^2 - x^2)^{1/2}}~ dx +  \frac{|\dnot - \xnot|}{\sqrt{m}} + \frac{\sqrt{1-x_0}}{n},\\
	& \lesssim \frac{1}{n \sqrt{\dnot}} \int_{z}^{\dnot} \frac{1}{(\dnot - x)^{1/2}}~ dx + \frac{|\dnot^2 - \xnot^2|}{\sqrt{\dnot} \sqrt{m}} + \frac{\sqrt{1-x_0}}{n}.
	\end{align*}
	It follows from Remark \ref{remark2} that $\dnot \gtrsim 1$ and hence we have
	\begin{align*}
		\int_{z}^{1} f(x)^2 dx & \lesssim \frac{\sqrt{\dnot - z}}{n} + \frac{1}{n^2 \sqrt{m}} +  \frac{\sqrt{1-x_0}}{n} \lesssim \frac{\sqrt{1-z}}{n},
	\end{align*}
	and this establishes Lemma \ref{lemmamain}.
\end{proof}	


  \section{Proof of Theorem \ref{thmarw}} \label{pfarw}
  In this section, we use the notations introduced in Section \ref{secarw}. We first start with a result used to establish \ref{as5} which  is about \textit{almost} equidistribution of $L^2$-mass of Laplace eigenfunctions on $\torus$ at scales slightly larger than the wavelength scale. 
  \begin{Result}[\cite{granwig}, Corollary 2.3]\label{granwig} There exists a density one subset $\set' \subseteq \set$ such that for every $n \in \set',~f \in \mathcal{W}_n$ and $\kappa>0$ the following holds. For every $R \geq (\log n)^{1+\frac{\log 2}{3}+\kappa}$ and every $z \in \torus$ we have
  	\begin{align*}
  	\int_{B(z,\frac{R}{\sqrt{n}})} f^2(z) dz = \left(\frac{\pi R^2}{n}\right)~ \|f\|_{L^2(\torus)}^{2}(1+o(1)).  
  	\end{align*}
  \end{Result}
  
  We will show that $\set'$ in Result \ref{granwig} will work as the density one subset mentioned in the statement of Theorem \ref{thmarw}. Let $\{n_j\} \subseteq \set'$ be a sequence such that $\nu_{n_j} \Rightarrow \nu$, where $\nu$ is a probability measure on $\mathbb{S}^1$ with no atoms. For $j \in \nat$, let $X_j := (-\pi \sqrt{n_j}, \pi \sqrt{n_j})^2$ and $g_j = g$. Define $F_j$ on $(-\pi \sqrt{n_j}, \pi \sqrt{n_j})^2$ by
  \begin{align*}
  F_j (x,y) := \arnj \left(\frac{x}{2\pi \sqrt{n_j}}, \frac{y}{2\pi \sqrt{n_j}}\right).
  \end{align*}	
  Since $\arnj$ satisfies $\Delta \arnj + 4\pi^2 n_j~ \arnj =0$, it follows that $F_j$ satisfies $\Delta F_j + F_j =0$. Also note that $F_j$ is a centered stationary Gaussian process with covariance 
  \begin{align*}
  K_{j}(z_1,z_2) = \frac{1}{|\Lambda_{n_j}|} \sum_{\lambda \in \Lambda_{n_j}^{+}} \cos \left( \frac{\lambda}{\sqrt{n_j}} \cdot (z_1 - z_2)\right).
  \end{align*}

  We now check that assumptions \ref{as1}--\ref{as6} hold for the above choice of $X_j$ and $F_j$. That \ref{as1}--\ref{as3}  hold is easy to check. \ref{as6} follows from \eqref{regpr}. Since $\nu_{n_j} \Rightarrow \nu$, it follows from Lemma \ref{lemnondeg} and \eqref{calc2} that for all large enough $j \in \nat$,  $\ell = 1, 2$ and every $m \leq 106$, the Gaussian vector $(\partial_{\ell} F_j(0),\partial_{{\ell}}^{2}F_{j}(0),\ldots,\partial_{{\ell}}^{m}F_{j}(0))$ is non-degenerate and its density is bounded above by $2\kappa_{m}$, where $\kappa_{m}$ is an upper bound for the density of $(\partial_{\ell} F_{\nu}(0),\partial_{{\ell}}^{2}F_{\nu}(0),\ldots,\partial_{{\ell}}^{m}F_{\nu}(0))$.  This along with the stationarity of $F_j$ implies \ref{as4}. Similarly $(A4')$ also follows from Lemma \ref{lemnondeg} and stationariy of $F_j$. We now show that\ref{as5} holds with $\psi_{j} \equiv C /|\Lambda_{n_j}|$, for some $C>0$. We first note that $F_j$ can be expressed as follows
  \begin{align*}
  F_j (z) = \sqrt{\frac{2}{|\Lambda_{n_j}|}} \sum_{\lambda \in \Lambda_{n_j}^{+}} \xi_{\lambda} \cos\left(\frac{\lambda}{\sqrt{n_j}}\cdot z\right) + \eta_{\lambda} \sin \left(\frac{\lambda}{\sqrt{n_j}}\cdot z\right),
  \end{align*}
  where $\{\xi_{\lambda},\eta_{\lambda} : \lambda  \in \Lambda_{n_j}^{+}\}$ are i.i.d. $\mathcal{N}(0,1)$ random variables. Let $H_j$ be the Hilbert space defined as follows
  \begin{align*}
  H_j = \textup{Span}\left\{\sqrt{2}\cos\left(\frac{\lambda}{\sqrt{n_j}}\cdot z\right),~ \sqrt{2}\sin \left(\frac{\lambda}{\sqrt{n_j}}\cdot z\right): \lambda \in \Lambda_{n_j}^{+}\right\},
  \end{align*}
   and the inner product in $H_j$ is defined so that the above spanning set is orthonormal. $F_j$ defines a Gaussian measure $\gamma_j$ on $H_j$. The Cameron-Martin space of $(H_j,\gamma_j)$ is $\mathcal{H}_j = (H_j, \langle \cdot, \cdot \rangle_{\mathcal{H}_j})$ with inner product defined by 
  \begin{align}\label{tornorms}
  \langle f,h \rangle_{\mathcal{H}_j} = |\Lambda_{n_j}|~ \langle f,h \rangle_{H_j} = |\Lambda_{n_j}|~ \langle f(2\pi\sqrt{n_j}~ \cdot),h(2\pi\sqrt{n_j}~ \cdot) \rangle_{L^{2}(\torus)},
  \end{align}
  for $f,h \in H_j$. The following claim which establishes \ref{as5} with $\psi_j \equiv C/|\Lambda_{n_j}|$, follows immediately from Result \ref{granwig} and \eqref{tornorms}. 
  \begin{Claim} For $h \in H_j$, the following relation between the $L^2$ norm and Cameron-Martin norm holds whenever $R_j \geq 2 \pi (\log n_j)^{1+\frac{\log 2}{3}+ \kappa}$
  	\begin{align*}
  	\int_{B(R_j)} h^2(x,y) dx dy \leq \frac{C}{|\Lambda_{n_j}|} \cdot R_{j}^{2}~ \|h\|_{\mathcal{H}_j}^{2}.
  	\end{align*}
  \end{Claim}

\section*{Acknowledgements}
This work was carried out during my Ph.D. under the supervision of Manjunath Krishnapur. I thank him for suggesting the questions considered in this paper, his invaluable guidance, encouragement and the numerous discussions we had on the subject of this paper.  
I thank Igor Wigman for sharing his insights on the subject. I thank Mikhail Sodin for a stimulating discussion I had with him and for his encourangement. I also thank Agnid Banerjee, Sugata Mondal and Matthew De Courcy-Ireland for  helpful discussions.

\appendix

\section{Gaussian measure on an infinite dimensional space} \label{app1}
\noin The definitions and  discussion in this section are mainly based on \cite{Bogachev}.
\subsection{Gaussian measure on a locally convex space}
Let $Y$ be a topological space and $\mathcal{B}(Y)$ its Borel sigma-algebra. A finite measure $\mu$ on $(Y,\mathcal{B}(Y))$ is said to be \textbf{Radon} if for every Borel set $B \in \mathcal{B}(Y)$ and every $\ep>0$, there is a compact set $K_{\ep} \subseteq B$ such that  $\mu(B \setminus K_{\ep}) \leq \ep$.
A \textbf{locally convex space} $X$ is a real vector space along with a collection of semi-norms $\{p_{\alpha}\}_{\alpha \in \mathcal{A}}$ which separate points in $X$. There is a natural way to endow $X$ with a topology  using these semi-norms; the following collection of subsets serve as a basis for this topology on $X$
\begin{align*}
\{ \{x \in X: p_{\alpha}(x-a)<\epsilon\}: \alpha \in \mathcal{A}, a \in X, \epsilon >0\}.
\end{align*}
We now define $X^{*}$ to be the topological dual of $X$
\begin{align*}
X^{*}:= \{\ell: X \ra \re ~|~ \mbox{$\ell$ is linear and continuous}\}.
\end{align*}
Let $\mathcal{B}(X)$ denote the Borel sigma-algebra on $X$ and let $\mathcal{E}(X)$ denote the minimal sigma-algebra on $X$ with respect to which every $\ell \in X^{*}$  is measurable. A measure $\gamma$ on $(X,\mathcal{E}(X))$ is called a \textbf{Gaussian measure} if for every $\ell \in X^{*}$, the push forward $\ga\circ \ell^{-1}$ is a Gaussian measure on $\re$. The \textbf{mean} of $\ell \in X^{*}$, denoted $m_{\ga}(\ell)$, is defined to be 
\begin{equation*}
m_{\ga}(\ell) := \int_X \ell(x) \ga(dx) = \e[\ga\circ \ell^{-1}].
\end{equation*}
The measure $\gamma$ is called \textbf{centered} if for every $\ell \in X^{*}$, $m_{\gamma}(\ell)=0$. If $\ga$ is a Gaussian measure on  $(X, \mathcal{E}(X))$, then $X^* \subseteq L^2 (\ga)$. We shall henceforth assume that $\gamma$ is centered. $X_{\ga}^{*}$ is defined to be the closure of $X^*$ in $L^2(\ga)$. $X_{\ga}^{*}$ equipped with the inner product from $L^2(\ga)$ is  a Hilbert space which is  called the \textbf{reproducing kernel Hilbert space} of $\ga$.  For $h \in X$, consider the evaluation map $\mbox{ev}_{h}: X^* \ra \re$ defined by
\begin{align*}
\mbox{ev}_{h}(\ell) := \ell(h).
\end{align*} 
If this map is continuous on $X^*$, then it extends in a unique way to all of $X_{\ga}^{*}$. The collection of elements in $X$ for which the corresponding evaluation map is continuous forms a subspace of $X$, denoted by ${\h}_{\ga}$. More formally, for $h \in X$ we define
\begin{align*}
|h|_{{\h}_{\ga}} &:= \mbox{sup}\{\ell(h): \ell \in X^* \text{ and } \langle \ell, \ell \rangle_{L^2(\ga)} \leq 1\}, \\
{\h}_{\ga} & :=\{h \in X : |h|_{{\h}_{\ga}} < \infty\}.
\end{align*}
$X_{\ga}^{*}$ being a Hilbert space, corresponding to every $h \in {\h}_{\ga}$, there is an element $h^* \in X_{\ga}^{*}$ such that $\mbox{ev}_{h}(\cdot) = \langle \cdot, h^* \rangle_{L^2(\ga)}$. For $h_1, h_2 \in {\h}_{\ga}$ define
\begin{align}\label{ipdef}
\langle h_1, h_2 \rangle_{{\h}_{\ga}} := \langle h_{1}^* , h_{2}^*\rangle_{L^2(\ga)}.
\end{align}
${\h}_{\ga}$ along with $\langle \cdot, \cdot \rangle_{{\h}_{\ga}}$ is a Hilbert space and is called the \textbf{Cameron-Martin space} of the measure $\ga$. A Borel measure $\gamma$ on $X$ is said to be a \textbf{Radon Gaussian measure} if it is Radon and its restriction to $\mathcal{E}(X)$ is Gaussian. 
\subsection{Viewing a Gaussian process as a Gaussian measure} \label{app2} Let $G$ be a centered, Gaussian process on $\re^2$ whose sample paths are  continuous almost surely. For $R>0$, let $\Lambda_R := [-R,R]^2 \subseteq \re^2$ and let $G_R := G|_{\lamr}$. It follows from Lemma A.3 in \cite{NS2} that $G_R$ defines a Borel measure $\ga$ on the locally convex space $X :=  (C(\lamr),\Vert \cdot \Vert_{L^{\infty}(\lamr)})$. The topological dual of $X$ is given by the Riesz representation theorem
\begin{equation*}
\begin{gathered}
X^* = \{\ell_{\mu}: \mu~\mbox{is a finite, positive Borel measure on $\lamr$}\},\\
\text{where }\ell_{\mu}(f) := \int_{\Lambda_R} f(x) d\mu(x),\text{ for } f \in C(\lamr).
\end{gathered}
\end{equation*}
The measure $\ga$ is in fact a Gaussian measure since for every $\ell_{\mu} \in X^*$, the measure $\gamma \circ \ell_{\mu}^{-1}$ has the same distribution as the random variable $\int_{\lamr} G_R(x) d\mu(x)$ which is Gaussian. It also follows that since $G_R$ is a centered Gaussian process on $\lamr$, the Gaussian measure $\ga$ is also centered. $\ga$ being a Borel measure on a complete, separable metric space, it is a Radon measure and hence $\ga$ is a Radon Gaussian measure.

Let $K(z_1,z_2) = \e[G(z_1)G(z_2)]$ be the covariance kernel of $G$. For $\ell_{\mu_1}, \ell_{\mu_2} \in X^*$, their $L^2(\ga)$ inner product  is given by
\begin{align}
\langle \ell_{\mu_1}, \ell_{\mu_2} \rangle_{L^2(\ga)} & = \int_{X} \ell_{\mu_1} (f)~ \ell_{\mu_2}(f) \ga(df), \nonumber \\
& =  \intl \intl \left[ \int_{X} f(z) f(w)\ga(df) \right] d\mu_1(z) d\mu_2(w), \nonumber \\
& =  \intl \intl K(z,w)~ d\mu_1(z) d\mu_2(w). \label{eq19}
\end{align}
For $\ell_{\mu} \in X^*$, define $\tilde{f_{\mu}}: \Lambda_R \ra \re$ by
\begin{align} \label{deffmu}
\tfm (z) := \int_{\lamr} K(z,w)~ d\mu(w).
\end{align}
Continuity of $K$ implies continuity of $\tfm$, hence $\tfm\in X$. For every $\ell_{\nu} \in X^*$, we have
\begin{align} \label{ev}
\textup{ev}_{\tfm}(\ell_{\nu}) = \int_{\lamr} \tfm(z)~ d\nu(z) = \int_{\lamr} \int_{\lamr} K(z,w)~ d\mu(w) d\nu(z) = \langle \ell_{\mu},\ell_{\nu} \rangle_{L^{2}_{\ga}}~,
\end{align}
and hence $\tfm \in \h_{\ga}$. Consider the following maps
\begin{align*}
X^*  \overset{\psi_1}{\longrightarrow}  \mathcal{H}_{\gamma}  \overset{\psi_2}{\longrightarrow}  (X^{*}_{\ga})^*  \overset{\psi_3}{\longrightarrow}  X^{*}_{\ga}
\end{align*}
where $\psi_1(\ell_{\mu}) := \tilde{f_{\mu}}$, $\psi_2(f) := \text{ev}_{f}$ and $\psi_3$ is the natural isomorphism between a Hilbert space and its dual, namely $\psi_3(\langle \cdot, \ell \rangle_{L^2(\ga)}) = \ell$. It follows from \eqref{ipdef} and \eqref{ev} that  $\psi_{2}$ and $\psi_1$ are isometries onto their respective images. We also note that $\psi_3 \circ \psi_2 \circ \psi_1 = \text{Id}|_{X^*}$ and $X^*$ is dense in $X_{\ga}^{*}$. Hence we conclude that $\psi_2$ is an isometry and thus the reproducing kernel Hilbert space and the Cameron-Martin space of $\ga$ are isomorphic.
\subsection{Cameron-Martin space of a stationary Gaussian process}
Let  $G$ be a centered, stationary Gaussian process on $\re^2$ whose sample paths are almost surely continuous. Let $\rho$ be the  spectral measure of $G$. Denote by $K(z_1,z_2) = k(z_1-z_2) = \widehat{\rho}(z_1-z_2)$ the covariance kernel of $G$. We discussed in \ref{app2} that if $G_R := G|_{\lamr}$, then $G_R$ induces a Radon Gaussian measure $\ga$ on $(C(\Lambda_R),\|\cdot\|_{L^{\infty}(\Lambda_R)})$. We shall continue to use the notations introduced in \ref{app2}. We now characterize functions which belong to the Cameron-Martin space $\mathcal{H}_{\ga}$.  Define $L^{2}_{\text{symm}}(\rho)$  as follows
\begin{equation*}
L^{2}_{\text{symm}}(\rho) := \{f: \mathbb{R}^2 \ra \mathbb{C}~|~ f \in L^2(\rho), f(-z) = \overline{f(z)}\text{ for every $z \in \mathbb{R}^2$}\}.
\end{equation*}
We note that $L^{2}_{\text{symm}}(\rho)$ is a real vector space and it follows from the symmetry of the measure $\rho$ that the inner product it inherits from $L^{2}(\rho)$ is real, that is, for $f,g \in L^{2}_{\text{symm}}(\rho)$ we have $\langle f,g\rangle_{L^2(\rho)} \in \re$. For $f \in L^2(\rho)$ and $z \in \re^2$, we define $\widehat{f}$ as follows
\begin{equation*}
\begin{gathered}
\widehat{f}(z)   := \int_{\re^2} e^{-i\langle z, w\rangle} f(w) d\rho(w).
\end{gathered}
\end{equation*}
\begin{Claim}
	$\h_{\ga} \subseteq \mathcal{F}L^{2}_{\textup{symm}}(\rho)|_{\Lambda_R} := \{\widehat{h}|_{\Lambda_R}: h \in L^{2}_{\textup{symm}}(\rho)\}$ equipped with the inner product  coming from $L^2(\rho)$, that is for $f,h \in L^{2}_{\textup{symm}}(\rho)$
	\begin{align*}
	\langle \widehat{f}|_{\Lambda_R}, \widehat{h}|_{\Lambda_R} \rangle := \langle f,h \rangle_{L^2(\rho)}.
	\end{align*} 
\end{Claim}
\begin{proof}
	For $\ell_{\mu} \in X^*$, define $f_{\mu}: \re^2 \ra \re$ by
	\begin{align*}
	f_{\mu} (z) := \intl e^{i\langle y,z \rangle} d\mu(y),~\text{for }z \in \re^2.
	\end{align*}
	We show below that $f_{\mu} \in L^{2}_{\text{symm}}(\rho)$. 
	\begin{align}
	\int_{\re^2} f_{\mu}(z) \overline{f_{\mu}(z)} d\rho(z) & = \int_{\re^2} \intl \intl e^{i\langle y-y',z \rangle} d\mu(y) d\mu(y') d\rho(z), \nonumber \\
	& =\intl \intl K(y,y') d\mu(y) d\mu(y'), \nonumber \\
	& = \langle \ell_{\mu},\ell_{\mu} \rangle_{L^{2}({\ga})} < \infty. \label{eq22}
	\end{align}
	The following establishes the relation between $f_{\mu}$ and $\tfm$ which was defined in \eqref{deffmu},
	\begin{align}
	\widehat{f_{\mu}}(z) & = \int_{\re^2}\intl e^{-i\langle z,x\rangle} e^{i\langle y,x \rangle} d\mu(y) d\rho(x), \nonumber \\
	& = \intl K(z,y) d\mu(y) = \tfm (z). \label{eq21}
	\end{align}
	From the discussion in  \ref{app2} and \eqref{eq21}, we conclude that  $\accentset{\circ}{\h_{\ga}} := \{\widehat{f_{\mu}}|_{\Lambda_R}: \ell_{\mu} \in X^*\}  $ is a dense subset of $\h_{\ga}$. We conclude using \eqref{eq22} and \eqref{eq21} that $\accentset{\circ}{\h_{\ga}}  \subseteq \mathcal{F}L^{2}_{\text{symm}}(\rho)|_{\Lambda_R}$ and that the following map is an isometry onto its image
	\begin{equation}\label{isom}
	\begin{aligned}
	(X^{*}, L^{2}({\ga})) & \longrightarrow L^{2}_{\text{symm}}(\rho)\\
	\ell_{\mu} & \longrightarrow f_{\mu}.
	\end{aligned}
	\end{equation}
	Hence the inner product in $\mathcal{H}_{\ga}$ is given by
	\begin{align*}
	\langle \widehat{f}_{\mu}|_{\Lambda_R}, \widehat{f}_{\nu}|_{\Lambda_R}\rangle_{\h_{\ga}} = \langle \ell_{\mu}, \ell_{\nu} \rangle = \langle f_{\mu}, f_{\nu} \rangle_{L^2(\rho)}.  
	\end{align*}
	Suppose $f \in \h_{\ga} \setminus \accentset{\circ}{\h_{\ga}}$, then there is a corresponding $\ell \in X^{*}_{\ga} \setminus X^*$ such that for every $ \nu \in X^*$
	\begin{align}\label{eq23}
	\intl f(x) d\nu(x) = \langle \ell,\ell_{\nu} \rangle_{L^{2}({\ga})}.
	\end{align}
	There exists a sequence $\ell_{\mu_n} \ra \ell$ in ${L^{2}({\ga})}$ and hence for every $\nu \in X^{*}$, we have
	\begin{align}\label{iplimits}
	\lim\limits_{n \ra \infty}\langle \ell_{\mu_n},\ell_{\nu} \rangle_{L^{2}({\ga})} = \langle \ell,\ell_{\nu} \rangle_{L^{2}({\ga})}.
	\end{align}
	For $p \in \lamr$, choosing $\nu = \delta_{p}$ in \eqref{iplimits} we conclude from \eqref{eq19} and  \eqref{eq23} that 
	\begin{align*}
	\lim\limits_{n \ra \infty} \widehat{f}_{\mu_n}(p)= \lim\limits_{n \ra \infty}\langle \ell_{\mu_n}, \ell_{\delta_{p}} \rangle_{L^{2}({\ga})} = \langle \ell,\delta_{p} \rangle_{L^{2}({\ga})} = f(p).
	\end{align*} 
	Since $\{\ell_{\mu_n}\}$ is Cauchy in $L^{2}(\ga)$, it follows from the isometry \eqref{isom} that  $\{f_{\mu_n}\}$ is Cauchy in $L^{2}_{\text{symm}}(\rho)$. Let $f_{\mu_n} \ra f_{\infty}$ in $L^2 (\rho)$, then $\widehat{f}_{\mu_n} \ra \widehat{f}_{\infty}$ pointwise and hence $f = \widehat{f}_{\infty}$. This proves our claim. 
\end{proof}
\section{Calculations}
\subsection{Derivatives of a stationary Gaussian process}
The calculations below are for Lemma \ref{lemnondeg} and we use the same notations used in the lemma.  For $\theta \in [0,2\pi]$, let $p_{\theta} =(\cos \theta,\sin \theta)$ and  let  $z,w \in \re^2$. For $r,s \in \nat_0$ we have
\begin{align}\label{calc1}
\e[\partial_{1}^{r}G(z) \partial_{1}^{s}G(w)]& = (-1)^{\frac{r+3s}{2}}   \int_{0}^{2\pi} (\cos \theta)^{r+s} \phi_{r,s}(\langle z-w, p_{\theta}\rangle) d\nu(\theta),
\end{align}
where $\phi_{r,s}(\cdot) = \cos(\cdot)$ if $r+s$ is even and $\phi_{r,s}(\cdot) = \sin(\cdot)$ otherwise.  With $z=w=(0,0)$,  \eqref{calc1} becomes
\begin{align} \label{calc2}
\e[\partial_{1}^{r}G(0) \partial_{1}^{s}G(0)] &= (-1)^{\frac{r+3s}{2}}    \int_{0}^{2\pi} (\cos \theta)^{r+s} \phi_{r,s}(0) d\nu(\theta), \nonumber \\
& = \begin{cases} 
(-1)^{\frac{r+3s}{2}}     \int_{0}^{2\pi} (\cos \theta)^{r+s} d\nu(\theta),&\textup{if $r+s$ is even},\\
0,&\textup{if $r+s$ is odd}.
\end{cases}
\end{align} 
The following is the calculation for the variance terms which appear in Lemma \ref{lemnondeg}. Let $\mathscr{I}_n := \{i \in 2\nat : i \leq n\}$ and $\mathscr{S}_n := \{s \in 2\nat +1 : s \leq n\}$, then we have
\begin{align} 
\mbox{Var} \left(\sum_{i \in \mathscr{I}_n} a_i \partial_{1}^{i}G(0) \right)& =\e\left(\sum_{i \in \mathscr{I}_n} a_i \partial_{1}^{i}G(0) \right)^2, \nonumber\\
& = \sum_{i,j \in \mathscr{I}_n} a_i a_j \e[\partial_{1}^{i}G(0)  \partial_{1}^{j}G(0)], \nonumber \\
& = \sum_{i,j \in  \mathscr{I}_n} a_i a_j  (-1)^{(i+j)/2}  \int_{0}^{2\pi} (\cos \theta)^{i+j} d\nu(\theta), \nonumber \\
& =  \int_{0}^{2\pi} \left[\sum_{i \in \mathscr{I}_n}(-1)^{i/2}a_i (\cos \theta)^i\right]^2 d\nu(\theta).\label{calc3} \\
 \mbox{Var} \left(\sum_{s \in \mathscr{S}_n}a_s \partial_{1}^{s}G(0) \right) &= \e \left(\sum_{s \in \mathscr{S}_n} a_s \partial_{1}^{s}G(0) \right)^2,\nonumber\\
 & = \sum_{r,s \in \mathscr{S}_n} a_{r} a_{s} \e[\partial_{1}^{r}G(0)  \partial_{1}^{s}G(0)], \nonumber \\
 & = \sum_{r,s \in \mathscr{S}_n } - a_{r} a_{s} (-1)^{(r+s)/2} \int_{0}^{2\pi} (\cos \theta)^{r+s} d\nu(\theta), \nonumber \\
 & = \int_{0}^{2\pi} \left[\sum_{s \in \mathscr{S}_n} (-1)^{(s-1)/2} a_{s} (\cos \theta)^s \right]^2 d\nu(\theta). \label{calc4}
\end{align}

\section{Figure}

\begin{figure}[ht]
	\def\svgwidth{.7\linewidth}
	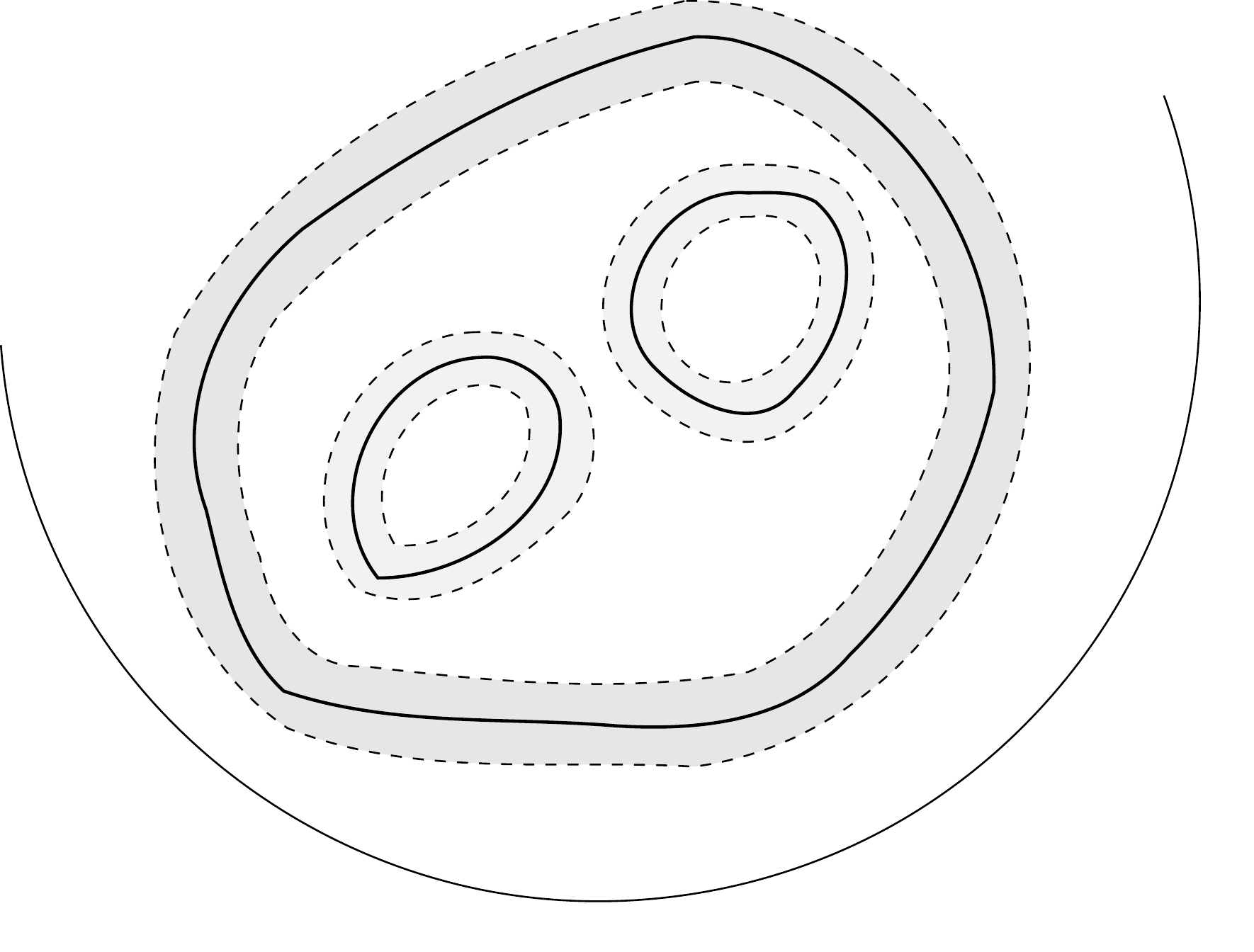
	\caption{An illustration of a nodal component $\gamma$ of $f$ with diameter less than $r$ intersecting a stable disc $D$. Perturbing $f$ by $h$ which has a small $C^1$ norm in $3D$ leads to perturbation of  $\gamma$, $\gamma_1$ and $\gamma_2$ in such a way that they are contained in their respective \textit{shells} and no other new components form in the interior of $\gamma$.}
		\label{fig:fig1}
\end{figure}
\section{Calculations on the sphere $\sph$} We collect a few facts about the exponential map on $\sph$.
\begin{enumerate}[label={\arabic*.}, align=left,leftmargin=*,widest={8}]
	\item The exponential map at $p \in \sph$ is explicitly given by the following formulae, for $v \in T_{p}\sph$ with $\|v\| \leq \pi/2$, 
	\begin{align*}
	\ex_p (v) = \cos(\|v\|)p + \sin(\|v\|) \frac{v}{\|v\|}.  
	\end{align*}
	\item Suppose that $p,q \in \sph$ and let $\mathscr{R}: \sph \ra \sph$ denote the rotation map such that $\mathscr{R}(p) = q$. We denote by $\mathscr{R}_{*}$ the map $d\mathscr{R}|_{T_{p}(\sph)}: T_{p}(\sph) \ra T_{q}(\sph)$ which is an isometry. 
	\begin{Claim} \label{claimcalc} With notations as above, we have  $\mathscr{R} \circ \ex_p = \ex_q \circ \mathscr{R}_{*}$.
	\end{Claim}
	\begin{proof} Without loss of generality, we may assume that $p=(1,0,0)$ and $q = (\cos \delta, \sin \delta, 0)$, for some $\delta \in [0,2\pi)$. $\mathscr{R}_{*}$ is given by
		\begin{align*}
		\mathscr{R}_{*}(0,a,b) = (-a \sin \delta, a \cos \delta, b).
		\end{align*}
		Let $(0,a,b) = (0, r\cos \theta, r\sin \theta)$ and hence we have
		\begin{align*}
		\ex_q \circ \mathscr{R}_{*}(0,a,b) = &	\ex_q (-a \sin \delta, a \cos \delta, b),\\
		= & \cos r(\cos \delta, \sin \delta, 0) + \sin r (-\cos \theta \sin \delta, \cos \theta \cos \delta, \sin \theta),\\
		=& (\cos r \cos\delta  - \sin r \cos \theta \sin \delta, \cos r \sin \delta + \sin r \cos \theta \cos \delta,\\
		&~~~ \sin r \sin \theta). 
		\end{align*}
		Then the action of rotation $\mathscr{R}$  is just multiplication by the following matrix,
		\begin{align*}
		\mathscr{R} = \begin{bmatrix}
		\cos \delta& -\sin \delta & 0 \\ 
		\sin \delta & \cos \delta & 0 \\ 
		0& 0 & 1
		\end{bmatrix}.
		\end{align*} 
		Hence we have
		\begin{align*}
		\mathscr{R}(\ex_p (0,a,b))& = \mathscr{R}(\cos r, \sin r \cos \theta, \sin r \sin \theta),\\
		& = \begin{bmatrix}
		\cos \delta& -\sin \delta & 0 \\ 
		\sin \delta & \cos \delta & 0 \\ 
		0& 0 & 1
		\end{bmatrix} \begin{bmatrix} \cos r\\ \sin r \cos \theta \\ \sin r \sin \theta \end{bmatrix}\\
		& = (\ex_q \circ \mathscr{R}_{*}) (0,a,b). 
		\end{align*}
	\end{proof}
	\item 	The following is a justification for the fact that $F_n \overset{d}{=} G_n$ on $\nsq$.  
	\begin{Claim} \label{claimcalc2}
		$F_n \overset{d}{=} G_n$ on $\nsq$.
	\end{Claim}
	\begin{proof}
		Recall that $F_n$ and $G_n$ are defined as follows
		\begin{align*}
		F_n(x,y) &= (\scf_n \circ \ex_p \circ I_p \circ \bup^{-1})(x,y),\\
		G_n (x,y) &= (\scf_n \circ \ex_q \circ {\mathscr{R}_{q_{*}}}  \circ I_p  \circ \bup^{-1})(x,y),\\
		& = (\scf_n \circ {\mathscr{R}_{q}}  \circ \ex_p  \circ I_p  \circ \bup^{-1})(x,y).
		\end{align*}
		For $u,v \in \nsq$, we have
		\begin{align*}
		\e[F_n(u) F_n(v)] & = P_n(\cos(\Theta(\ex_p \circ I_p \circ \bup^{-1}(u), \ex_p \circ I_p \circ \bup^{-1}(v)))), \\ 
		& = P_n(\cos(\Theta(\mathscr{R}_{q} \circ \ex_p \circ I_p \circ \bup^{-1}(u), \mathscr{R}_{q} \circ \ex_p \circ I_p \circ \bup^{-1}(v)))),\\
		& = \e[G_n(u) G_n(v)].
		\end{align*}
	\end{proof}
\end{enumerate}

\subsection{Generalized Fa\`{a} di Bruno formula} The following formulae for the higher order partial derivatives of a composite function is a special case of Theorem 2.1  in \cite{CS}. 
\begin{Result} \label{faa} Let $U,V \subseteq \re^2$ and $g: U \ra V$ and $f: V \ra \re$ be smooth functions. Let $h = f\circ g$ and let $g = (g_1,g_2)$. Let $u_0 \in U$ and $v_0 \in V$ be such that $g(u_0) = v_0$. Then for every $n \in \nat$, the following holds
	\begin{align} \label{faaeqn}
	\partial_{1}^{n} h (u_0) = \sum_{\substack{|\lambda| = 1 \\ \lambda = (\lambda_1, \lambda_2)}}^{n} n! ~(\partial_{1}^{\lambda_1} \partial_{2}^{\lambda_2}f)(v_0) \sum_{s=1}^{n} \sum_{P_s(n,\lambda)} \prod_{j=1}^{s} \frac{(\partial_{1}^{t_j} g_1 (u_0))^{a_j} ~ (\partial_{1}^{t_j} g_2 (u_0))^{b_j}}{a_{j}! ~b_{j}! ~(t_{j}!)^{a_{j}+b_{j}}},
	\end{align}
	where $P_s(n,\lambda) = \{((a_1,b_1),(a_2,b_2),\ldots,(a_j,b_j),t_1,t_2,\ldots,t_j): a_i,b_i,t_i \in \nat_0 $ satisfying the following conditions
	\begin{enumerate}[align=left,leftmargin=*,widest={10}]
		\item For every$1 \leq i \leq s$, we have $a_i + b_i \geq 1$ and $0 < t_1 < t_2 < \cdots < t_s$,
		\item $\sum_{i=1}^{s} a_i = \lambda_1$, $\sum_{i=1}^{s}b_i = \lambda_2$ and $\sum_{i=1}^{s} (a_i + b_i)t_i = n\}$. 
	\end{enumerate}
	
	\begin{Rem}In the setting of Result \ref{faa}, if $f$ is a smooth Gaussian process on $V$ and $g$ is a smooth deterministic function, then $h$ defines a smooth Gaussian process on $U$. It follows from \eqref{faaeqn} that the Gaussian random  variable $\partial_{1}^{n} h (u_0)$ can be written as a linear combination of the Gaussian variables $\{(\partial_{1}^{\lambda_1} \partial_{2}^{\lambda_2}f)(v_0): 1 \leq |\lambda| \leq n\}$. We make the following easy observations about  expression \eqref{faaeqn}.
		\begin{enumerate}[align=left,leftmargin=*,widest={10}]
			\item The coefficient of $\partial_{1}^{n}f(v_0)$ is $(\partial_1 g_1 (u_0))^n$.
			\item For  $\lambda \neq (n,0)$, the coefficient of $\partial_{1}^{\lambda_1} \partial_{2}^{\lambda_2} f (v_0)$ is a weighted sum of terms each of which contains a factor of either $\partial_{1}^{k} g_1 (u_0)$, for some $1 < k \leq n$ or $\partial_{1}^{r} g_2 (u_0)$, for some $1 \leq r \leq n$.
		\end{enumerate}
	\end{Rem}
\end{Result}

\section{Estimates for Bessel functions}\label{appbessel}
We collect a few facts about  Bessel functions, these are taken from \cite{besselbook}.
\begin{enumerate}[label={\arabic*.}, align=left,leftmargin=*,widest={8}]
	\item \textit{Poisson representation of Bessel functions}. For $n\in \nat$ and $r \in \re$,
	\begin{align}\label{bessel}
	J_{n}(r) = \frac{(r/2)^n}{\Gamma(n+\frac{1}{2}) \sqrt{\pi}} \int_{-1}^{1} e^{irt} (1-t^2)^{n-\frac{1}{2}}~ dt. 
	\end{align}
	\item \textit{Recurrence relation.} For $n \in \nat$, we have
	\begin{align}\label{besdif}
	2 J'_n(x) = J_{n-1}(x) - J_{n+1}(x). 
	\end{align}
	\item The following  asymptotic expression for $J_n$  is taken from \cite{ikrasikov}. For $n \in \nat_0$ and $x \in \re$, we have
	\begin{align}\label{besasym}
	J_n(x) = \sqrt{\frac{2}{\pi x}} \left[\cos \left( x - \frac{n \pi}{2} - \frac{\pi}{4}\right) + O\left(\frac{n^2}{|x|}\right)\right]. 
	\end{align}
	\item \textit{Legendre duplication formula}. For every $n \in \nat$,
	\begin{align}
	\Gamma(n) \Gamma\left(n+\frac{1}{2}\right) &= 2^{1-2n} \sqrt{\pi}~ \Gamma(2n) \nonumber,\\
	\Gamma\left(n+\frac{1}{2}\right) = \frac{2^{1-2n} \sqrt{\pi}~ \Gamma(2n)}{\Gamma(n)} &= \frac{2^{-2n} \sqrt{\pi}~ (2n)!}{n!}  \sim \sqrt{2\pi} \left(\frac{n}{e}\right)^n. \label{Gamma}
	\end{align}
	\item We conclude from \eqref{bessel} and \eqref{Gamma} that for large enough $n \in \nat$ and every $r>0$,
	\begin{align}\label{besselesti}
	|J_{n}(r)| \leq 10 \cdot \frac{(r/2)^n}{\sqrt{2\pi} \left(n/e\right)^n \sqrt{\pi}} \leq \left( \frac{2r}{n}\right)^n.
	\end{align}
	\item It now follows from \eqref{besdif} and \eqref{besasym} that
	\begin{align} \label{stab}
	J_{n}'(x) = - \sqrt{\frac{2}{\pi x}} \left[ \sin \left( x - \frac{n \pi}{2} - \frac{\pi}{4} \right) + O\left(\frac{n^2}{|x|}\right) \right].
	\end{align}

\end{enumerate}
\section{Calculations in the proof of Lemma \ref{lemmamain}}

\subsection{A simple estimate} For $0 \leq a \leq b \leq 1$, we have
\begin{equation}\label{caldis}
\begin{aligned}
\lb \frac{b}{2} -a \rb \leq \sqrt{1-a} - \sqrt{1-b}  \leq \lb b - \frac{a}{2} \rb.
\end{aligned}
\end{equation}

\subsection{$q_2$ increases on $[s,1]$}\label{appq2inc} $q_2$ and $s$ were defined in \eqref{qone} and \eqref{qonepos}.
\begin{align*}
q_2(x)  &= \frac{n(n+1)( x^2 - x_{0}^{2}) - 1}{(1-x^2)^2},\\
q'_2(x)& = \frac{2n(n+1)x(1-x^2)^2 + 4x(1-x^2)((n(n+1)( x^2 - x_{0}^{2})-1)}{(1-x^2)^4}.
\end{align*}
The first term in the numerator of the above expression for $q'_2$ is non-negative and recalling that $s$ was defined as follows
\begin{align*}
s^{2} = \xnot^2 + \frac{1}{n(n+1)},
\end{align*}
we can conclude that the second term is positive when $x >s$.

\bibliographystyle{plain}
\bibliography{reference}
 \end{document}